\documentclass[reqno]{amsart}
\usepackage{amssymb}
\usepackage{graphicx}

\usepackage[usenames, dvipsnames]{color}
\usepackage{verbatim}
\usepackage{mathrsfs}
\usepackage{bm}
\usepackage{cite}

\numberwithin{equation}{section}

\newtheorem{theorem}{Theorem}[section]
\newtheorem{corollary}[theorem]{Corollary}
\newtheorem{lemma}[theorem]{Lemma}
\newtheorem{prop}[theorem]{Proposition}
\newtheorem{example}[theorem]{Example}

\theoremstyle{definition}
\newtheorem{remark}[theorem]{Remark}

\theoremstyle{definition}

\theoremstyle{definition}

\makeatletter
\def\dashint{\operatorname%
{\,\,\text{\bf-}\kern-.98em\DOTSI\intop\ilimits@\!\!}}
\makeatother

\def\\det{\text{\det}}

\def\.5{\frac{1}{2}}

\newcommand{\RN}[1]{%
  \textup{\uppercase\expandafter{\romannumeral#1}}%
}

\renewcommand{\epsilon}{\varepsilon}

\newcounter{marnote}

\begin{document}

\title[Asymptotic analysis of the stress concentration]{Asymptotic analysis of the stress concentration between two adjacent stiff inclusions in all dimensions}

\author[Z.W. Zhao]{Zhiwen Zhao}

\address[Z.W. Zhao]{Beijing Computational Science Research Center, Beijing 100193, China.}
\email{zwzhao365@163.com}

\author[X. Hao]{Xia Hao}
\address[X. Hao]{School of Mathematical Sciences, Beijing Normal University, Beijing 100875, China. }
\email{xiahao0915@163.com}


\date{\today} 


\maketitle
\begin{abstract}
In the region between two closely located stiff inclusions, the stress, which is the gradient of a solution to the Lam\'{e} system with partially infinite coefficients, may become arbitrarily large as the distance between interfacial boundaries of inclusions tends to zero. The primary aim of this paper is to give a sharp description in terms of the asymptotic behavior of the stress concentration, as the distance between interfacial boundaries of inclusions goes to zero. For that purpose we capture all the blow-up factor matrices, whose elements comprise of some certain integrals of the solutions to the case when two inclusions are touching. Then we are able to establish the asymptotic formulas of the stress concentration in the presence of two close-to-touching $m$-convex inclusions in all dimensions. Furthermore, an example of curvilinear squares with rounded-off angles is also presented for future application in numerical computations and simulations.
\end{abstract}

\section{Introduction}
We consider a mathematical model of a high-contrast fiber-reinforced composite material which consists of a matrix and densely packed inclusions. The problem is described by elliptic systems, especially the Lam\'{e} systems. The most important quantities from an engineering point of view is the stress, which is the gradient of a solution to the Lam\'{e} systems. If two inclusions are nearly touching, then the stress field may concentrate highly in the thin gap between them. Over the past twenty years, much effort has been dedicated to quantitative understanding of such high concentration. In this paper, we continue the investigation on this topic and aim to establish the asymptotic formulas of the stress concentration in the presence of $m$-convex inclusions in all dimensions. Moreover, we also give more precise characterization for this high concentration in the presence of two closely located curvilinear squares with rounded-off angles in two dimensions. This type of axisymmetric inclusions is related to the ``Vigdergauz microstructure'' in the shape optimization of fibers, which was extensively studied in previous work \cite{GK1995,GB2005,G1996,V1986}.

Our initial interest on the study of high stress concentration arising from composite material was motivated by the great work of Babu\u{s}ka et al. \cite{BASL1999}, where numerical computation of damage and fracture in linear composite systems was investigated. They observed numerically that the gradient of a solution to certain homogeneous isotropic linear systems of elasticity remains bounded regardless of the distance between inclusions. Subsequently, Li and Nirenberg proved this in \cite{LN2003} for general second-order elliptic systems including systems of elasticity with piecewise smooth coefficients, where stronger $C^{1,\alpha}$ estimates were established. We also refer to \cite{BV2000,LV2000,DL2019} for scalar divergence form second-order elliptic equations. It is worthwhile to emphasize that the equations under consideration in \cite{DL2019} are non-homogeneous. Kim and Lim \cite{KL2019} recently used the single and double layer potentials with image line charges to establish an asymptotic expansion of the potential function for core-shell geometry with circular boundaries in two dimensions. Calo, Efendiev and Galvis \cite{CEG2014} studied both high- and low-conductivity inclusions and established an asymptotic formula of a solution to elliptic equations with respect to the contrast $k$ of sufficiently small or large quantity.

For the purpose of making clear the stress concentration of composite materials, we always first consider its simplified conductivity model of electrostatics and analyze the singular behavior of the electric field, which is the gradient of a solution to the Laplace equation. Denote by $\varepsilon$ the distance between interfacial boundaries of inclusions. In the case of $2$-convex inclusions, it has been proved that the generic blow-up rates of the electric field are $\varepsilon^{-1/2}$ in two dimensions \cite{AKLLL2007,BC1984,BLY2009,AKL2005,Y2007,Y2009,K1993}, $|\varepsilon\ln\varepsilon|^{-1}$ in three dimensions \cite{BLY2009,LY2009,BLY2010,L2012}, and $\varepsilon^{-1}$ in dimensions greater than three \cite{BLY2009}, respectively. We would like to point out that Bao, Li and Yin \cite{BLY2009} also studied the generalized $m$-convex inclusions, especially when $m>2$. Besides these gradient estimates, there has been a lot of literature, starting from \cite{KLY2013}, on the asymptotic behavior of the concentrated field. In two dimensions, Kang et al. \cite{KLY2013} gave a complete characterization in terms of the singularities of the electric field for two neighbouring disks. Subsequently, Ammari et al. \cite{ACKLY2013} utilzied the method of disks osculating to convex domains to extend their results to the case of abritrary $2$-convex inclusions. In three dimensions, Kang et al. \cite{KLY2014} established a precise characterization for the electric field in the presence of two spherical conductors. Kang, Lee and Yun \cite{KLY2015} captured a stress concentration factor for the generalized $m$-convex inclusions as the distance between two inclusions tends to zero. Li, Li and Yang \cite{LLY2019} presented a precise calculation of the energy and obtained a sharp characterization for the concentrated field with two adjacent $2$-convex inclusions in dimensions two and three. Li \cite{Li2020} then extended the asymptotic results to the case of $m$-convex inclusions and captured a unified blow-up factor different from that in \cite{LLY2019}. The subsequent work \cite{ZH202101} completely solved the optimality of the blow-up rate for two adjacent $m$-convex inclusions in all dimensions. Bonnetier and Triki \cite{BT2013} obtained the asymptotics of the eigenvalues for the Poincar\'{e} variational problem in the presence of two closely located inclusions as the distance between these two inclusions tends to zero. For the nonlinear equation, see \cite{CS2019,CS20192,G2015,GN2012}.

In recent years, there has also made significant progress on the study of singular behavior of the stress concentration in the context of the full elasticity. Due to the fact that some techniques such as the maximum principle used in the scalar equations cannot apply to linear systems of elasticity, a delicate iterate technique with respect to the energy was built in \cite{LLBY2014} to overcome these difficulties. Li, Li, Bao and Yin \cite{LLBY2014} showed that the gradients for solutions to a class of elliptic systems with the same boundary data on the upper and bottom boundaries of the narrow region possess the exponentially decaying property. Bao et al. \cite{BLL2015,BLL2017} then utilized the iterate technique to establish the pointwise upper bound estimates on the gradient and captured the stress blow-up rate for two close-to-touching $2$-convex inclusions. Subsequently, Li \cite{L2018} demonstrated the optimality of the blow-up rate in dimensions two and three by finding a unified stress concentration factor to construct a lower bound of the gradient. Li and Xu \cite{LX2020} further improved the results in \cite{L2018} to be precise asymptotic formulas. However, the stress concentration factors, which are vital to the establishments of the lower bounds and asymptotic expansions on the gradients, are captured in \cite{L2018,LX2020} only when the domain and the boundary data satisfy some special symmetric conditions. Subsequently, Miao and Zhao \cite{MZ2021} got ride of those strict symmetric conditions by accurately constructing the unified stress concentration factors for the generalized $m$-convex inclusions in all dimensions and then established the optimal upper and lower bounds on the gradient. Their results gave a perfect answer for the optimality of the stress blow-up rate in any dimension. In addition, we refer to \cite{BJL2017,LZ2019,MZ202102} for the corresponding boundary estimates.

Kang and Yu \cite{KY2019} recently constructed singular functions for the two-dimensional Lam\'{e} systems and then gave a complete characterization in terms of the singular behavior of the gradient. These singular functions were used to study the effective property of the composite in their subsequent work \cite{KY2020}. Specifically, they utilized the primal-dual variational principle and the singular functions to rigorously prove the Flaherty-Keller formula for the effective elastic properties, where the Flaherty-Keller formula was previously investigated in \cite{FK1973}. It is worth mentioning that Ando, Kang and Miyanishi \cite{AKM2020} recently studied the Neumann-Poincar\'{e} type operator associated with the Lam\'{e} system and proved that its eigenvalues converge at a polynomial rate if the boundary of the domain is smooth and at an exponential rate on real analytic boundaries, respectively.

In this paper, we make use of all the systems of equations in linear decomposition to capture all the blow-up factor matrices in any dimension and then establish the precise asymptotic expressions of the concentrated field for two adjacent $m$-convex inclusions in all dimensions. This is different from that in \cite{L2018,LX2020,MZ2021}, which pursued to use a family of unified blow-up factors to characterize the singular behavior of the stress concentration.

Finally, we present an overview of the rest of this paper. In Section \ref{SEC002}, we give a mathematical formulation of the problem and list the main results, including Theorems \ref{Lthm066}, \ref{ZHthm002} and \ref{ZHthm003}. In Section \ref{SEC003}, we review some preliminary facts on linear decomposition and construction of the auxiliary function. Section \ref{SEC004} is dedicated to the proofs of Theorems \ref{Lthm066},  \ref{ZHthm002} and \ref{ZHthm003}. Example \ref{coro00389} is presented in Section \ref{SEC006}.

\section{Problem setting and main results}\label{SEC002}
\subsection{Mathematical formulation of the problem}
To formulate our problem and state the main results in a precise manner, we first describe our domain and introduce some notations. Let $D\subseteq\mathbb{R}^{d}\,(d\geq2)$ be a bounded domain with $C^{2,\alpha}\,(0<\alpha<1)$ boundary. Assume that there is a pair of $C^{2,\alpha}$-subdomains $D_{1}^{\ast}$ and $D_{2}$ inside $D$, such that these two subdomains touch only at one point and they are far away from the external boundary $\partial D$. That is, by a translation and rotation of the coordinates, if necessary,
\begin{align*}
\partial D_{1}^{\ast}\cap\partial D_{2}=\{0\}\subset\mathbb{R}^{d},
\end{align*}
and
\begin{align*}
D_{1}^{\ast}\subset\{(x',x_{d})\in\mathbb{R}^{d}|\,x_{d}>0\},\quad D_{2}\subset\{(x',x_{d})\in\mathbb{R}^{d}|\,x_{d}<0\}.
\end{align*}
Throughout the paper, we use superscript prime to denote ($d-1$)-dimensional domains and variables, such as $B'$ and $x'$. Via a translation, we set
\begin{align*}
D_{1}^{\varepsilon}:=D_{1}^{\ast}+(0',\varepsilon),
\end{align*}
where $\varepsilon>0$ is a sufficiently small constant. For ease of notation, we drop superscripts and denote
\begin{align*}
D_{1}:=D_{1}^{\varepsilon},\quad\mathrm{and}\quad\Omega:=D\setminus\overline{D_{1}\cup D_{2}}.
\end{align*}

Let $\Omega$ and $D_{1}\cup D_{2}$ be occupied, respectively, by two different isotropic and homogeneous elastic materials with different Lam$\mathrm{\acute{e}}$ constants $(\lambda,\mu)$ and $(\lambda_{1},\mu_{1})$. The elasticity tensors $\mathbb{C}^0$ and $\mathbb{C}^1$ corresponding to $\Omega$ and $D_{1}\cup D_{2}$ are expressed, respectively, as
$$C_{ijkl}^0=\lambda\delta_{ij}\delta_{kl} +\mu(\delta_{ik}\delta_{jl}+\delta_{il}\delta_{jk}),$$
and
$$C_{ijkl}^1=\lambda_1\delta_{ij}\delta_{kl} +\mu_1(\delta_{ik}\delta_{jl}+\delta_{il}\delta_{jk}),$$
where $i, j, k, l=1,2,...,d$ and $\delta_{ij}$ represents the kronecker symbol: $\delta_{ij}=0$ for $i\neq j$, $\delta_{ij}=1$ for $i=j$. Denote the elastic displacement field by $u=(u^{1},u^{2},...,u^{d})^{T}:D\rightarrow\mathbb{R}^{d}$. Let $\chi_{\Omega}$ be the characteristic function of $\Omega\subset\mathbb{R}^{d}$. Given a boundary data $\varphi=(\varphi^{1},\varphi^{2},...,\varphi^{d})^{T}$, we consider the Dirichlet problem for the Lam$\mathrm{\acute{e}}$ system with piecewise constant coefficients
\begin{align}\label{La.001}
\begin{cases}
\nabla\cdot \left((\chi_{\Omega}\mathbb{C}^0+\chi_{D_{1}\cup D_{2}}\mathbb{C}^1)e(u)\right)=0,&\hbox{in}~~D,\\
u=\varphi, &\hbox{on}~~\partial{D},
\end{cases}
\end{align}
where $e(u)=\frac{1}{2}\left(\nabla u+(\nabla u)^{T}\right)$ is the elastic strain. It is well known that there is a unique variational solution $u\in H^{1}(D;\mathbb{R}^{d})$ of problem (\ref{La.001}) for $\varphi\in H^{1}( D;\mathbb{R}^{d})$, under the strong ellipticity condition
\begin{align*}
\mu>0,\quad d\lambda+2\mu>0,\quad \mu_1>0,\quad d\lambda_1+2\mu_1>0,
\end{align*}
and $\nabla u$ is piecewise H\"older continuous, see \cite{LN2003, BLL2015}.

Let
\begin{align}\label{LAK01}
\Psi:=\{\psi\in C^1(\mathbb{R}^{d}; \mathbb{R}^{d})\ |\ \nabla\psi+(\nabla\psi)^T=0\}
\end{align}
be the linear space of rigid displacement in $\mathbb{R}^{d}$. It is known that
\begin{align}\label{OPP}
\left\{\;e_{i},\;x_{k}e_{j}-x_{j}e_{k}\;\big|\;1\leq\,i\leq\,d,\;1\leq\,j<k\leq\,d\;\right\}
\end{align}
is a basis of $\Psi$, with $\{e_{1},...,e_{d}\}$ representing the standard basis of $\mathbb{R}^{d}$. We rewrite them as $\left\{\psi_{\alpha}\big|\,\alpha=1,2,...,\frac{d(d+1)}{2}\right\}$. For the convenience of computations, we adopt the following order with respect to $\psi_{\alpha}$: for $\alpha=1,2,...,d$, $\psi_{\alpha}=e_{\alpha}$; for $\alpha=d+1,...,2d-1$, $\psi_{\alpha}=x_{d}e_{\alpha-d}-x_{\alpha-d}e_{d}$; for $\alpha=2d,...,\frac{d(d+1)}{2}\,(d\geq3)$, there exist two indices $1\leq i_{\alpha}<j_{\alpha}<d$ such that
$\psi_{\alpha}=(0,...,0,x_{j_{\alpha}},0,...,0,-x_{i_{\alpha}},0,...,0)$.

Fix $\lambda$ and $\mu$ and let $u_{\lambda_{1},\mu_{1}}$ be the solution of (\ref{La.001}). As shown in the Appendix of \cite{BLL2015}, we know that
\begin{align*}
u_{\lambda_1,\mu_1}\rightarrow u\quad\hbox{in}\ H^1(D; \mathbb{R}^{d}),\quad \hbox{as}\ \min\{\mu_1, d\lambda_1+2\mu_1\}\rightarrow\infty,
\end{align*}
where $u\in H^1(D; \mathbb{R}^{d})$ solves
\begin{align}\label{La.002}
\begin{cases}
\mathcal{L}_{\lambda, \mu}u:=\nabla\cdot(\mathbb{C}^0e(u))=0,\quad&\hbox{in}\ \Omega,\\
u|_{+}=u|_{-},&\hbox{on}\ \partial{D}_{i},\,i=1,2,\\
e(u)=0,&\hbox{in}~D_{i},\,i=1,2,\\
\int_{\partial{D}_{i}}\frac{\partial u}{\partial \nu_0}\big|_{+}\cdot\psi_{\alpha}=0,&i=1,2,\,\alpha=1,2,...,\frac{d(d+1)}{2},\\
u=\varphi,&\hbox{on}\ \partial{D}.
\end{cases}
\end{align}
The co-normal derivative $\frac{\partial u}{\partial \nu_0}\big|_{+}$ is defined by
\begin{align*}
\frac{\partial u}{\partial \nu_0}\Big|_{+}&:=(\mathbb{C}^0e(u))\nu=\lambda(\nabla\cdot u)\nu+\mu(\nabla u+(\nabla u)^T)\nu,
\end{align*}
and $\nu$ represents the unit outer normal of $\partial D_{i}$, $i=1,2$. Here and throughout this paper the subscript $\pm$ shows the limit from outside and inside the domain, respectively. We refer to \cite{BLL2015} for the existence, uniqueness and regularity of weak solutions to (\ref{La.002}). Furthermore, the $H^{1}$ weak solution $u$ of problem (\ref{La.002}) has been proved to be in $C^1(\overline{\Omega};\mathbb{R}^{d})\cap C^1(\overline{D}_{1}\cup\overline{D}_{2};\mathbb{R}^{d})$ for any $C^{2,\alpha}$-domains.

We further assume that there exists a positive constant $R$, independent of $\varepsilon$, such that $\partial D_{1}$ and $\partial D_{2}$ near the origin can be represented, respectively, by the graphs of $x_{d}=\varepsilon+h_{1}(x')$ and $x_{d}=h_{2}(x')$, and $h_{i}$, $i=1,2$, verify that for $m\geq2$,
\begin{enumerate}
{\it\item[(\bf{H1})]
$h_{1}(x')-h_{2}(x')=\tau|x'|^{m}+O(|x'|^{m+\sigma}),\;\mathrm{if}\;x'\in B_{2R}',$
\item[(\bf{H2})]
$|\nabla_{x'}^{j}h_{i}(x')|\leq \kappa_{1}|x'|^{m-j},\;\mathrm{if}\;x'\in B_{2R}',\;i,j=1,2,$
\item[(\bf{H3})]
$\|h_{1}\|_{C^{2,\alpha}(B'_{2R})}+\|h_{2}\|_{C^{2,\alpha}(B'_{2R})}\leq \kappa_{2},$}
\end{enumerate}
where $\tau$, $\sigma$ and $\kappa_{i},i=1,2$, are four positive constants independent of $\varepsilon$. Furthermore, we suppose that for $x'\in B_{R}'$, $h_{1}(x')-h_{2}(x')$ is even with respect to $x_{i}$, $i=1,...,d-1$. For $z'\in B'_{R}$ and $0<t\leq2R$, denote the thin gap by
\begin{align*}
\Omega_{t}(z'):=&\left\{x\in \mathbb{R}^{d}~\big|~h_{2}(x')<x_{d}<\varepsilon+h_{1}(x'),~|x'-z'|<{t}\right\}.
\end{align*}
We adopt the abbreviated notation $\Omega_{t}$ for $\Omega_{t}(0')$ and denote its top and bottom boundaries by
\begin{align*}
\Gamma^{+}_{t}:=\left\{x\in\mathbb{R}^{d}|\,x_{d}=\varepsilon+h_{1}(x'),\;|x'|<r\right\},~~
\Gamma^{-}_{t}:=\left\{x\in\mathbb{R}^{d}|\,x_{d}=h_{2}(x'),\;|x'|<r\right\},
\end{align*}
respectively.

We first introduce a Keller-type scalar auxiliary function $\bar{v}\in C^{2}(\mathbb{R}^{d})$ such that $\bar{v}=1$ on $\partial D_{1}$, $\bar{v}=0$ on $\partial D_{2}\cup\partial D$,
\begin{align}\label{zh001}
\bar{v}(x',x_{d}):=\frac{x_{d}-h_{2}(x')}{\varepsilon+h_{1}(x')-h_{2}(x')},\;\,\mathrm{in}\;\Omega_{2R},\quad\mbox{and}~\|\bar{v}\|_{C^{2}(\Omega\setminus\Omega_{R})}\leq C.
\end{align}
Define a family of vector-valued auxiliary functions as follows: for $\alpha=1,2,...,\frac{d(d+1)}{2}$,
\begin{align}\label{zzwz002}
\bar{u}_{1}^{\alpha}=&\psi_{\alpha}\bar{v}+\mathcal{F}_{\alpha},\quad \mathrm{in}\;\Omega_{2R},
\end{align}
where $\psi_{\alpha}$ is defined in (\ref{OPP}) and
\begin{align}\label{QLA001}
\mathcal{F}_{\alpha}=\frac{\lambda+\mu}{\mu}f(\bar{v})\psi^{d}_{\alpha}\sum^{d-1}_{i=1}\partial_{x_{i}}\delta\,e_{i}+\frac{\lambda+\mu}{\lambda+2\mu}f(\bar{v})\sum^{d-1}_{i=1}\psi^{i}_{\alpha}\partial_{x_{i}}\delta\,e_{d},
\end{align}
and
\begin{align}\label{deta}
\delta(x'):=\varepsilon+h_{1}(x')-h_{2}(x'),\quad f(\bar{v}):=\frac{1}{2}\left(\bar{v}-\frac{1}{2}\right)^{2}-\frac{1}{8}.
\end{align}
We here remark that the correction terms $\mathcal{F}_{\alpha}$, $\alpha=1,2,...,d$, were captured in the previous work \cite{LX2020}.
\subsection{Main results}
To state the main results of this paper, we first fix some notations. Denote
\begin{align*}
\Gamma\Big[\frac{d+i-1}{m}\Big]:=&
\begin{cases}
\Gamma(1-\frac{d+i-1}{m})\Gamma(\frac{d+i-1}{m}),&m>d+i-1,\\
1,&m=d+i-1,
\end{cases}
\end{align*}
where $\Gamma(s)=\int^{+\infty}_{0}t^{s-1}e^{-t}dt$, $s>0$ is the Gamma function. Introduce the definite constants as follows:
\begin{align}\label{WEN}
\mathcal{M}_{0}=\frac{(d-1)\omega_{d-1}\Gamma[\frac{d-1}{m}]}{m\tau^{\frac{d-1}{m}}},\quad\mathcal{M}_{2}=\frac{\omega_{d-1}\Gamma[\frac{d+1}{m}]}{m\tau^{\frac{d+1}{m}}},
\end{align}
where $\omega_{d-1}$ represents the volume of $(d-1)$-dimensional unit ball and $\tau$ is defined in condition ({\bf{H1}}). Let
\begin{align}
&(\mathcal{L}_{2}^{1},\mathcal{L}_{2}^{2},\mathcal{L}_{2}^{3})=(\mu,\lambda+2\mu,\lambda+2\mu),\quad d=2,\label{AZ}\\
&(\mathcal{L}_{d}^{1},...,\mathcal{L}_{d}^{d-1},\mathcal{L}_{d}^{d},...,\mathcal{L}_{d}^{2d-1},\mathcal{L}_{d}^{2d},...,\mathcal{L}_{d}^{\frac{d(d+1)}{2}})\notag\\
&=(\mu,...,\mu,\lambda+2\mu,...,\lambda+2\mu,2\mu,...,2\mu),\quad d\geq3.\label{AZ110}
\end{align}
In addition, assume that for some constant $\kappa_{3}>0$,
\begin{align}\label{ITERA}
\kappa_{3}\leq\mu,d\lambda+2\mu\leq\frac{1}{\kappa_{3}}.
\end{align}
As for the blow-up rate indices, we denote, for $i=0,2$,
\begin{align}\label{rate00}
\rho_{i}(d,m;\varepsilon):=&
\begin{cases}
\varepsilon^{\frac{d+i-1}{m}-1},&m>d+i-1,\\
|\ln\varepsilon|,&m=d+i-1,\\
1,&m<d+i-1.
\end{cases}
\end{align}

Let $\Omega^{\ast}:=D\setminus\overline{D_{1}^{\ast}\cup D_{2}}$. For $i,j=1,2,\,\alpha,\beta=1,2,...,\frac{d(d+1)}{2}$, we define
\begin{align*}
a_{ij}^{\ast\alpha\beta}=\int_{\Omega^{\ast}}(\mathbb{C}^0e(v_{i}^{\ast\alpha}), e(v_j^{\ast\beta}))dx,\quad b_i^{\ast\alpha}=-\int_{\partial D}\frac{\partial v_{i}^{\ast\alpha}}{\partial \nu_0}\large\Big|_{+}\cdot\varphi,
\end{align*}
where $\varphi\in C^{2}(\partial D;\mathbb{R}^{d})$ is a given function and $v_{i}^{\ast\alpha}\in{C}^{2}(\Omega^{\ast};\mathbb{R}^d)$, $i=1,2$, $\alpha=1,2,...,\frac{d(d+1)}{2}$, respectively, solve
\begin{equation}\label{qaz001111}
\begin{cases}
\mathcal{L}_{\lambda,\mu}v_{1}^{\ast\alpha}=0,&\mathrm{in}~\Omega^{\ast},\\
v_{1}^{\ast\alpha}=\psi^{\alpha},&\mathrm{on}~\partial{D}_{1}^{\ast}\setminus\{0\},\\
v_{1}^{\ast\alpha}=0,&\mathrm{on}~\partial{D_{2}}\cup\partial{D},
\end{cases}\quad
\begin{cases}
\mathcal{L}_{\lambda,\mu}v_{2}^{\ast\alpha}=0,&\mathrm{in}~\Omega^{\ast},\\
v_{2}^{\ast\alpha}=\psi^{\alpha},&\mathrm{on}~\partial{D}_{2},\\
v_{2}^{\alpha}=0,&\mathrm{on}~(\partial{D_{1}^{\ast}}\setminus\{0\})\cup\partial{D}.
\end{cases}
\end{equation}
We here would like to point out that the definition of $a_{ij}^{\ast\alpha\beta}$ is valid only in some cases, see Lemma \ref{lemmabc} below for further details. Define the blow-up factor matrices as follows:
\begin{align}
&\mathbb{A}^{\ast}=(a_{11}^{\ast\alpha\beta})_{\frac{d(d+1)}{2}\times\frac{d(d+1)}{2}},\quad \mathbb{B}^{\ast}=\bigg(\sum\limits^{2}_{i=1}a_{i1}^{\ast\alpha\beta}\bigg)_{\frac{d(d+1)}{2}\times\frac{d(d+1)}{2}},\notag\\
&\mathbb{C}^{\ast}=\bigg(\sum\limits^{2}_{j=1}a_{1j}^{\ast\alpha\beta}\bigg)_{\frac{d(d+1)}{2}\times\frac{d(d+1)}{2}},\quad \mathbb{D}^{\ast}=\bigg(\sum\limits^{2}_{i,j=1}a_{ij}^{\ast\alpha\beta}\bigg)_{\frac{d(d+1)}{2}\times\frac{d(d+1)}{2}},\label{WEN002}
\end{align}
and for $\alpha=1,2,...,\frac{d(d+1)}{2}$,
\begin{gather}
\mathbb{F}^{\ast\alpha}=\begin{pmatrix} b_{1}^{\ast\alpha}&\sum\limits^{2}_{i=1}a_{i1}^{\ast\alpha1}&\cdots&\sum\limits^{2}_{i=1}a_{i1}^{\ast\alpha\frac{d(d+1)}{2}} \\ \sum\limits^{2}_{i=1}b_{i}^{\ast1}&\sum\limits_{i,j=1}^{2}a_{ij}^{\ast11}&\cdots&\sum\limits_{i,j=1}^{2}a_{ij}^{\ast1\frac{d(d+1)}{2}} \\ \vdots&\vdots&\ddots&\vdots\\
\sum\limits^{2}_{i=1}b_{i}^{\ast\frac{d(d+1)}{2}}&\;\;\sum\limits^{2}_{i,j=1}a_{ij}^{\ast\frac{d(d+1)}{2}1}&\;\cdots&\;\;\sum\limits_{i,j=1}^{2}a_{ij}^{\ast\frac{d(d+1)}{2}\frac{d(d+1)}{2}}
\end{pmatrix}.\label{WEN003}
\end{gather}
For the order of the rest terms, we use
\begin{align}
\varepsilon_{0}(d,m;\sigma):=&
\begin{cases}
\varepsilon^{\min\{\frac{1}{4},\frac{\sigma}{m},\frac{m-d-1}{m}\}},&m>d-1+\sigma,\,\sigma\geq2,\\
\varepsilon^{\min\{\frac{1}{4},\frac{m-d-1}{m}\}},&d+1<m\leq d-1+\sigma,\,\sigma>2,\\
|\ln\varepsilon|^{-1},&m=d+1,
\end{cases}\label{ZWZHAO112}\\
\varepsilon_{2}(d,m;\sigma):=&
\begin{cases}
\varepsilon^{\min\{\frac{1}{4},\frac{\sigma}{m}\}},&m>d+1+\sigma,\\
\max\{\varepsilon^{\frac{1}{4}},\varepsilon^{\frac{\sigma}{m}}|\ln\varepsilon|\},&m=d+1+\sigma,\\
\varepsilon^{\min\{\frac{1}{4},\frac{m-d-1}{m}\}},&d+1<m<d+1+\sigma,\\
|\ln\varepsilon|^{-1},&m=d+1.
\end{cases}\label{ZWZHAO113}
\end{align}

In what following we use, unless otherwise stated, $C$ to represent various positive constants whose values may vary from line to line and which depend only on $m,d,\tau,\kappa_{1},\kappa_{2},R$ and an upper bound of the $C^{2,\alpha}$ norms of $\partial D_{1}$, $\partial D_{2}$ and $\partial D$, but not on $\varepsilon$. Let $O(1)$ be some quantity satisfying  $|O(1)|\leq\,C$. Observe that by utilizing the standard elliptic theory (see Agmon et al. \cite{ADN1959,ADN1964}), we know
\begin{align*}
\|\nabla u\|_{L^{\infty}(\Omega\setminus\Omega_{R})}\leq\,C\|\varphi\|_{C^{2}(\partial D)}.
\end{align*}
Then it suffices to investigate the asymptotic behavior of $\nabla u$ in the narrow region $\Omega_{R}$.

To begin with, we state the asymptotic result in the case of $m\geq d+1$ as follows:
\begin{theorem}\label{Lthm066}
Let $D_{1},D_{2}\subset D\subset\mathbb{R}^{d}\,(d\geq2)$ be defined as above, conditions $\rm{(}${\bf{H1}}$\rm{)}$--$\rm{(}${\bf{H3}}$\rm{)}$ hold, and $\det\mathbb{F}^{\ast\alpha}\neq0,$ $\alpha=1,2,...,\frac{d(d+1)}{2}$. Let $u\in H^{1}(D;\mathbb{R}^{d})\cap C^{1}(\overline{\Omega};\mathbb{R}^{d})$ be the solution of (\ref{La.002}). Then for a sufficiently small $\varepsilon>0$ and $x\in\Omega_{R}$, if $m\geq d+1$,
\begin{align*}
\nabla u=&\sum\limits_{\alpha=1}^{d}\frac{\det\mathbb{F}^{\ast\alpha}}{\det \mathbb{D}^{\ast}}\frac{1+O(\varepsilon_{0}(d,m;\sigma))}{\mathcal{L}_{d}^{\alpha}\mathcal{M}_{0}\rho_{0}(d,m;\varepsilon)}\nabla\bar{u}^{\alpha}_{1}\notag\\
&+\sum\limits_{\alpha=d+1}^{\frac{d(d+1)}{2}}\frac{\det \mathbb{F}^{\ast\alpha}}{\det \mathbb{D}^{\ast}}\frac{1+O(\varepsilon_{2}(d,m;\sigma))}{\mathcal{L}_{d}^{\alpha}\mathcal{M}_{2}\rho_{2}(d,m;\varepsilon)}\nabla\bar{u}^{\alpha}_{1}+O(1)\|\varphi\|_{C^{0}(\partial D)},
\end{align*}
where the explicit auxiliary functions $\bar{u}_{1}^{\alpha}$, $\alpha=1,2,...,\frac{d(d+1)}{2}$ are defined in \eqref{zzwz002}, the constants $\mathcal{M}_{i}$, $i=0,2$, are defined in \eqref{WEN}, the Lam\'{e} constants $\mathcal{L}_{d}^{\alpha}$ are defined by \eqref{AZ}--\eqref{AZ110}, $\rho_{i}(d,m;\varepsilon)$, $i=0,2$ are defined in \eqref{rate00}, the blow-up factor matrices $\mathbb{D}^{\ast}$ and $\mathbb{F}^{\ast\alpha}$, $\alpha=1,2,...,\frac{d(d+1)}{2}$, are defined by \eqref{WEN002}--\eqref{WEN003}, the rest terms $\varepsilon_{i}(d,m;\sigma)$, $i=0,2,$ are defined in \eqref{ZWZHAO112}--\eqref{ZWZHAO113}.
\end{theorem}
\begin{remark}
From the results in Theorems \ref{Lthm066}, \ref{ZHthm002}, \ref{ZHthm003} and Corollary \ref{MGA001}, we give a complete description for the singularities of the gradient and meanwhile answer all the optimality of the blow-up rates for two adjacent $m$-convex inclusions in all dimensions. This improves and makes complete the gradient estimates and asymptotics in the previous work \cite{BLL2015,BLL2017,HL2018,L2018,LX2020,MZ2021}. By contrast with the gradient estimates, the main advantage of the precise asymptotic expansions lies in exhibiting all the higher and lower order singular terms in a unified expression.

\end{remark}
\begin{remark}
In view of decomposition \eqref{Decom002}, we see from the asymptotic expressions in Theorems \ref{Lthm066}--\ref{ZHthm003} that the singularity of $\nabla u$ consists of $\sum^{d}_{\alpha=1}(C_{1}^{\alpha}-C_{2}^{\alpha})\nabla v_{1}^{\alpha}$ and $\sum^{\frac{d(d+1)}{2}}_{\alpha=d+1}(C_{1}^{\alpha}-C_{2}^{\alpha})\nabla v_{1}^{\alpha}$. Observe that for $\alpha=1,2,...,\frac{d(d+1)}{2}$, the major singularity of $\nabla\bar{u}_{1}^{\alpha}$ is determined by $\partial_{x_{d}}\bar{u}_{1}^{\alpha}=\psi_{\alpha}\delta^{-1}+O(1)|\psi_{\alpha}|\delta^{-\frac{1}{m}}$. Then the first part $\sum^{d}_{\alpha=1}(C_{1}^{\alpha}-C_{2}^{\alpha})\nabla v_{1}^{\alpha}$ blows up at the rate of $(\varepsilon\rho_{0}(d,m;\varepsilon))^{-1}$ in $(d-1)$-dimensional sphere $\{|x'|\leq \sqrt[m]{\varepsilon}\}\cap\Omega$, while the second part $\sum^{\frac{d(d+1)}{2}}_{\alpha=d+1}(C_{1}^{\alpha}-C_{2}^{\alpha})\nabla v_{1}^{\alpha}$ blows up at the rate of $\big(\varepsilon^{\frac{m-1}{m}}\rho_{2}(d,m;\varepsilon)\big)^{-1}$ on the cylinder surface $\{|x'|=\sqrt[m]{\varepsilon}\}\cap\Omega$. Then we conclude that if $m>d$, then the maximal blow-up rate of $\nabla u$ lies in the second part; if $m=d$, then the maximal blow-up rate attains at these two parts simultaneously; if $m<d$, then the maximal blow-up rate is determined by the first part.
\end{remark}
\begin{remark}
The hypothetical condition $\det\mathbb{F}^{\ast\alpha}\neq0$, $\alpha=1,2,...,\frac{d(d+1)}{2}$ in Theorem \ref{Lthm066} implies that $\varphi\not\equiv0$ on $\partial D$. Otherwise, if $\varphi\equiv0$ on $\partial D$, then $b_i^{\alpha}=-\int_{\partial D}\frac{\partial v_{i}^{\alpha}}{\partial \nu_0}|_{+}\cdot\varphi=0$ such that $\det\mathbb{F}^{\ast\alpha}=0$, which is a contradiction. We also would like to point out that it is difficult to theoretically demonstrate $\det\mathbb{F}^{\ast\alpha}\neq0$ for any given boundary data mainly due to their dissymmetrical structure.

\end{remark}

Define
\begin{gather*}\mathbb{A}^{\ast}_{0}=\begin{pmatrix} a_{11}^{\ast d+1\,d+1}&\cdots&a_{11}^{\ast d+1\frac{d(d+1)}{2}} \\\\ \vdots&\ddots&\vdots\\\\a_{11}^{\ast\frac{d(d+1)}{2}d+1}&\cdots&a_{11}^{\ast\frac{d(d+1)}{2}\frac{d(d+1)}{2}}\end{pmatrix}  ,\;\,
\mathbb{B}^{\ast}_{0}=\begin{pmatrix} \sum\limits^{2}_{i=1}a_{i1}^{\ast d+1\,1}&\cdots&\sum\limits^{2}_{i=1}a_{i1}^{\ast d+1\,\frac{d(d+1)}{2}} \\\\ \vdots&\ddots&\vdots\\\\ \sum\limits^{2}_{i=1}a_{i1}^{\ast\frac{d(d+1)}{2}1}&\cdots&\sum\limits^{2}_{i=1}a_{i1}^{\ast\frac{d(d+1)}{2}\frac{d(d+1)}{2}}\end{pmatrix} ,\end{gather*}\begin{gather*}
\mathbb{C}^{\ast}_{0}=\begin{pmatrix} \sum\limits^{2}_{j=1}a_{1j}^{\ast1\,d+1}&\cdots&\sum\limits^{2}_{j=1}a_{1j}^{\ast1\frac{d(d+1)}{2}} \\\\ \vdots&\ddots&\vdots\\\\ \sum\limits^{2}_{j=1}a_{1j}^{\ast\frac{d(d+1)}{2}\,d+1}&\cdots&\sum\limits^{2}_{j=1}a_{1j}^{\ast\frac{d(d+1)}{2}\frac{d(d+1)}{2}}\end{pmatrix}.
\end{gather*}
Denote
\begin{align}\label{GC003}
\mathbb{F}^{\ast}_{0}=\begin{pmatrix} \mathbb{A}^{\ast}_{0}&\mathbb{B}^{\ast}_{0} \\  \mathbb{C}^{\ast}_{0}&\mathbb{D}^{\ast}
\end{pmatrix}.
\end{align}

On one hand, for $\alpha=1,2,...,d$, we denote
\begin{gather*}
\mathbb{A}^{\ast\alpha}_{1}=\begin{pmatrix}b_{1}^{\ast\alpha}&a_{11}^{\ast\alpha\,d+1}&\cdots&a_{11}^{\ast\alpha\frac{d(d+1)}{2}} \\ b_{1}^{\ast d+1}&a_{11}^{\ast d+1\,d+1}&\cdots&a_{11}^{\ast d+1\frac{d(d+1)}{2}}\\ \vdots&\vdots&\ddots&\vdots\\b_{1}^{\ast\frac{d(d+1)}{2}}&a_{11}^{\ast\frac{d(d+1)}{2}d+1}&\cdots&a_{11}^{\ast\frac{d(d+1)}{2}\frac{d(d+1)}{2}}
\end{pmatrix},
\end{gather*}
\begin{gather*}
\mathbb{B}^{\ast\alpha}_{1}=\begin{pmatrix}\sum\limits_{i=1}^{2}a_{i1}^{\ast\alpha1}&\sum\limits_{i=1}^{2}a_{i1}^{\ast\alpha2}&\cdots&\sum\limits_{i=1}^{2}a_{i1}^{\ast\alpha\,\frac{d(d+1)}{2}} \\ \sum\limits_{i=1}^{2}a_{i1}^{\ast d+1\,1}&\sum\limits_{i=1}^{2}a_{i1}^{\ast d+1\,2}&\cdots&\sum\limits_{i=1}^{2}a_{i1}^{\ast d+1\,\frac{d(d+1)}{2}} \\ \vdots&\vdots&\ddots&\vdots\\\sum\limits_{i=1}^{2}a_{i1}^{\ast\frac{d(d+1)}{2}\,1}&\sum\limits_{i=1}^{2}a_{i1}^{\ast\frac{d(d+1)}{2}\,2}&\cdots&\sum\limits_{i=1}^{2}a_{i1}^{\ast\frac{d(d+1)}{2}\frac{d(d+1)}{2}}
\end{pmatrix},
\end{gather*}
\begin{gather*}
\mathbb{C}^{\ast\alpha}_{1}=\begin{pmatrix}\sum\limits_{i=1}^{2}b_{i}^{\ast1}&\sum\limits_{j=1}^{2}a_{1j}^{\ast1\,d+1}&\cdots&\sum\limits_{j=1}^{2}a_{1j}^{\ast1\,\frac{d(d+1)}{2}} \\\sum\limits_{i=1}^{2}b_{i}^{\ast2}&\sum\limits_{j=1}^{2}a_{1j}^{\ast2\,d+1}&\cdots&\sum\limits_{j=1}^{2}a_{1j}^{\ast2\,\frac{d(d+1)}{2}}\\ \vdots&\vdots&\ddots&\vdots\\\sum\limits_{i=1}^{2}b_{i}^{\ast\frac{d(d+1)}{2}}&\sum\limits_{j=1}^{2}a_{1j}^{\ast\frac{d(d+1)}{2}\,d+1}&\cdots&\sum\limits_{j=1}^{2}a_{1j}^{\ast\frac{d(d+1)}{2}\frac{d(d+1)}{2}}
\end{pmatrix}.
\end{gather*}
Write
\begin{align}\label{GC005}
\mathbb{F}^{\ast\alpha}_{1}=\begin{pmatrix} \mathbb{A}^{\ast\alpha}_{1}&\mathbb{B}^{\ast\alpha}_{1} \\  \mathbb{C}^{\ast\alpha}_{1}&\mathbb{D}^{\ast}
\end{pmatrix},\quad\alpha=1,2,...,d.
\end{align}

On the other hand, for $\alpha=d+1,...,\frac{d(d+1)}{2}$, we replace the elements of $\alpha$-th column in the matrices $\mathbb{A}^{\ast}_{0}$ and  $\mathbb{C}^{\ast}_{0}$ by column vector $\Big(b_{1}^{\ast d+1},...,b_{1}^{\ast\frac{d(d+1)}{2}}\Big)^{T}$ and $\Big(\sum\limits_{i=1}^{2}b_{i}^{\ast1},...,\sum\limits_{i=1}^{2}b_{i}^{\ast\frac{d(d+1)}{2}}\Big)^{T}$, respectively. Then we denote these two new matrices by $\mathbb{A}_{2}^{\ast\alpha}$ and $\mathbb{C}_{2}^{\ast\alpha}$ as follows:
\begin{gather*}
\mathbb{A}_{2}^{\ast\alpha}=
\begin{pmatrix}
a_{11}^{\ast d+1\,d+1}&\cdots&b_{1}^{d+1}&\cdots&a_{11}^{\ast d+1\,\frac{d(d+1)}{2}} \\\\ \vdots&\ddots&\vdots&\ddots&\vdots\\\\a_{11}^{\ast\frac{d(d+1)}{2}\,d+1}&\cdots&b_{1}^{\ast\frac{d(d+1)}{2}}&\cdots&a_{11}^{\ast\frac{d(d+1)}{2}\,\frac{d(d+1)}{2}}
\end{pmatrix},
\end{gather*}
and
\begin{gather*}
\mathbb{C}_{2}^{\ast\alpha}=
\begin{pmatrix}
\sum\limits_{j=1}^{2} a_{1j}^{\ast1\,d+1}&\cdots&\sum\limits_{i=1}^{2}b_{i}^{\ast1}&\cdots&\sum\limits_{j=1}^{2} a_{1j}^{\ast1\,\frac{d(d+1)}{2}} \\\\ \vdots&\ddots&\vdots&\ddots&\vdots\\\\\sum\limits_{j=1}^{2} a_{1j}^{\ast\frac{d(d+1)}{2}\,d+1}&\cdots&\sum\limits_{i=1}^{2}b_{i}^{\ast\frac{d(d+1)}{2}}&\cdots&\sum\limits_{j=1}^{2} a_{1j}^{\ast\frac{d(d+1)}{2}\,\frac{d(d+1)}{2}}
\end{pmatrix}.
\end{gather*}
Define
\begin{align}\label{GC006}
\mathbb{F}^{\ast\alpha}_{2}=\begin{pmatrix} \mathbb{A}^{\ast\alpha}_{2}&\mathbb{B}^{\ast\alpha}_{0} \\  \mathbb{C}^{\ast\alpha}_{2}&\mathbb{D}^{\ast}
\end{pmatrix},\quad\alpha=d+1,...,\frac{d(d+1)}{2}.
\end{align}
For the remaining terms, define
\begin{align}
\bar{\varepsilon}_{0}(d,m;\sigma):=&
\begin{cases}
\varepsilon^{\min\{\frac{\sigma}{m},\frac{d+1-m}{12m}\}},&d-1+\sigma<m<d+1,\,\sigma<2,\\
\max\{\varepsilon^{\frac{\sigma}{m}}|\ln\varepsilon|,\varepsilon^{\frac{d+1-m}{12m}}\},&m=d-1+\sigma,\,\sigma<2,\\
\varepsilon^{\min\{\frac{m-d+1}{m},\frac{d+1-m}{12m}\}},&d-1<m<d-1+\sigma,\,\sigma\leq2,\\
|\ln\varepsilon|^{-1},&m=d-1,
\end{cases}\label{GC001}\\
\bar{\varepsilon}_{2}(d,m;\sigma):=&
\begin{cases}
\varepsilon^{\min\{\frac{m-d+1}{m},\frac{d+1-m}{12m}\}},&d-1<m<d+1,\\
|\ln\varepsilon|^{-1},&m=d-1.
\end{cases}\label{GC002}
\end{align}

Second, in the case of $d-1\leq m<d+1$, we obtain the following asymptotic expansion.
\begin{theorem}\label{ZHthm002}
Let $D_{1},D_{2}\subset D\subset\mathbb{R}^{d}\,(d\geq2)$ be defined as above, conditions $\rm{(}${\bf{H1}}$\rm{)}$--$\rm{(}${\bf{H3}}$\rm{)}$ hold, $\det\mathbb{F}^{\ast\alpha}_{1}\neq0,$ $\alpha=1,2,...,d$, and $\det\mathbb{F}^{\ast\alpha}_{2}\neq0,$ $\alpha=d+1,...,\frac{d(d+1)}{2}$. Let $u\in H^{1}(D;\mathbb{R}^{d})\cap C^{1}(\overline{\Omega};\mathbb{R}^{d})$ be the solution of (\ref{La.002}). Then for a sufficiently small $\varepsilon>0$ and $x\in\Omega_{R}$, if $d-1\leq m<d+1$,
\begin{align*}
\nabla u=&\sum\limits_{\alpha=1}^{d}\frac{\det\mathbb{F}_{1}^{\ast\alpha}}{\det \mathbb{F}_{0}^{\ast}}\frac{1+O(\bar{\varepsilon}_{0}(d,m;\sigma))}{\mathcal{L}_{d}^{\alpha}\mathcal{M}_{0}\rho_{0}(d,m;\varepsilon)}\nabla\bar{u}_{1}^{\alpha}\notag\\
&+\sum\limits_{\alpha=d+1}^{\frac{d(d+1)}{2}}\frac{\det\mathbb{F}_{2}^{\ast\alpha}}{\det \mathbb{F}_{0}^{\ast}}(1+O(\bar{\varepsilon}_{2}(d,m;\sigma)))\nabla\bar{u}^{\alpha}_{1}+O(1)\|\varphi\|_{C^{0}(\partial D)},
\end{align*}
where the explicit auxiliary functions $\bar{u}_{1}^{\alpha}$, $\alpha=1,2,...,\frac{d(d+1)}{2}$ are defined in \eqref{zzwz002}, the constant $\mathcal{M}_{0}$ is defined in \eqref{WEN}, the Lam\'{e} constants $\mathcal{L}_{d}^{\alpha}$ are defined by \eqref{AZ}--\eqref{AZ110}, $\rho_{i}(d,m;\varepsilon)$, $i=0,2$ are defined in \eqref{rate00}, the blow-up factor matrices $\mathbb{F}_{0}^{\ast}$, $\mathbb{F}^{\ast\alpha}_{1},$ $\alpha=1,2,...,d$, $\mathbb{F}^{\ast\alpha}_{2},$ $\alpha=d+1,...,\frac{d(d+1)}{2}$ are defined by \eqref{GC003}--\eqref{GC006}, the rest terms $\bar{\varepsilon}_{i}(d,m;\sigma)$, $i=0,2,$ are defined in \eqref{GC001}--\eqref{GC002}.

\end{theorem}

For $\alpha=1,2,...,\frac{d(d+1)}{2}$, we replace the elements of $\alpha$-th column in the matrices $\mathbb{A}^{\ast}$ and $\mathbb{C}^{\ast}$ by column vectors $\Big(b_{1}^{\ast1},...,b_{1}^{\frac{\ast d(d+1)}{2}}\Big)^{T}$ and $\Big(\sum\limits_{i=1}^{2}b_{i}^{\ast1},...,\sum\limits_{i=1}^{2}b_{i}^{\ast\frac{d(d+1)}{2}}\Big)^{T}$, respectively, and then denote these two new matrices by $\mathbb{A}_{3}^{\ast\alpha}$ and $\mathbb{C}_{3}^{\ast\alpha}$ as follows:
\begin{gather*}
\mathbb{A}_{3}^{\ast\alpha}=
\begin{pmatrix}
a_{11}^{\ast11}&\cdots&b_{1}^{\ast1}&\cdots&a_{11}^{\ast1\,\frac{d(d+1)}{2}} \\\\ \vdots&\ddots&\vdots&\ddots&\vdots\\\\a_{11}^{\ast\frac{d(d+1)}{2}\,1}&\cdots&b_{1}^{\ast\frac{d(d+1)}{2}}&\cdots&a_{11}^{\ast\frac{d(d+1)}{2}\,\frac{d(d+1)}{2}}
\end{pmatrix},
\end{gather*}
and
\begin{gather*}
\mathbb{C}_{3}^{\ast\alpha}=
\begin{pmatrix}
\sum\limits_{j=1}^{2} a_{1j}^{\ast11}&\cdots&\sum\limits_{i=1}^{2}b_{i}^{\ast1}&\cdots&\sum\limits_{j=1}^{2} a_{1j}^{\ast1\,\frac{d(d+1)}{2}} \\\\ \vdots&\ddots&\vdots&\ddots&\vdots\\\\\sum\limits_{j=1}^{2} a_{1j}^{\ast\frac{d(d+1)}{2}\,1}&\cdots&\sum\limits_{i=1}^{2}b_{i}^{\ast\frac{d(d+1)}{2}}&\cdots&\sum\limits_{j=1}^{2} a_{1j}^{\ast\frac{d(d+1)}{2}\,\frac{d(d+1)}{2}}
\end{pmatrix}.
\end{gather*}
Define
\begin{align}\label{GC009}
\mathbb{F}^{\ast\alpha}_{3}=\begin{pmatrix} \mathbb{A}^{\ast\alpha}_{3}&\mathbb{B}^{\ast} \\  \mathbb{C}^{\ast\alpha}_{3}&\mathbb{D}^{\ast}
\end{pmatrix},\;\,\alpha=1,2,...,\frac{d(d+1)}{2},\quad \mathbb{F}=\begin{pmatrix} \mathbb{A}^{\ast}&\mathbb{B}^{\ast} \\  \mathbb{C}^{\ast}&\mathbb{D}^{\ast}
\end{pmatrix}.
\end{align}

Then, the asymptotic result corresponding to the case of $m<d-1$ is listed as follows:
\begin{theorem}\label{ZHthm003}
Let $D_{1},D_{2}\subset D\subset\mathbb{R}^{d}\,(d\geq2)$ be defined as above, conditions $\rm{(}${\bf{H1}}$\rm{)}$--$\rm{(}${\bf{H3}}$\rm{)}$ hold, and $\det\mathbb{F}_{3}^{\ast\alpha}\neq0,$ $\alpha=1,2,...,\frac{d(d+1)}{2}$. Let $u\in H^{1}(D;\mathbb{R}^{d})\cap C^{1}(\overline{\Omega};\mathbb{R}^{d})$ be the solution of (\ref{La.002}). Then for a sufficiently small $\varepsilon>0$ and $x\in\Omega_{R}$, if $m<d-1$,
\begin{align*}
\nabla u=&\sum\limits_{\alpha=1}^{\frac{d(d+1)}{2}}\frac{\det\mathbb{F}_{3}^{\ast\alpha}}{\det \mathbb{F}^{\ast}}(1+O(\varepsilon^{\min\{\frac{1}{6},\frac{d-1-m}{12m}\}}))\nabla\bar{u}^{\alpha}_{1}+O(1)\|\varphi\|_{C^{0}(\partial D)},
\end{align*}
where the explicit auxiliary functions $\bar{u}_{1}^{\alpha}$, $\alpha=1,2,...,\frac{d(d+1)}{2}$ are defined in \eqref{zzwz002}, the blow-up factor matrices $\mathbb{F}^{\ast}$ and $\mathbb{F}_{3}^{\ast\alpha},$ $\alpha=1,2,...,\frac{d(d+1)}{2}$, are defined by \eqref{GC009}.

\end{theorem}

For the inclusions of arbitrary $m$-convex shape satisfying
\begin{align}\label{KANZ001}
\tau_{1}|x'|^{m}\leq (h_{1}-h_{2})(x')\leq\tau_{2}|x'|^{m},\quad\mathrm{for}\;|x'|\leq2R,\;\tau_{i}>0,\;i=1,2,
\end{align}
by applying the proofs of Theorems \ref{Lthm066}, \ref{ZHthm002} and \ref{ZHthm003} with a minor modification, we obtain the optimal pointwise upper and lower bounds on the gradients as follows:

\begin{corollary}\label{MGA001}
Let $D_{1},D_{2}\subset D\subset\mathbb{R}^{d}\,(d\geq2)$ be defined as above, conditions \eqref{KANZ001} and $\rm{(}${\bf{H2}}$\rm{)}$--$\rm{(}${\bf{H3}}$\rm{)}$ hold. Let $u\in H^{1}(D;\mathbb{R}^{d})\cap C^{1}(\overline{\Omega};\mathbb{R}^{d})$ be the solution of (\ref{La.002}). Then for a sufficiently small $\varepsilon>0$,

$(i)$ if $m>d$, there exists some integer $d+1\leq \alpha_{0}\leq\frac{d(d+1)}{2}$ such that $\det\mathbb{F}^{\ast\alpha_{0}}\neq0$, then for some $1\leq i_{\alpha_{0}}\leq d-1$ and $x\in\{x'|\,x_{i_{\alpha_{0}}}=\sqrt[m]{\varepsilon},\,x_{i}=0,i\neq i_{\alpha_{0}},1\leq i\leq d-1\}\cap\Omega$,
\begin{align*}
|\nabla u|\leq
\begin{cases}
\frac{\max\limits_{d+1\leq\alpha\leq\frac{d(d+1)}{2}}\tau_{2}^{\frac{d+1}{m}}|\mathcal{L}_{d}^{\alpha}|^{-1}|\det\mathbb{F}^{\ast\alpha}|}{(1+\tau_{1})|\det\mathbb{D}^{\ast}|}\frac{C\varepsilon^{\frac{1-m}{m}}}{\rho_{2}(d,m;\varepsilon)},&m\geq d+1,\\\\
\frac{\max\limits_{d+1\leq\alpha\leq\frac{d(d+1)}{2}}|\det\mathbb{F}_{2}^{\ast\alpha}|}{(1+\tau_{1})|\det\mathbb{F}_{0}^{\ast}|}\frac{C}{\varepsilon^{\frac{m-1}{m}}},&d<m<d+1,
\end{cases}
\end{align*}
and
\begin{align*}
|\nabla u|\geq
\begin{cases}
\frac{\tau_{1}^{\frac{d+1}{m}}|\det\mathbb{F}^{\ast\alpha_{0}}|}{(1+\tau_{2})|\mathcal{L}_{d}^{\alpha_{0}}||\det\mathbb{D}^{\ast}|}\frac{\varepsilon^{\frac{1-m}{m}}}{C\rho_{2}(d,m;\varepsilon)},&m\geq d+1,\\\\
\frac{|\det\mathbb{F}^{\ast\alpha_{0}}_{2}|}{(1+\tau_{2})|\det\mathbb{F}^{\ast}_{0}|}\frac{1}{C\varepsilon^{\frac{m-1}{m}}},&d<m<d+1;
\end{cases}
\end{align*}

$(ii)$ if $m\leq d$, there exists some integer $1\leq \alpha_{0}\leq d$ such that $\det\mathbb{F}_{1}^{\ast\alpha_{0}}\neq0$, then for $x\in\{|x'|=0\}\cap\Omega$,
\begin{align*}
|\nabla u|\leq
\begin{cases}
\frac{\max\limits_{1\leq\alpha\leq d}\tau_{2}^{\frac{d-1}{m}}|\mathcal{L}_{d}^{\alpha}|^{-1}|\det\mathbb{F}_{1}^{\ast\alpha}|}{|\det\mathbb{F}_{0}^{\ast}|}\frac{C}{\varepsilon\rho_{0}(d,m;\varepsilon)},&d-1\leq m\leq d,\\\\
\frac{\max\limits_{1\leq\alpha\leq d}|\det\mathbb{F}_{3}^{\ast\alpha}|}{|\det \mathbb{F}^{\ast}|}\frac{C}{\varepsilon},&m<d-1,
\end{cases}
\end{align*}
and
\begin{align*}
|\nabla u|\geq
\begin{cases}
\frac{\tau_{1}^{\frac{d-1}{m}}|\det\mathbb{F}_{1}^{\ast\alpha_{0}}|}{|\mathcal{L}_{d}^{\alpha_{0}}||\det\mathbb{F}_{0}^{\ast}|}\frac{1}{C\varepsilon\rho_{0}(d,m;\varepsilon)},&d-1\leq m\leq d,\\\\
\frac{|\det\mathbb{F}_{3}^{\ast\alpha_{0}}|}{|\det \mathbb{F}^{\ast}|}\frac{1}{C\varepsilon},&m<d-1,
\end{cases}
\end{align*}
where the blow-up factor matrices $\mathbb{D}^{\ast}$, $\mathbb{F}^{\ast}$, $\mathbb{F}_{0}^{\ast}$, $\mathbb{F}_{1}^{\ast\alpha}$, $\alpha=1,2,...,d$, and $\mathbb{F}^{\ast\alpha}$, $\mathbb{F}_{2}^{\ast3}$, $\alpha=d+1,...,\frac{d(d+1)}{2}$, are defined by \eqref{WEN002}--\eqref{WEN003}, \eqref{GC003}--\eqref{GC006} and \eqref{GC009}, respectively.

\end{corollary}
\begin{remark}
It is worth emphasizing that the optimality of the gradient blow-up rate in all dimensions under condition \eqref{KANZ001} was first proved in \cite{MZ202102}, where a family of unified blow-up factors were used to establish the optimal gradient estimates.
\end{remark}


\section{Preliminary}\label{SEC003}

\subsection{Properties of the tensor $\mathbb{C}^{0}$}
We first make note of some properties of the tensor $\mathbb{C}^{0}$. For the isotropic elastic material, let
\begin{align*}
\mathbb{C}^{0}:=(C_{ijkl}^{0})=(\lambda\delta_{ij}\delta_{kl}+\mu(\delta_{ik}\delta_{jl}+\delta_{il}\delta_{jk})),\quad \mu>0,\quad d\lambda+2\mu>0.
\end{align*}
Observe that the components $C_{ijkl}^{0}$ have the following symmetry property:
\begin{align}\label{symm}
C_{ijkl}^{0}=C_{klij}^{0}=C_{klji}^{0},\quad i,j,k,l=1,2,...,d.
\end{align}
For every pair of $d\times d$ matrices $\mathbb{A}=(a_{ij})$ and $\mathbb{B}=(b_{ij})$, write
\begin{align*}
(\mathbb{C}^{0}\mathbb{A})_{ij}=\sum_{k,l=1}^{n}C_{ijkl}^{0}a_{kl},\quad\hbox{and}\quad(\mathbb{A},\mathbb{B})\equiv \mathbb{A}:\mathbb{B}=\sum_{i,j=1}^{d}a_{ij}b_{ij}.
\end{align*}
Therefore,
$$(\mathbb{C}^{0}\mathbb{A},\mathbb{B})=(\mathbb{A}, \mathbb{C}^{0}\mathbb{B}).$$
From (\ref{symm}), we see that $\mathbb{C}^{0}$ verifies the ellipticity condition, that is, for every $d\times d$ real symmetric matrix $\xi=(\xi_{ij})$,
\begin{align}\label{ellip}
\min\{2\mu, d\lambda+2\mu\}|\xi|^2\leq(\mathbb{C}^{0}\xi, \xi)\leq\max\{2\mu, d\lambda+2\mu\}|\xi|^2,
\end{align}
where $|\xi|^2=\sum\limits_{ij}\xi_{ij}^2.$ In particular,
\begin{align*}
\min\{2\mu, d\lambda+2\mu\}|\mathbb{A}+\mathbb{A}^T|^2\leq(\mathbb{C}(\mathbb{A}+\mathbb{A}^T), (\mathbb{A}+\mathbb{A}^T)).
\end{align*}
In addition, we know that for any open set $O$ and $u, v\in C^2(O;\mathbb{R}^{d})$,
\begin{align}\label{Le2.01222}
\int_O(\mathbb{C}^0e(u), e(v))\,dx=-\int_O\left(\mathcal{L}_{\lambda, \mu}u\right)\cdot v+\int_{\partial O}\frac{\partial u}{\partial \nu_0}\Big|_{+}\cdot v.
\end{align}

\subsection{Solution split}\label{sec_thm1}

As shown in \cite{BLL2015,BLL2017}, we carry out a linear decomposition for the solution $u(x)$ of \eqref{La.002} as follows:
\begin{equation}\label{Decom}
u(x)=\sum_{\alpha=1}^{\frac{d(d+1)}{2}}C_1^{\alpha}v_{1}^{\alpha}(x)+\sum_{\alpha=1}^{\frac{d(d+1)}{2}}C_2^{\alpha}v_2^{\alpha}(x)+v_{0}(x),\qquad~x\in\,\Omega ,
\end{equation}
where the constants $C_{i}^{\alpha}$, $i=1,2,\,\alpha=1,2,...,\frac{d(d+1)}{2}$ can be determined by the fourth line of \eqref{La.002}, $v_{0}$ and $v_{i}^{\alpha}\in{C}^{2}(\Omega;R^d)$, $i=1,2$, $\alpha=1,2,\cdots,\frac{d(d+1)}{2}$, respectively, solve
\begin{equation}\label{qaz001}
\begin{cases}
\mathcal{L}_{\lambda,\mu}v_{0}=0,&\mathrm{in}~\Omega,\\
v_{0}=0,&\mathrm{on}~\partial{D}_{1}\cup\partial{D_{2}},\\
v_{0}=\varphi,&\mathrm{on}~\partial{D},
\end{cases}\quad
\begin{cases}
\mathcal{L}_{\lambda,\mu}v_{i}^{\alpha}=0,&\mathrm{in}~\Omega,\\
v_{i}^{\alpha}=\psi^{\alpha},&\mathrm{on}~\partial{D}_{i},~i=1,2,\\
v_{i}^{\alpha}=0,&\mathrm{on}~\partial{D_{j}}\cup\partial{D},~j\neq i.
\end{cases}
\end{equation}
From \eqref{Decom}, we have
\begin{align}\label{Decom002}
\nabla{u}=&\sum_{\alpha=1}^{\frac{d(d+1)}{2}}(C_{1}^\alpha-C_{2}^\alpha)\nabla{v}_{1}^\alpha+\sum_{\alpha=1}^{\frac{d(d+1)}{2}}C_{2}^\alpha\nabla({v}_{1}^\alpha+{v}_{2}^\alpha)+\nabla{v}_{0}.
\end{align}
Using \eqref{Decom002}, we decompose $\nabla u$ into two parts as follows: one of them is the singular part  $\sum_{\alpha=1}^{\frac{d(d+1)}{2}}(C_{1}^\alpha-C_{2}^\alpha)\nabla{v}_{1}^\alpha$ and appears blow-up; the other is the regular part $\sum_{\alpha=1}^{\frac{d(d+1)}{2}}C_{2}^\alpha\nabla({v}_{1}^\alpha+{v}_{2}^\alpha)+\nabla{v}_{0}$ and decays exponentially near the origin. We give precise statements for these results in the following sections.

\subsection{A general boundary value problem}

To investigate the asymptotic behavior of $\nabla v_{1}^{\alpha}$, $\alpha=1,2,...,\frac{d(d+1)}{2}$, we begin with considering the following general boundary value problem:
\begin{equation}\label{P2.008}
\begin{cases}
\mathcal{L}_{\lambda,\mu}v:=\nabla\cdot(\mathbb{C}^{0}e(v))=0,\quad\;\,&\mathrm{in}\;\,\Omega,\\
v=\psi(x),&\mathrm{on}\;\,\partial D_{1},\\
v=\phi(x),&\mathrm{on}\;\,\partial D_{2},\\
v=0,&\mathrm{on}\;\,\partial D,
\end{cases}
\end{equation}
where $\psi\in C^{2}(\partial D_{1};\mathbb{R}^{d})$ and $\phi\in C^{2}(\partial D_{2};\mathbb{R}^{d})$ are two given vector-valued functions.

Introduce a vector-valued auxiliary function as follows:
\begin{align}\label{CAN01}
\tilde{v}=&\psi(x',\varepsilon+h_{1}(x'))\bar{v}+\phi(x',h_{2}(x'))(1-\bar{v})\notag\\
&+\frac{\lambda+\mu}{\mu}f(\bar{v})(\psi^{d}(x',\varepsilon+h_{1}(x'))-\phi^{d}(x',h_{2}(x')))\sum^{d-1}_{i=1}\partial_{x_{i}}\delta\,e_{i}\notag\\
&+\frac{\lambda+\mu}{\lambda+2\mu}f(\bar{v})\sum^{d-1}_{i=1}\partial_{x_{i}}\delta(\psi^{i}(x',\varepsilon+h_{1}(x'))-\phi^{i}(x',h_{2}(x')))\,e_{d},
\end{align}
where $\bar{v}$ is defined by \eqref{zh001}, $\delta$ and $f(\bar{v})$ are defined in \eqref{deta}. This auxiliary function was previously captured in Theorem 3.1 of \cite{MZ202102}. For the residual part, denote
\begin{align}\label{remin001}
\mathcal{R}_{\delta}(\psi,\phi)=&|\psi(x',\varepsilon+h_{1}(x'))-\phi(x',h_{2}(x'))|\delta^{\frac{m-2}{m}}+\delta\big(\|\psi\|_{C^{2}(\partial D_{1})}+\|\phi\|_{C^{2}(\partial D_{2})})\notag\\
&+|\nabla_{x'}(\psi(x',\varepsilon+h_{1}(x'))-\phi(x',h_{2}(x')))|.
\end{align}
\begin{prop}[Theorem 3.1 of \cite{MZ202102}]\label{thm8698}
Assume as above. Let $v\in H^{1}(\Omega;\mathbb{R}^{d})$ be a weak solution of (\ref{P2.008}). Then for a sufficiently small $\varepsilon>0$ and $x\in\Omega_{R}$,
\begin{align*}
\nabla v=\nabla\tilde{v}+O(1)\mathcal{R}_{\delta}(\psi,\phi),
\end{align*}
where $\delta$ is defined in \eqref{deta}, the leading term $\tilde{v}$ is defined by \eqref{CAN01}, the residual part $\mathcal{R}_{\delta}(\psi,\phi)$ is defined by \eqref{remin001}.
\end{prop}
From Proposition \ref{thm8698}, we see that $\nabla\tilde{v}$ completely describe the singularity of the gradient $\nabla v$ of a solution to a general boundary value problem \eqref{P2.008}. For the convenience of readability, we also leave the proof of Proposition \ref{thm8698} in the Appendix.

\section{Proofs of Theorems \ref{Lthm066}, \ref{ZHthm002} and \ref{ZHthm003}}\label{SEC004}

For $\alpha=1,2,...,\frac{d(d+1)}{2}$, we define
\begin{align}\label{GAKL001}
\bar{u}_{2}^{\alpha}=&\psi_{\alpha}(1-\bar{v})-\mathcal{F}_{\alpha},
\end{align}
where the correction terms $\mathcal{F}_{\alpha}$, $\alpha=1,2,...,\frac{d(d+1)}{2}$ are given in \eqref{QLA001}. Then applying Proposition \ref{thm8698} with $\psi=\psi_{\alpha},\,\phi=0$ or $\psi=0,\,\phi=\psi_{\alpha}$, $\alpha=1,2,...,\frac{d(d+1)}{2}$, we have
\begin{corollary}\label{thm86}
Assume as above. Let $v_{i}^{\alpha}\in H^{1}(\Omega;\mathbb{R}^{d})$, $i=1,2$, $\alpha=1,2,...,\frac{d(d+1)}{2}$ be a weak solution of \eqref{qaz001}. Then, for a sufficiently small $\varepsilon>0$ and $x\in\Omega_{R}$,
\begin{align}\label{Le2.025}
\nabla v_{i}^{\alpha}=\nabla\bar{u}_{i}^{\alpha}+O(1)
\begin{cases}
\delta^{\frac{m-2}{m}},&\alpha=1,2,...,d,\\
1,&\alpha=d+1,...,\frac{d(d+1)}{2},
\end{cases}
\end{align}
where $\delta$ is defined in \eqref{deta}, the main terms $\bar{u}_{i}^{\alpha}$, $i=1,2,\,\alpha=1,2,...,\frac{d(d+1)}{2}$ are defined in \eqref{zzwz002} and \eqref{GAKL001}.
\end{corollary}

As seen in Theorem 1.1 of \cite{LLBY2014}, the gradients of solutions to a class of elliptic
systems with the same boundary data on the upper and bottom boundaries of the narrow regions will appear no blow-up and possess the exponentially decaying property. A direct application of Theorem 1.1 in \cite{LLBY2014} yields that
\begin{corollary}\label{coro00z}
Assume as above. Let $v_{i}^{\ast\alpha}$ and $v_{i}^{\alpha}$, $i=1,2$, $\alpha=1,2,...,\frac{d(d+1)}{2}$ be the solutions of \eqref{qaz001}, respectively. Then, we have
\begin{align*}
|\nabla v_{0}|+\left|\sum^{2}_{i=1}\nabla v_{i}^{\alpha}\right|\leq C\delta^{-\frac{d}{2}}e^{-\frac{1}{2C\delta^{1-1/m}}},\;\;\mathrm{in}\;\Omega_{R},
\end{align*}
and
\begin{align*}
\left|\sum^{2}_{i=1}\nabla v_{i}^{\ast\alpha}\right|\leq C|x'|^{-\frac{md}{2}}e^{-\frac{1}{2C|x'|^{m-1}}},\;\;\mathrm{in}\;\Omega_{R}^{\ast},
\end{align*}
where the constant $C$ depends on $m,d,\lambda,\mu,\tau,\kappa_{1},\kappa_{2}$, but not on $\varepsilon$.
\end{corollary}
The proof of this corollary is a slight modification of Theorem 1.1 in \cite{LLBY2014} and thus omitted here.

We now state a result in terms of the boundedness of $C_{i}^{\alpha}$, $i=1,2,$ $\alpha=1,2,...,\frac{d(d+1)}{2}$. Its proof was given in Lemma 4.1 of \cite{BLL2015}.
\begin{lemma}\label{PAK001}
Let $C_{i}^{\alpha}$, $i=1,2,\,\alpha=1,2,...,\frac{d(d+1)}{2}$ be defined in \eqref{Decom}. Then
\begin{align*}
|C_{i}^{\alpha}|\leq C,\quad i=1,2,\,\alpha=1,2,...,\frac{d(d+1)}{2},
\end{align*}
where $C$ is a positive constant independent of $\varepsilon$.
\end{lemma}

On the other hand, with regard to the asymptotic expansions of $C^{\alpha}_{1}-C_{2}^{\alpha}$, $\alpha=1,2,...,\frac{d(d+1)}{2}$, we obtain the following results with their proofs given in Section \ref{SEC005}.

\begin{theorem}\label{OMG123}
Let $C_{i}^{\alpha}$, $i=1,2,\,\alpha=1,2,...,\frac{d(d+1)}{2}$ be defined in \eqref{Decom}. Then for a sufficiently small $\varepsilon>0$,
\begin{itemize}
\item[$(i)$] if $m\geq d+1$, for $\alpha=1,2,...,d$,
\begin{align*}
C_{1}^{\alpha}-C_{2}^{\alpha}=&
\frac{\det\mathbb{F}^{\ast\alpha}}{\det \mathbb{D}^{\ast}}\frac{1+O(\varepsilon_{0}(d,m;\sigma))}{\mathcal{L}_{d}^{\alpha}\mathcal{M}_{0}\rho_{0}(d,m;\varepsilon)},
\end{align*}
and for $\alpha=d+1,...,\frac{d(d+1)}{2}$,
\begin{align*}
C_{1}^{\alpha}-C_{2}^{\alpha}=&
\frac{\det \mathbb{F}^{\ast\alpha}}{\det \mathbb{D}^{\ast}}\frac{1+O(\varepsilon_{2}(d,m;\sigma))}{\mathcal{L}_{d}^{\alpha}\mathcal{M}_{2}\rho_{2}(d,m;\varepsilon)},
\end{align*}
where the constants $\mathcal{M}_{i}$, $i=0,2$, are defined in \eqref{WEN}, the Lam\'{e} constants $\mathcal{L}_{d}^{\alpha}$ are defined by \eqref{AZ}--\eqref{AZ110}, $\rho_{i}(d,m;\varepsilon)$, $i=0,2$ are defined in \eqref{rate00}, the blow-up factor matrices $\mathbb{D}^{\ast}$ and $\mathbb{F}^{\ast\alpha}$, $\alpha=1,2,...,\frac{d(d+1)}{2}$, are defined by \eqref{WEN002}--\eqref{WEN003}, the rest terms $\varepsilon_{i}(d,m;\sigma)$, $i=0,2,$ are defined in \eqref{ZWZHAO112}--\eqref{ZWZHAO113}.

\item[$(ii)$] if $d-1\leq m<d+1$, for $\alpha=1,2,...,d$,
\begin{align*}
C_{1}^{\alpha}-C_{2}^{\alpha}=&\frac{\det\mathbb{F}_{1}^{\ast\alpha}}{\det \mathbb{F}_{0}^{\ast}}\frac{1+O(\bar{\varepsilon}_{0}(d,m;\sigma))}{\mathcal{L}_{d}^{\alpha}\mathcal{M}_{0}\rho_{0}(d,m;\varepsilon)},
\end{align*}
and for $\alpha=d+1,...,\frac{d(d+1)}{2}$,
\begin{align*}
C_{1}^{\alpha}-C_{2}^{\alpha}=\frac{\det\mathbb{F}_{2}^{\ast\alpha}}{\det \mathbb{F}_{0}^{\ast}}(1+O(\bar{\varepsilon}_{2}(d,m;\sigma))),
\end{align*}
where the blow-up factor matrices $\mathbb{F}_{0}^{\ast}$, $\mathbb{F}^{\ast\alpha}_{1},$ $\alpha=1,2,...,d$, $\mathbb{F}^{\ast\alpha}_{2},$ $\alpha=d+1,...,\frac{d(d+1)}{2}$ are defined by \eqref{GC003}--\eqref{GC006}, the rest terms $\bar{\varepsilon}_{i}(d,m;\sigma)$, $i=0,2,$ are defined in \eqref{GC001}--\eqref{GC002}.

\item[$(iii)$] if $m<d-1$, for $\alpha=1,2,...,\frac{d(d+1)}{2}$,
\begin{align*}
C_{1}^{\alpha}-C_{2}^{\alpha}=\frac{\det\mathbb{F}_{3}^{\ast\alpha}}{\det \mathbb{F}^{\ast}}(1+O(\varepsilon^{\min\{\frac{1}{6},\frac{d-1-m}{12m}\}})),
\end{align*}
where the blow-up factor matrices $\mathbb{F}^{\ast}$ and $\mathbb{F}_{3}^{\ast\alpha},$ $\alpha=1,2,...,\frac{d(d+1)}{2}$, are defined by \eqref{GC009}.
\end{itemize}
\end{theorem}

Combining the aforementioned results, we are ready to prove Theorems \ref{Lthm066}, \ref{ZHthm002} and \ref{ZHthm003}.
\begin{proof}[Proofs of Theorems \ref{Lthm066}, \ref{ZHthm002} and \ref{ZHthm003}.]
To begin with, it follows from Corollary \ref{coro00z} and Lemma \ref{PAK001} that
\begin{align*}
\left|\sum_{\alpha=1}^{\frac{d(d+1)}{2}}C_{2}^\alpha\nabla({v}_{1}^\alpha+{v}_{2}^\alpha)+\nabla{v}_{0}\right|\leq C\delta^{-\frac{d}{2}}e^{-\frac{1}{2C\delta^{1-1/m}}},\;\;\mathrm{in}\;\Omega_{R}.
\end{align*}
This, together with decomposition \eqref{Decom002}, Corollary \ref{thm86} and Theorem \ref{OMG123}, yields that

$(i)$ for $m\geq d+1$, then
\begin{align*}
\nabla u=&\sum^{d}_{\alpha=1}\frac{\det\mathbb{F}^{\ast\alpha}}{\det \mathbb{D}^{\ast}}\frac{1+O(\varepsilon_{0}(d,m;\sigma))}{\mathcal{L}_{d}^{\alpha}\mathcal{M}_{0}\rho_{0}(d,m;\varepsilon)}(\nabla\bar{u}_{1}^{\alpha}+O(\delta^{\frac{m-2}{m}}))\\
&+\sum^{\frac{d(d+1)}{2}}_{\alpha=d+1}\frac{\det \mathbb{F}^{\ast\alpha}}{\det \mathbb{D}^{\ast}}\frac{1+O(\varepsilon_{2}(d,m;\sigma))}{\mathcal{L}_{d}^{\alpha}\mathcal{M}_{2}\rho_{2}(d,m;\varepsilon)}(\nabla\bar{u}_{1}^{\alpha}+O(1))+O(1)\delta^{-\frac{d}{2}}e^{-\frac{1}{2C\delta^{1-1/m}}}\\
=&\sum\limits_{\alpha=1}^{d}\frac{\det\mathbb{F}^{\ast\alpha}}{\det \mathbb{D}^{\ast}}\frac{1+O(\varepsilon_{0}(d,m;\sigma))}{\mathcal{L}_{d}^{\alpha}\mathcal{M}_{0}\rho_{0}(d,m;\varepsilon)}\nabla\bar{u}^{\alpha}_{1}\notag\\
&+\sum\limits_{\alpha=d+1}^{\frac{d(d+1)}{2}}\frac{\det \mathbb{F}^{\ast\alpha}}{\det \mathbb{D}^{\ast}}\frac{1+O(\varepsilon_{2}(d,m;\sigma))}{\mathcal{L}_{d}^{\alpha}\mathcal{M}_{2}\rho_{2}(d,m;\varepsilon)}\nabla\bar{u}^{\alpha}_{1}+O(1)\|\varphi\|_{C^{0}(\partial D)};
\end{align*}

$(ii)$ for $d-1\leq m<d+1$, then
\begin{align*}
\nabla u=&\sum^{d}_{\alpha=1}\frac{\det\mathbb{F}_{1}^{\ast\alpha}}{\det \mathbb{F}_{0}^{\ast}}\frac{1+O(\bar{\varepsilon}_{0}(d,m;\sigma))}{\mathcal{L}_{d}^{\alpha}\mathcal{M}_{0}\rho_{0}(d,m;\varepsilon)}(\nabla\bar{u}_{1}^{\alpha}+O(\delta^{\frac{m-2}{m}}))\\
&+\sum^{\frac{d(d+1)}{2}}_{\alpha=d+1}\frac{\det\mathbb{F}_{2}^{\ast\alpha}}{\det \mathbb{F}_{0}^{\ast}}(1+O(\bar{\varepsilon}_{2}(d,m;\sigma)))(\nabla\bar{u}_{1}^{\alpha}+O(1))+O(1)\delta^{-\frac{d}{2}}e^{-\frac{1}{2C\delta^{1-1/m}}}\\
=&\sum\limits_{\alpha=1}^{d}\frac{\det\mathbb{F}_{1}^{\ast\alpha}}{\det \mathbb{F}_{0}^{\ast}}\frac{1+O(\bar{\varepsilon}_{0}(d,m;\sigma))}{\mathcal{L}_{d}^{\alpha}\mathcal{M}_{0}\rho_{0}(d,m;\varepsilon)}\nabla\bar{u}_{1}^{\alpha}\notag\\
&+\sum\limits_{\alpha=d+1}^{\frac{d(d+1)}{2}}\frac{\det\mathbb{F}_{2}^{\ast\alpha}}{\det \mathbb{F}_{0}^{\ast}}(1+O(\bar{\varepsilon}_{2}(d,m;\sigma)))\nabla\bar{u}^{\alpha}_{1}+O(1)\|\varphi\|_{C^{0}(\partial D)};
\end{align*}

$(iii)$ for $m<d-1$, then
\begin{align*}
\nabla u=&\sum^{d}_{\alpha=1}\frac{\det\mathbb{F}_{3}^{\ast\alpha}}{\det \mathbb{F}^{\ast}}(1+O(\varepsilon^{\min\{\frac{1}{6},\frac{d-1-m}{12m}\}}))(\nabla\bar{u}_{1}^{\alpha}+O(\delta^{\frac{m-2}{m}}))\\
&+\sum^{\frac{d(d+1)}{2}}_{\alpha=d+1}\frac{\det\mathbb{F}_{3}^{\ast\alpha}}{\det \mathbb{F}^{\ast}}(1+O(\varepsilon^{\min\{\frac{1}{6},\frac{d-1-m}{12m}\}}))(\nabla\bar{u}_{1}^{\alpha}+O(1))+O(1)\delta^{-\frac{d}{2}}e^{-\frac{1}{2C\delta^{1-1/m}}}\\
=&\sum\limits_{\alpha=1}^{\frac{d(d+1)}{2}}\frac{\det\mathbb{F}_{3}^{\ast\alpha}}{\det \mathbb{F}^{\ast}}(1+O(\varepsilon^{\min\{\frac{1}{6},\frac{d-1-m}{12m}\}}))\nabla\bar{u}^{\alpha}_{1}+O(1)\|\varphi\|_{C^{0}(\partial D)}.
\end{align*}

Consequently, Theorems \ref{Lthm066}, \ref{ZHthm002} and \ref{ZHthm003} hold.

\end{proof}

\subsection{Proof of Theorem \ref{OMG123}}\label{SEC005}
For $i,j=1,2$ and $\alpha, \beta=1,2,...,\frac{d(d+1)}{2}$, write
\begin{align*}
a_{ij}^{\alpha\beta}:=-\int_{\partial{D}_{j}}\frac{\partial v_{i}^{\alpha}}{\partial \nu_0}\large\Big|_{+}\cdot\psi_{\beta},\quad b_j^{\beta}:=-\int_{\partial D}\frac{\partial v_{j}^{\beta}}{\partial \nu_0}\large\Big|_{+}\cdot\varphi.
\end{align*}
In view of the fourth line of \eqref{La.002}, we obtain that for $\beta=1,2,...,\frac{d(d+1)}{2}$,
\begin{align}\label{AHNTW009}
\begin{cases}
\sum\limits_{\alpha=1}^{\frac{d(d+1)}{2}}(C_{1}^\alpha-C_{2}^{\alpha}) a_{11}^{\alpha\beta}+\sum\limits_{\alpha=1}^{\frac{d(d+1)}{2}}C_{2}^\alpha \sum\limits^{2}_{i=1}a_{i1}^{\alpha\beta}=b_1^\beta,\\
\sum\limits_{\alpha=1}^{\frac{d(d+1)}{2}}(C_{1}^\alpha-C_{2}^{\alpha}) a_{12}^{\alpha\beta}+\sum\limits_{\alpha=1}^{\frac{d(d+1)}{2}}C_{2}^\alpha \sum\limits^{2}_{i=1}a_{i2}^{\alpha\beta}=b_2^\beta.
\end{cases}
\end{align}
By adding the first line of \eqref{AHNTW009} to the second line, we have
\begin{align}\label{zzw002}
\begin{cases}
\sum\limits_{\alpha=1}^{\frac{d(d+1)}{2}}(C_{1}^\alpha-C_{2}^{\alpha}) a_{11}^{\alpha\beta}+\sum\limits_{\alpha=1}^{\frac{d(d+1)}{2}}C_{2}^\alpha \sum\limits^{2}_{i=1}a_{i1}^{\alpha\beta}=b_1^\beta,\\
\sum\limits_{\alpha=1}^{\frac{d(d+1)}{2}}(C_{1}^{\alpha}-C_{2}^{\alpha})\sum\limits^{2}_{j=1}a_{1j}^{\alpha\beta}+\sum\limits_{\alpha=1}^{\frac{d(d+1)}{2}}C_{2}^\alpha \sum\limits^{2}_{i,j=1}a_{ij}^{\alpha\beta}=\sum\limits^{2}_{i=1}b_{i}^{\beta}.
\end{cases}
\end{align}
For the purpose of calculating the difference of $C_{1}^\alpha-C_{2}^{\alpha}$, $\alpha=1,2,...,\frac{d(d+1)}{2}$, we make use of all the systems of equations in \eqref{zzw002}, which is essentially different from the idea adopted in \cite{L2018,LX2020}.

For brevity, write
\begin{align*}
&X^{1}=\big(C_{1}^1-C_{2}^{1},...,C_{1}^\frac{d(d+1)}{2}-C_{2}^{\frac{d(d+1)}{2}}\big)^{T},\quad X^{2}=\big(C_{2}^{1},...,C_{2}^{\frac{d(d+1)}{2}}\big)^{T},\\
&Y^{1}=(b_1^1,...,b_{1}^{\frac{d(d+1)}{2}})^{T},\quad Y^{2}=\bigg(\sum\limits^{2}_{i=1}b_{i}^{1},...,\sum\limits^{2}_{i=1}b_{i}^{\frac{d(d+1)}{2}}\bigg)^{T},
\end{align*}
and
\begin{align*}
&\mathbb{A}=(a_{11}^{\alpha\beta})_{\frac{d(d+1)}{2}\times\frac{d(d+1)}{2}},\quad \mathbb{B}=\bigg(\sum\limits^{2}_{i=1}a_{i1}^{\alpha\beta}\bigg)_{\frac{d(d+1)}{2}\times\frac{d(d+1)}{2}},\\
&\mathbb{C}=\bigg(\sum\limits^{2}_{j=1}a_{1j}^{\alpha\beta}\bigg)_{\frac{d(d+1)}{2}\times\frac{d(d+1)}{2}},\quad \mathbb{D}=\bigg(\sum\limits^{2}_{i,j=1}a_{ij}^{\alpha\beta}\bigg)_{\frac{d(d+1)}{2}\times\frac{d(d+1)}{2}}.
\end{align*}
Then \eqref{zzw002} can be rewritten as
\begin{gather}\label{PLA001}
\begin{pmatrix} \mathbb{A}&\mathbb{B} \\  \mathbb{C}&\mathbb{D}
\end{pmatrix}
\begin{pmatrix}
X^{1}\\
X^{2}
\end{pmatrix}=
\begin{pmatrix}
Y^{1}\\
Y^{2}
\end{pmatrix}.
\end{gather}
The main objective in the following is to solve $X^{1}=(C_{1}^{1}-C_{2}^{1},...,C_{1}^{\frac{d(d+1)}{2}}-C_{2}^{\frac{d(d+1)}{2}})^{T}$ by using  systems of equations \eqref{PLA001}. This is different from that in \cite{MZ2021}, where Miao and Zhao utilized \eqref{PLA001} to calculate $X^{2}=(C_{2}^{1},...,C_{2}^{\frac{d(d+1)}{2}})^{T}$ for the purpose of constructing a family of unified blow-up factors. In addition, in view of the symmetry of $a_{ij}^{\alpha\beta}=a_{ji}^{\beta\alpha}$, we know that $\mathbb{C}=\mathbb{B}^{T}$.

\begin{lemma}\label{KM323}
Assume as above. Then for a sufficiently small $\varepsilon>0$,
\begin{align}\label{FNYZ}
b_{i}^{\beta}=b_{i}^{\ast\beta}+O(\varepsilon^{\frac{1}{2}}),\quad i=1,2,\;\beta=1,2,...,\frac{d(d+1)}{2},
\end{align}
which yields that
\begin{align*}
\sum\limits^{2}_{i=1}b_{i}^{\beta}=\sum\limits^{2}_{i=1}b_{i}^{\ast\beta}+O(\varepsilon^{\frac{1}{2}}).
\end{align*}

\end{lemma}
\begin{proof}
We only give the proof of \eqref{FNYZ} in the case of $i=1$, since the case of $i=2$ can be treated in the exactly same way. For $\beta=1,2,...,\frac{d(d+1)}{2}$,
\begin{align*}
b_1^{\beta}-b_{1}^{\ast\beta}=-\int_{\partial D}\frac{\partial(v_{1}^{\beta}-v_{1}^{\ast\beta})}{\partial \nu_0}\large\Big|_{+}\cdot\varphi,
\end{align*}
where $v_{1}^{\ast\beta}$ and $v_{1}^{\beta}$ solve \eqref{qaz001111} and \eqref{qaz001}, respectively. For $0<t\leq2R$, denote $\Omega_{t}^{\ast}:=\Omega^{\ast}\cap\{|x'|<t\}$. For $\beta=1,2,...,\frac{d(d+1)}{2}$, we introduce a family of auxiliary functions as follows:
\begin{align*}
\bar{u}^{\ast\beta}_{1}=&\psi_{\beta}\bar{v}^{\ast}+\mathcal{F}_{\beta}^{\ast},\;\;\mathcal{F}_{\beta}^{\ast}=\frac{\lambda+\mu}{\mu}f(\bar{v}^{\ast})\psi^{d}_{\beta}\sum^{d-1}_{i=1}\partial_{x_{i}}\delta\,e_{i}+\frac{\lambda+\mu}{\lambda+2\mu}f(\bar{v}^{\ast})\sum^{d-1}_{i=1}\psi^{i}_{\beta}\partial_{x_{i}}\delta\,e_{d},
\end{align*}
where $\bar{v}^{\ast}$ satisfies that $\bar{v}^{\ast}=1$ on $\partial D_{1}^{\ast}\setminus\{0\}$, $\bar{v}^{\ast}=0$ on $\partial D_{2}\cup\partial D$, and
\begin{align*}
\bar{v}^{\ast}(x',x_{d})=\frac{x_{d}-h_{2}(x')}{h_{1}(x')-h_{2}(x')},\;\,\mathrm{in}\;\Omega_{2R}^{\ast},\quad\|\bar{v}^{\ast}\|_{C^{2}(\Omega^{\ast}\setminus\Omega^{\ast}_{R})}\leq C.
\end{align*}
It follows from ({\bf{H1}})--({\bf{H2}}) that for $x\in\Omega_{R}^{\ast}$, $\beta=1,2,...,\frac{d(d+1)}{2}$,
\begin{align}\label{LKT6.003}
|\nabla_{x'}(\bar{u}_{1}^{\beta}-\bar{u}_{1}^{\ast\beta})|\leq\frac{C}{|x'|},\quad|\partial_{x_{d}}(\bar{u}_{1}^{\beta}-\bar{u}_{1}^{\ast\beta})|\leq\frac{C\varepsilon}{|x'|^{m}(\varepsilon+|x'|^{m})}.
\end{align}
Applying Corollary \ref{thm86} to $v_{1}^{\ast\beta}$, we deduce that for $x\in\Omega_{R}^{\ast}$, $\beta=1,2,...,\frac{d(d+1)}{2}$,
\begin{align}\label{LKT6.005}
|\nabla_{x'}v_{1}^{\ast\beta}|\leq\frac{C}{|x'|},\quad|\partial_{x_{d}}v_{1}^{\ast\beta}|\leq\frac{C}{|x'|^{m}},\quad|\nabla(v_{1}^{\ast\beta}-\bar{u}_{1}^{\ast\beta})|\leq C.
\end{align}
For $0<t<R$, denote
\begin{align*}
\mathcal{C}_{t}:=\left\{x\in\mathbb{R}^{d}\Big|\;2\min_{|x'|\leq t}h_{2}(x')\leq x_{d}\leq\varepsilon+2\max_{|x'|\leq t}h_{1}(x'),\;|x'|<t\right\}.
\end{align*}

Note that for $\beta=1,2,...,\frac{d(d+1)}{2}$, $v_{1}^{\beta}-v_{1}^{\ast\beta}$ satisfies
\begin{align*}
\begin{cases}
\mathcal{L}_{\lambda,\mu}(v_{1}^{\beta}-v_{1}^{\ast\beta})=0,&\mathrm{in}\;\,D\setminus(\overline{D_{1}\cup D_{1}^{\ast}\cup D_{2}}),\\
v_{1}^{\beta}-v_{1}^{\ast\beta}=\psi_{\beta}-v_{1}^{\ast\beta},&\mathrm{on}\;\,\partial D_{1}\setminus D_{1}^{\ast},\\
v_{1}^{\beta}-v_{1}^{\ast\beta}=v_{1}^{\beta}-\psi_{\beta},&\mathrm{on}\;\,\partial D_{1}^{\ast}\setminus(D_{1}\cup\{0\}),\\
v_{1}^{\beta}-v_{1}^{\ast\beta}=0,&\mathrm{on}\;\,\partial D_{2}\cup\partial D.
\end{cases}
\end{align*}
To begin with, using the standard boundary and interior estimates of elliptic systems, it follows that for $x\in\partial D_{1}\setminus D_{1}^{\ast}$,
\begin{align}\label{LKT6.007}
|(v_{1}^{\beta}-v_{1}^{\ast\beta})(x',x_{d})|=|v_{1}^{\ast\beta}(x',x_{d}-\varepsilon)-v_{1}^{\ast\beta}(x',x_{d})|\leq C\varepsilon.
\end{align}
From \eqref{Le2.025}, we conclude that for $x\in\partial D_{1}^{\ast}\setminus(D_{1}\cup\mathcal{C}_{\varepsilon^{\gamma}})$, $0<\gamma<\frac{1}{2}$,
\begin{align}\label{LKT6.008}
|(v_{1}^{\beta}-v_{1}^{\ast\beta})(x',x_{d})|=|v_{1}^{\beta}(x',x_{d})-v_{1}^{\beta}(x',x_{d}+\varepsilon)|\leq C\varepsilon^{1-m\gamma}.
\end{align}
Combining \eqref{Le2.025} and (\ref{LKT6.003})--(\ref{LKT6.005}), we deduce that for $x\in\Omega_{R}^{\ast}\cap\{|x'|=\varepsilon^{\gamma}\}$,
\begin{align*}
|\partial_{x_{d}}(v_{1}^{\beta}-v_{1}^{\ast\beta})|\leq&|\partial_{x_{d}}(v_{1}^{\beta}-\bar{u}_{1}^{\beta})|+|\partial_{x_{d}}(\bar{u}_{1}^{\beta}-\bar{u}_{1}^{\ast\beta}|+|\partial_{x_{d}}(v_{1}^{\ast\beta}-\bar{u}_{1}^{\ast\beta})|\notag\\
\leq&C\Big(\frac{1}{\varepsilon^{2m\gamma-1}}+1\Big),
\end{align*}
which, together with $v_{1}^{\beta}-v_{1}^{\ast\beta}=0$ on $\partial D_{2}$, reads that
\begin{align}
|(v_{1}^{\beta}-v_{1}^{\ast\beta})(x',x_{d})|=&|(v_{1}^{\beta}-v_{1}^{\ast\beta})(x',x_{d})-(v_{1}^{\beta}-v_{1}^{\ast\beta})(x',h_{2}(x'))|\notag\\
\leq& C(\varepsilon^{1-m\gamma}+\varepsilon^{m\gamma}).\label{LKT6.009}
\end{align}
Choose $\gamma=\frac{1}{2m}$. Combining with (\ref{LKT6.007})--(\ref{LKT6.009}), we get
$$|v_{1}^{\beta}-v_{1}^{\ast\beta}|\leq C\varepsilon^{\frac{1}{2}},\quad\;\,\mathrm{on}\;\,\partial\big(D\setminus\big(\overline{D_{1}\cup D_{1}^{\ast}\cup D_{2}\cup\mathcal{C}_{\varepsilon^{\frac{1}{2m}}}}\big)\big).$$
Then it follows from the maximum principle for the Lam\'{e} system in \cite{MMN2007} that
\begin{align}\label{AST123}
|v_{1}^{\beta}-v_{1}^{\ast\beta}|\leq C\varepsilon^{\frac{1}{2}},\quad\;\,\mathrm{in}\;\,D\setminus\big(\overline{D_{1}\cup D_{1}^{\ast}\cup D_{2}\cup\mathcal{C}_{\varepsilon^{\frac{1}{2m}}}}\big),
\end{align}
which, in combination with the standard boundary estimates, yields that
\begin{align*}
|\nabla(v_{1}^{\beta}-v_{1}^{\ast\beta})|\leq C\varepsilon^{\frac{1}{2}},\quad\mathrm{on}\;\,\partial D.
\end{align*}
Hence
\begin{align*}
|b_1^{\beta}-b_{1}^{\ast\beta}|\leq\left|\int_{\partial D}\frac{\partial(v_{1}^{\beta}-v_{1}^{\ast\beta})}{\partial \nu_0}\large\Big|_{+}\cdot\varphi\right|\leq C\|\varphi\|_{C^{0}(\partial D)}\varepsilon^{\frac{1}{2}}.
\end{align*}

\end{proof}

For $i,j=1,2$ and $\alpha,\beta=1,2,...,\frac{d(d+1)}{2}$, in view of the definition of $a_{ij}^{\alpha\beta}$, it follows from \eqref{Le2.01222} that
\begin{align*}
a_{ij}^{\alpha\beta}=\int_{\Omega}(\mathbb{C}^0e(v_{i}^{\alpha}), e(v_j^\beta))dx.
\end{align*}
For $i=0,2$, we define
\begin{align}\label{ZZW0a2}
\hat{\varepsilon}_{i}(d,m;\sigma)=&
\begin{cases}
\varepsilon^{\frac{\sigma}{m}},&m>d+i-1+\sigma,\\
\varepsilon^{\frac{\sigma}{m}}|\ln\varepsilon|,&m=d+i-1+\sigma,\\
\varepsilon^{1-\frac{d+i-1}{m}},&d+i-1<m<d+i-1+\sigma,\\
|\ln\varepsilon|^{-1},&m=d+i-1,\\
\varepsilon^{\min\{\frac{1}{6},\frac{d+i-1-m}{12m}\}},&m<d+i-1.
\end{cases}
\end{align}
\begin{lemma}\label{lemmabc}
Assume as above. Then, for a sufficiently small $\varepsilon>0$,

$(i)$ for $\alpha=1,2,...,d$, then
\begin{align}\label{LMC1}
a_{11}^{\alpha\alpha}=&
\begin{cases}
\mathcal{L}_{d}^{\alpha}\mathcal{M}_{0}\rho_{0}(d,m;\varepsilon)(1+O(\hat{\varepsilon}_{0}(d,m;\sigma))),&m\geq d-1,\\
a_{11}^{\ast\alpha\alpha}+O(\hat{\varepsilon}_{0}(d,m;\sigma)),&m<d-1;
\end{cases}
\end{align}

$(ii)$ for $\alpha=d+1,...,\frac{d(d+1)}{2}$, then
\begin{align}\label{LMC}
a_{11}^{\alpha\alpha}=&
\begin{cases}
\mathcal{L}_{d}^{\alpha}\mathcal{M}_{2}\rho_{2}(d,m;\varepsilon)(1+O(\hat{\varepsilon}_{2}(d,m;\sigma))),&m\geq d+1,\\
a_{11}^{\ast\alpha\alpha}+O(\hat{\varepsilon}_{2}(d,m;\sigma)),&m<d+1;
\end{cases}
\end{align}

$(iii)$ if $d=2$, for $\alpha,\beta=1,2,\alpha\neq\beta$, then
\begin{align}\label{LVZQ001}
a_{11}^{12}=a_{11}^{21}=O(1)|\ln\varepsilon|,
\end{align}
and if $d\geq3$, for $\alpha,\beta=1,2,...,d,\,\alpha\neq\beta$, then
\begin{align}\label{ADzc}
a_{11}^{\alpha\beta}=a_{11}^{\beta\alpha}=a_{11}^{\ast\alpha\beta}+O(1)\varepsilon^{\min\{\frac{1}{6},\frac{d-2}{12m}\}},
\end{align}
and if $d\geq2$, for $\alpha=1,2,...,d,\,\beta=d+1,...,\frac{d(d+1)}{2},$ then
\begin{align}\label{LVZQ0011gdw}
a_{11}^{\alpha\beta}=a_{11}^{\beta\alpha}=a_{11}^{\ast\alpha\beta}+O(1)\varepsilon^{\min\{\frac{1}{6},\frac{d-1}{12m}\}},
\end{align}
and if $d\geq3$, for $\alpha,\beta=d+1,...,\frac{d(d+1)}{2},\,\alpha\neq\beta$, then
\begin{align}\label{LVZQ0011}
a_{11}^{\alpha\beta}=a_{11}^{\beta\alpha}=a_{11}^{\ast\alpha\beta}+O(1)\varepsilon^{\min\{\frac{1}{6},\frac{d}{12m}\}};
\end{align}

$(iv)$ for $\alpha,\beta=1,2,...,\frac{d(d+1)}{2}$,
\begin{align}\label{AZQ001}
\sum\limits^{2}_{i=1}a_{i1}^{\alpha\beta}=&\sum\limits^{2}_{i=1}a_{i1}^{\ast\alpha\beta}+O(\varepsilon^{\frac{1}{4}}),\quad\sum\limits^{2}_{j=1}a_{1j}^{\alpha\beta}=\sum\limits^{2}_{j=1}a_{1j}^{*\alpha\beta}+O(\varepsilon^{\frac{1}{4}}),
\end{align}
and
\begin{align}\label{AZQ00111}
\sum\limits^{2}_{i,j=1}a_{ij}^{\alpha\beta}=\sum\limits^{2}_{i,j=1}a_{ij}^{\ast\alpha\beta}+O(\varepsilon^{\frac{1}{4}}),
\end{align}
where $\hat{\varepsilon}_{i}(d,m;\sigma)$, $i=0,2$ are defined by \eqref{ZZW0a2}.
\end{lemma}

\begin{proof}
Since the proofs of \eqref{LVZQ001}--\eqref{AZQ00111} are entirely contained in Lemma 2.3 of \cite{MZ2021}, it is sufficient to give a precise computation for the principal diagonal elements as in \eqref{LMC1}--\eqref{LMC}.

{\bf Step 1. Proof of (\ref{LMC1})}. Pick $\bar{\gamma}=\frac{1}{12m}$. For $\alpha=1,2,...,d$, we first decompose $a_{11}^{\alpha\alpha}$ into three parts as follows:
\begin{align}\label{a1111}
a_{11}^{\alpha\alpha}=&\int_{\Omega\setminus\Omega_{R}}(\mathbb{C}^{0}e(v_{1}^{\alpha}),e(v_{1}^{\alpha}))+\int_{\Omega_{\varepsilon^{\bar{\gamma}}}}(\mathbb{C}^{0}e(v_{1}^{\alpha}),e(v_{1}^{\alpha}))+\int_{\Omega_{R}\setminus\Omega_{\varepsilon^{\bar{\gamma}}}}(\mathbb{C}^{0}e(v_{1}^{\alpha}),e(v_{1}^{\alpha}))\nonumber\\
=&:\mathrm{I}_{1}+\mathrm{II}_{1}+\mathrm{III}_{1}.
\end{align}
For $\varepsilon^{\bar{\gamma}}\leq|z'|\leq R$, we utilize a change of variable
\begin{align*}
\begin{cases}
x'-z'=|z'|^{m}y',\\
x_{d}=|z'|^{m}y_{d},
\end{cases}
\end{align*}
to rescale $\Omega_{|z'|+|z'|^{m}}\setminus\Omega_{|z'|}$ and $\Omega_{|z'|+|z'|^{m}}^{\ast}\setminus\Omega_{|z'|}^{\ast}$ into two nearly unit-size squares (or cylinders) $Q_{1}$ and $Q_{1}^{\ast}$, respectively. Let
\begin{align*}
V_{1}^{\alpha}(y)=v_{1}^{\alpha}(z'+|z'|^{m}y',|z'|^{m}y_{d}),\quad\mathrm{in}\;Q_{1},
\end{align*}
and
\begin{align*}
V_{1}^{\ast\alpha}(y)=v_{1}^{\ast\alpha}(z'+|z'|^{m}y',|z'|^{m}y_{d}),\quad\mathrm{in}\;Q_{1}^{\ast}.
\end{align*}
Due to the fact that $0<V_{1}^{\alpha},V_{1}^{\ast\alpha}<1$, it follows from the standard elliptic estimate that
\begin{align*}
|\nabla^{2}V_{1}^{\alpha}|\leq C,\quad\mathrm{in}\;Q_{1},\quad\mathrm{and}\;|\nabla^{2}V_{1}^{\ast\alpha}|\leq C,\quad\mathrm{in}\;Q_{1}^{\ast}.
\end{align*}
Applying an interpolation with \eqref{AST123}, we obtain
\begin{align*}
|\nabla(V_{1}^{\alpha}-V_{1}^{\ast\alpha})|\leq C\varepsilon^{\frac{1}{2}(1-\frac{1}{2})}\leq C\varepsilon^{\frac{1}{4}}.
\end{align*}
Then rescaling it back to $v_{1}^{\alpha}-v_{1}^{\ast\alpha}$ and in view of $\varepsilon^{\bar{\gamma}}\leq|z'|\leq R$, we have
\begin{align*}
|\nabla(v_{1}^{\alpha}-v_{1}^{\ast\alpha})(x)|\leq C\varepsilon^{\frac{1}{4}}|z'|^{-m}\leq C\varepsilon^{\frac{1}{6}},\quad x\in\Omega^{\ast}_{|z'|+|z'|^{m}}\setminus\Omega_{|z'|}^{\ast}.
\end{align*}
That is, for $\alpha=1,2,...,d$,
\begin{align}\label{con035}
|\nabla(v_{1}^{\alpha}-v_{1}^{\ast\alpha})|\leq C\varepsilon^{\frac{1}{6}},\quad\;\,\mathrm{in}\;\,D\setminus\big(\overline{D_{1}\cup D_{1}^{\ast}\cup D_{2}\cup\mathcal{C}_{\varepsilon^{\bar{\gamma}}}}\big).
\end{align}

For the first term $\mathrm{I}_{1}$ in \eqref{a1111}, since $|\nabla v_{1}^{\alpha}|$ is bounded in $(D_{1}^{\ast}\setminus(D_{1}\cup\Omega_{R}))\cup(D_{1}\setminus D_{1}^{\ast})$ and the volume of $D_{1}^{\ast}\setminus(D_{1}\cup\Omega_{R})$ and $D_{1}\setminus D_{1}^{\ast}$ is of order $O(\varepsilon)$, we deduce from (\ref{con035}) that
\begin{align}\label{KKAA123}
\mathrm{I}_{1}=&\int_{D\setminus(D_{1}\cup D_{1}^{\ast}\cup D_{2}\cup\Omega_{R})}(\mathbb{C}^{0}e(v_{1}^{\alpha}),e(v_{1}^{\alpha}))+O(1)\varepsilon\notag\\
=&\int_{D\setminus(D_{1}\cup D_{1}^{\ast}\cup D_{2}\cup\Omega_{R})}\left((\mathbb{C}^{0}e(v_{1}^{\ast\alpha}),e(v_{1}^{\ast\alpha}))+2(\mathbb{C}^{0}e(v_{1}^{\alpha}-v_{1}^{\ast\alpha}),e(v_{1}^{\ast\alpha}))\right)\notag\\
&+\int_{D\setminus(D_{1}\cup D_{1}^{\ast}\cup D_{2}\cup\Omega_{R})}(\mathbb{C}^{0}e(v_{1}^{\alpha}-v_{1}^{\ast\alpha}),e(v_{1}^{\alpha}-v_{1}^{\ast\alpha}))+O(1)\varepsilon\notag\\
=&\int_{\Omega^{\ast}\setminus\Omega^{\ast}_{R}}(\mathbb{C}^{0}e(v_{1}^{\ast\alpha}),e(v_{1}^{\ast\alpha}))+O(1)\varepsilon^{\frac{1}{6}}.
\end{align}

In view of the definitions of $\bar{u}_{1}^{\alpha}$ and $\mathbb{C}^{0}$, it follows from a direct calculation that for $\alpha=1,2,...,d$,
\begin{align*}
(\mathbb{C}^{0}e(\bar{u}_{1}^{\alpha}),e(\bar{u}_{1}^{\alpha}))=&2(\mathbb{C}^{0}e(\psi_{\alpha}\bar{v}),e(\mathcal{F}_{\alpha}))+(\mathbb{C}^{0}e(\mathcal{F}_{\alpha}),e(\mathcal{F}_{\alpha}))\\
&+(\lambda+\mu)(\partial_{x_{\alpha}}\bar{v})^{2}+\mu\sum\limits^{d}_{i=1}(\partial_{x_{i}}\bar{v})^{2},
\end{align*}
where the correction term $\mathcal{F}_{\alpha}$ is defined by \eqref{QLA001}. This, together with Corollary \ref{thm86}, yields that
\begin{align}\label{con03365}
\mathrm{II}_{1}=&\int_{\Omega_{\varepsilon^{\bar{\gamma}}}}(\mathbb{C}^{0}e(\bar{u}_{1}^{\alpha}),e(\bar{u}_{1}^{\alpha}))+2\int_{\Omega_{\varepsilon^{\bar{\gamma}}}}(\mathbb{C}^{0}e(v_{1}^{\alpha}-\bar{u}_{1}^{\alpha}),e(\bar{u}_{1}^{\alpha}))\notag\\
&+\int_{\Omega_{\varepsilon^{\bar{\gamma}}}}(\mathbb{C}^{0}e(v_{1}^{\alpha}-\bar{u}_{1}^{\alpha}),e(v_{1}^{\alpha}-\bar{u}_{1}^{\alpha}))\notag\\
=&\,\mathcal{L}_{d}^{\alpha}\int_{|x'|<\varepsilon^{\bar{\gamma}}}\frac{dx'}{\varepsilon+h_{1}(x')-h_{2}(x')}+O(1)\varepsilon^{\frac{d-1}{12m}},
\end{align}
where $\mathcal{L}_{d}^{\alpha}$ is defined in \eqref{AZ}--\eqref{AZ110}.

For the last term $\mathrm{III}$ in \eqref{a1111}, we further split it into three parts as follows:
\begin{align*}
\mathrm{III}^{1}_{1}=&\int_{(\Omega_{R}\setminus\Omega_{\varepsilon^{\bar{\gamma}}})\setminus(\Omega^{\ast}_{R}\setminus\Omega^{\ast}_{\varepsilon^{\bar{\gamma}}})}(\mathbb{C}^{0}e(v_{1}^{\alpha}),e(v_{1}^{\alpha})),\\
\mathrm{III}_{1}^{2}=&\int_{\Omega^{\ast}_{R}\setminus\Omega^{\ast}_{\varepsilon^{\bar{\gamma}}}}(\mathbb{C}^{0}e(v_{1}^{\alpha}-v_{1}^{\ast\alpha}),e(v_{1}^{\alpha}-v_{1}^{\ast\alpha}))+2\int_{\Omega^{\ast}_{R}\setminus\Omega^{\ast}_{\varepsilon^{\bar{\gamma}}}}(\mathbb{C}^{0}e(v_{1}^{\alpha}-v_{1}^{\ast\alpha}),e(v_{1}^{\ast\alpha})),\\
\mathrm{III}_{1}^{3}=&\int_{\Omega^{\ast}_{R}\setminus\Omega^{\ast}_{\varepsilon^{\bar{\gamma}}}}(\mathbb{C}^{0}e(v_{1}^{\ast\alpha}),e(v_{1}^{\ast\alpha})).
\end{align*}
Based on the fact that the thickness of $(\Omega_{R}\setminus\Omega_{\varepsilon^{\bar{\gamma}}})\setminus(\Omega^{\ast}_{R}\setminus\Omega^{\ast}_{\varepsilon^{\bar{\gamma}}})$ is $\varepsilon$, it follows from (\ref{Le2.025}) that
\begin{align}\label{con0333355}
|\mathrm{III}_{1}^{1}|\leq\,C\varepsilon\int_{\varepsilon^{\bar{\gamma}}<|x'|<R}\frac{dx'}{|x'|^{2m}}\leq& C
\begin{cases}
\varepsilon,&m<\frac{d-1}{2},\\
\varepsilon|\ln\varepsilon|,&m=\frac{d-1}{2},\\
\varepsilon^{\frac{10m+d-1}{12m}},&m>\frac{d-1}{2}.
\end{cases}
\end{align}
From (\ref{LKT6.005}) and (\ref{con035}), we get
\begin{align}\label{con036666}
|\mathrm{III}_{1}^{2}|\leq\,C\varepsilon^{\frac{1}{6}}.
\end{align}

As for $\mathrm{III}_{1}^{3}$, it follows from (\ref{LKT6.005}) again that
\begin{align*}
\mathrm{III}_{1}^{3}=&\int_{\Omega_{R}^{\ast}\setminus\Omega^{\ast}_{\varepsilon^{\bar{\gamma}}}}(\mathbb{C}^{0}e(\bar{u}_{1}^{\ast\alpha}),e(\bar{u}_{1}^{\ast\alpha}))+2\int_{\Omega_{R}^{\ast}\setminus\Omega^{\ast}_{\varepsilon^{\bar{\gamma}}}}(\mathbb{C}^{0}e(v_{1}^{\ast\alpha}-\bar{u}_{1}^{\ast\alpha}),e(\bar{u}_{1}^{\ast\alpha}))\notag\\
&+\int_{\Omega_{R}^{\ast}\setminus\Omega^{\ast}_{\varepsilon^{\bar{\gamma}}}}(\mathbb{C}^{0}e(v_{1}^{\ast\alpha}-\bar{u}_{1}^{\ast\alpha}),e(v_{1}^{\ast\alpha}-\bar{u}_{1}^{\ast\alpha}))\notag\\
=&\,\mathcal{L}_{d}^{\alpha}\int_{\varepsilon^{\bar{\gamma}}<|x'|<R}\frac{dx'}{h_{1}-h_{2}}+M_{d}^{\ast\alpha}-\int_{\Omega^{\ast}\setminus\Omega^{\ast}_{R}}(\mathbb{C}^{0}e(v_{1}^{\ast\alpha}),e(v_{1}^{\ast\alpha}))+O(1)\varepsilon^{\min\{\frac{1}{6},\frac{d-1}{12 m}\}},
\end{align*}
where
\begin{align}\label{LLLM001}
M_{d}^{\ast\alpha}=&\int_{\Omega^{\ast}\setminus\Omega^{\ast}_{R}}(\mathbb{C}^{0}e(\bar{u}_{1}^{\ast\alpha}),e(\bar{u}_{1}^{\ast\alpha}))+\int_{\Omega_{R}^{\ast}}(\mathbb{C}^{0}e(v_{1}^{\ast\alpha}-\bar{u}_{1}^{\ast\alpha}),e(v_{1}^{\ast\alpha}-\bar{u}_{1}^{\ast\alpha}))\notag\\
&+\int_{\Omega_{R}^{\ast}}\big[2(\mathbb{C}^{0}e(v_{1}^{\ast\alpha}-\bar{u}_{1}^{\ast\alpha}),e(\bar{u}_{1}^{\ast\alpha}))+2(\mathbb{C}^{0}e(\psi_{\alpha}\bar{v}^{\ast}),e(\mathcal{F}_{\alpha}^{\ast}))+(\mathbb{C}^{0}e(\mathcal{F}_{\alpha}^{\ast}),e(\mathcal{F}_{\alpha}^{\ast}))\big]\notag\\
&+
\begin{cases}
\int_{\Omega_{R}^{\ast}}(\lambda+\mu)(\partial_{x_{\alpha}}\bar{v}^{\ast})^{2}+\mu\sum\limits^{d-1}_{i=1}(\partial_{x_{i}}\bar{v}^{\ast})^{2},&\alpha=1,...,d-1,\\
\int_{\Omega_{R}^{\ast}}\mu\sum\limits^{d-1}_{i=1}(\partial_{x_{i}}\bar{v}^{\ast})^{2},&\alpha=d.
\end{cases}
\end{align}
This, in combination with (\ref{KKAA123})--\eqref{con036666}, yields that
\begin{align}\label{ZPH001}
a_{11}^{\alpha\alpha}=&\mathcal{L}_{d}^{\alpha}\left(\int_{\varepsilon^{\bar{\gamma}}<|x'|<R}\frac{dx'}{h_{1}(x')-h_{2}(x')}+\int_{|x'|<\varepsilon^{\bar{\gamma}}}\frac{dx'}{\varepsilon+h_{1}(x')-h_{2}(x')}\right)\nonumber\\
&+M_{d}^{\ast\alpha}+O(1)\varepsilon^{\min\{\frac{1}{6},\frac{d-1}{12m}\}}.
\end{align}
On one hand, for $m\geq d-1$,
\begin{align}\label{LNQ001}
&\int_{|x'|<R}\frac{dx'}{\varepsilon+h_{1}-h_{2}}+\int_{\varepsilon^{\bar{\gamma}}<|x'|<R}\frac{\varepsilon\,dx'}{(h_{1}-h_{2})(\varepsilon+h_{1}-h_{2})}\notag\\
=&\int_{|x'|<R}\frac{1}{\varepsilon+\tau|x'|^{m}}+\int_{|x'|<R}\left(\frac{1}{\varepsilon+h_{1}-h_{2}}-\frac{1}{\varepsilon+\tau|x'|^{m}}\right)+O(1)\varepsilon^{\frac{10m+d-1}{12m}}\notag\\
=&(d-1)\omega_{d-1}\int_{0}^{R}\frac{s^{d-2}}{\varepsilon+\tau s^{m}}+O(1)\int_{0}^{R}\frac{s^{d+\beta-2}}{\varepsilon+\tau s^{m}}\notag\\
=&\mathcal{M}_{0}\rho_{0}(d,m;\varepsilon)(1+O(\hat{\varepsilon}_{0}(d,m;\sigma))).
\end{align}
On the other hand, for $m<d-1$,
\begin{align}\label{LNQ002}
&\int_{|x'|<R}\frac{dx'}{h_{1}-h_{2}}-\int_{|x'|<\varepsilon^{\bar{\gamma}}}\frac{\varepsilon\,dx'}{(h_{1}-h_{2})(\varepsilon+h_{1}-h_{2})}\notag\\
=&\int_{\Omega_{R}^{\ast}}|\partial_{x_{d}}\bar{u}^{\ast\alpha}_{1}|^{2}+O(1)\varepsilon^{\min\{\frac{1}{6},\frac{d-1-m}{12m}\}}.
\end{align}
Consequently, combining with \eqref{ZPH001}--\eqref{LNQ002}, we deduce that (\ref{LMC1}) holds.

{\bf Step 2. Proof of (\ref{LMC})}. Observe that for every $d+1\leq\alpha\leq\frac{d(d+1)}{2}$, there exist two indices $1\leq i_{\alpha}<j_{\alpha}\leq d$ such that
$\psi_{\alpha}=(0,...,0,x_{j_{\alpha}},0,...,0,-x_{i_{\alpha}},0,...,0)$. In particular, if $d+1\leq\alpha\leq2d-1$, we have $i_{\alpha}=\alpha-d,\,j_{\alpha}=d$ and then $\psi_{\alpha}=(0,...,0,x_{d},0,...,0,-x_{\alpha-d})$. Similarly as in \eqref{a1111}, for $\alpha=d+1,...,\frac{d(d+1)}{2}$, we decompose $a_{11}^{\alpha\alpha}$ into three parts as follows:
\begin{align*}
\mathrm{I}_{2}=&\int_{\Omega\setminus\Omega_{R}}(\mathbb{C}^{0}e(v_{1}^{\alpha}),e(v_{1}^{\alpha})),\\
\mathrm{II}_{2}=&\int_{\Omega_{\varepsilon^{\bar{\gamma}}}}(\mathbb{C}^{0}e(v_{1}^{\alpha}),e(v_{1}^{\alpha})),\\
\mathrm{III}_{2}=&\int_{\Omega_{R}\setminus\Omega_{\varepsilon^{\bar{\gamma}}}}(\mathbb{C}^{0}e(v_{1}^{\alpha}),e(v_{1}^{\alpha})),
\end{align*}
where $\bar{\gamma}=\frac{1}{12m}$. Similarly as before, it follows from \eqref{AST123}, the rescale argument, the interpolation inequality and the standard elliptic estimates that for $\alpha=d+1,...,\frac{d(d+1)}{2}$,
\begin{align}\label{OPKT001}
|\nabla(v_{1}^{\alpha}-v_{1}^{\ast\alpha})|\leq C\varepsilon^{\frac{1}{6}},\quad\;\,\mathrm{in}\;\,D\setminus\big(\overline{D_{1}\cup D_{1}^{\ast}\cup D_{2}\cup\mathcal{C}_{\varepsilon^{\bar{\gamma}}}}\big).
\end{align}
By the same argument as in \eqref{KKAA123}, it follows from \eqref{OPKT001} that
\begin{align}\label{KHT01}
\mathrm{I}_{2}=&\int_{\Omega^{\ast}\setminus\Omega^{\ast}_{R}}(\mathbb{C}^{0}e(v_{1}^{\ast\alpha}),e(v_{1}^{\ast\alpha}))+O(1)\varepsilon^{\frac{1}{6}}.
\end{align}
For $\alpha=d+1,...,\frac{d(d+1)}{2}$, a straightforward computation gives that
\begin{align*}
(\mathbb{C}^{0}e(\bar{u}_{1}^{\alpha}),e(\bar{u}_{1}^{\alpha}))=&\mu(x_{i_{\alpha}}^{2}+x_{j_{\alpha}}^{2})\sum^{d}_{k=1}(\partial_{x_{k}}\bar{v})^{2}+(\lambda+\mu)(x_{j_{\alpha}}\partial_{x_{i_{\alpha}}}\bar{v}-x_{i_{\alpha}}\partial_{x_{j_{\alpha}}}\bar{v})^{2}\\
&+2(\mathbb{C}^{0}e(\psi_{\alpha}\bar{v}),e(\mathcal{F}_{\alpha}))+(\mathbb{C}^{0}e(\mathcal{F}_{\alpha}),e(\mathcal{F}_{\alpha})),
\end{align*}
which, in combination with Theorem \ref{thm86}, reads that
\begin{align}\label{coJKL5}
\mathrm{II}_{2}
=&\frac{\mathcal{L}_{d}^{\alpha}}{d-1}\int_{|x'|<\varepsilon^{\bar{\gamma}}}\frac{|x'|^{2}}{\varepsilon+h_{1}(x')-h_{2}(x')}\,dx'+O(1)\varepsilon^{\frac{d-1}{12m}},
\end{align}
where the correction term $\mathcal{F}_{\alpha}$ is defined by \eqref{QLA001} and $\mathcal{L}_{d}^{\alpha}$ is defined in (\ref{AZ})--(\ref{AZ110}).

For the third term $\mathrm{III}_{2}$, similarly as above, we further decompose it as follows:
\begin{align*}
\mathrm{III}_{2}=&\int_{(\Omega_{R}\setminus\Omega_{\varepsilon^{\bar{\gamma}}})\setminus(\Omega^{\ast}_{R}\setminus\Omega^{\ast}_{\varepsilon^{\bar{\gamma}}})}(\mathbb{C}^{0}e(v_{1}^{\alpha}),e(v_{1}^{\alpha}))+\int_{\Omega^{\ast}_{R}\setminus\Omega^{\ast}_{\varepsilon^{\bar{\gamma}}}}(\mathbb{C}^{0}e(v_{1}^{\alpha}-v_{1}^{\ast\alpha}),e(v_{1}^{\alpha}-v_{1}^{\ast\alpha}))\notag\\
&+2\int_{\Omega^{\ast}_{R}\setminus\Omega^{\ast}_{\varepsilon^{\bar{\gamma}}}}(\mathbb{C}^{0}e(v_{1}^{\alpha}-v_{1}^{\ast\alpha}),e(v_{1}^{\ast\alpha}))+\int_{\Omega_{R}^{\ast}\setminus\Omega^{\ast}_{\varepsilon^{\bar{\gamma}}}}(\mathbb{C}^{0}e(\bar{u}_{1}^{\ast\alpha}),e(\bar{u}_{1}^{\ast\alpha}))\notag\\
&+2\int_{\Omega_{R}^{\ast}\setminus\Omega^{\ast}_{\varepsilon^{\bar{\gamma}}}}(\mathbb{C}^{0}e(v_{1}^{\ast\alpha}-\bar{u}_{1}^{\ast\alpha}),e(\bar{u}_{1}^{\ast\alpha}))+\int_{\Omega_{R}^{\ast}\setminus\Omega^{\ast}_{\varepsilon^{\bar{\gamma}}}}(\mathbb{C}^{0}e(v_{1}^{\ast\alpha}-\bar{u}_{1}^{\ast\alpha}),e(v_{1}^{\ast\alpha}-\bar{u}_{1}^{\ast\alpha})).
\end{align*}
Due to the fact that the thickness of $(\Omega_{R}\setminus\Omega_{\varepsilon^{\bar{\gamma}}})\setminus(\Omega^{\ast}_{R}\setminus\Omega^{\ast}_{\varepsilon^{\bar{\gamma}}})$ is $\varepsilon$, it follows from \eqref{LKT6.005} and \eqref{OPKT001} that
\begin{align}\label{QPQIJKL01}
\mathrm{III}_{2}=&\,\frac{\mathcal{L}_{d}^{\alpha}}{d-1}\int_{\varepsilon^{\bar{\gamma}}<|x'|<R}\frac{|x'|^{2}}{h_{1}(x')-h_{2}(x')}dx'-\int_{\Omega^{\ast}\setminus\Omega^{\ast}_{R}}(\mathbb{C}^{0}e(v_{1}^{\ast\alpha}),e(v_{1}^{\ast\alpha}))\notag\\
&+M_{d}^{\ast\alpha}+O(1)\varepsilon^{\min\{\frac{1}{6},\frac{d-1}{12 m}\}},
\end{align}
where
\begin{align*}
M_{d}^{\ast\alpha}=&\int_{\Omega^{\ast}\setminus\Omega^{\ast}_{R}}(\mathbb{C}^{0}e(\bar{u}_{1}^{\ast\alpha}),e(\bar{u}_{1}^{\ast\alpha}))+\int_{\Omega_{R}^{\ast}}(\mathbb{C}^{0}e(v_{1}^{\ast\alpha}-\bar{u}_{1}^{\ast\alpha}),e(v_{1}^{\ast\alpha}-\bar{u}_{1}^{\ast\alpha}))\notag\\
&+\int_{\Omega_{R}^{\ast}}\big[2(\mathbb{C}^{0}e(v_{1}^{\ast\alpha}-\bar{u}_{1}^{\ast\alpha}),e(\bar{u}_{1}^{\ast\alpha}))+2(\mathbb{C}^{0}e(\psi_{\alpha}\bar{v}^{\ast}),e(\mathcal{F}_{\alpha}^{\ast}))+(\mathbb{C}^{0}e(\mathcal{F}_{\alpha}^{\ast}),e(\mathcal{F}_{\alpha}^{\ast}))\big]\\
&+
\begin{cases}
\int_{\Omega_{R}^{\ast}}\Big[\mu(x_{\alpha-d}^{2}+x_{d}^{2})\sum\limits^{d-1}_{k=1}(\partial_{x_{k}}\bar{v}^{\ast})^{2}+\mu(x_{d}\partial_{x_{d}}\bar{v}^{\ast})^{2}+(\lambda+\mu)(x_{d}\partial_{x_{\alpha-d}}\bar{v}^{\ast})^{2}\notag\\
\quad\quad\;-2(\lambda+\mu)x_{\alpha-d}x_{d}\partial_{x_{\alpha-d}}\bar{v}^{\ast}\partial_{x_{d}}\bar{v}^{\ast}\Big],\;\quad\alpha=d+1,...,2d-1,\\
\int_{\Omega_{R}^{\ast}}\Big[\mu(x_{i_{\alpha}}^{2}+x_{j_{\alpha}}^{2})\sum\limits^{d-1}_{k=1}(\partial_{x_{k}}\bar{v}^{\ast})^{2}\notag\\
\quad\quad\;+(\lambda+\mu)(x_{j_{\alpha}}\partial_{x_{i_{\alpha}}}\bar{v}^{\ast}-x_{i_{\alpha}}\partial_{x_{j_{\alpha}}}\bar{v}^{\ast})^{2}\Big],\;\quad\alpha=2d,...,\frac{d(d+1)}{2},\,d\geq3.
\end{cases}
\end{align*}
Then combining \eqref{KHT01}--\eqref{QPQIJKL01}, we obtain that

$(i)$ if $m\geq d+1$, then for $\alpha=d+1,...,\frac{d(d+1)}{2}$, we have
\begin{align}\label{QZH001}
a_{11}^{\alpha\alpha}=&\frac{\mathcal{L}_{d}^{\alpha}}{d-1}\left(\int_{\varepsilon^{\bar{\gamma}}<|x'|<R}\frac{|x'|^{2}}{h_{1}(x')-h_{2}(x')}+\int_{|x'|<\varepsilon^{\bar{\gamma}}}\frac{|x'|^{2}}{\varepsilon+h_{1}(x')-h_{2}(x')}\right)\nonumber\\
&+M_{d}^{\ast\alpha}+O(1)\varepsilon^{\min\{\frac{1}{6},\frac{d-1}{12 m}\}}.
\end{align}
Since
\begin{align*}
&\int_{\varepsilon^{\bar{\gamma}}<|x'|<R}\frac{|x'|^{2}}{h_{1}-h_{2}}+\int_{|x'|<\varepsilon^{\bar{\gamma}}}\frac{|x'|^{2}}{\varepsilon+h_{1}-h_{2}}\\
=&\int_{|x'|<R}\frac{|x'|^{2}}{\varepsilon+\tau|x'|^{m}}+\int_{|x'|<R}\left(\frac{|x'|^{2}}{\varepsilon+h_{1}-h_{2}}-\frac{1}{\varepsilon+\tau|x'|^{m}}\right)\\
&+\int_{\varepsilon^{\bar{\gamma}}<|x'|<R}\frac{\varepsilon|x'|^{2}}{(h_{1}-h_{2})(\varepsilon+h_{1}-h_{2})}\\
=&(d-1)\omega_{d-1}\int_{0}^{R}\frac{s^{d}}{\varepsilon+\tau s^{m}}+O(1)\int_{0}^{R}\frac{s^{d+\beta}}{\varepsilon+\tau s^{m}}\notag\\
=&(d-1)\mathcal{M}_{2}\rho_{2}(d,m;\varepsilon)(1+O(\hat{\varepsilon}_{2}(d,m;\sigma))),
\end{align*}
then
\begin{align*}
a_{11}^{\alpha\alpha}=&\mathcal{L}_{d}^{\alpha}\mathcal{M}_{2}\rho_{2}(d,m;\varepsilon)(1+O(\hat{\varepsilon}_{2}(d,m;\sigma)));
\end{align*}

$(ii)$ if $m<d+1$, on one hand, for $\alpha=d+1,...,2d-1$, we have
\begin{align}\label{QZH002}
a_{11}^{\alpha\alpha}=&\mathcal{L}_{d}^{\alpha}\left(\int_{\varepsilon^{\bar{\gamma}}<|x'|<R}\frac{x_{i_{\alpha}}^{2}}{h_{1}-h_{2}}+\int_{|x'|<\varepsilon^{\bar{\gamma}}}\frac{x_{i_{\alpha}}^{2}}{\varepsilon+h_{1}-h_{2}}\right)\notag\\
&+M_{d}^{\ast\alpha}+O(1)\varepsilon^{\min\{\frac{1}{6},\frac{d-1}{12 m}\}}.
\end{align}
Since
\begin{align*}
&\int_{\varepsilon^{\bar{\gamma}}<|x'|<R}\frac{x_{i_{\alpha}}^{2}}{h_{1}-h_{2}}+\int_{|x'|<\varepsilon^{\bar{\gamma}}}\frac{x_{i_{\alpha}}^{2}}{\varepsilon+h_{1}-h_{2}}\\
=&\int_{|x'|<R}\frac{x_{i_{\alpha}}^{2}}{h_{1}-h_{2}}-\int_{|x'|<\varepsilon^{\bar{\gamma}}}\frac{\varepsilon x_{i_{\alpha}}^{2}}{(h_{1}-h_{2})(\varepsilon+h_{1}-h_{2})}\\
=&\int_{\Omega_{R}^{\ast}}|x_{i_{\alpha}}\partial_{x_{d}}\bar{v}^{\ast}|^{2}+O(1)\varepsilon^{\frac{d+1-m}{12m}},
\end{align*}
then
\begin{align*}
a_{11}^{\alpha\alpha}=a_{11}^{\ast\alpha\alpha}+O(1)\varepsilon^{\min\{\frac{1}{6},\frac{d+1-m}{12 m}\}}.
\end{align*}
On the other hand, for $\alpha=2d,...,\frac{d(d+1)}{2},\,d\geq3$, we have
\begin{align*}
a_{11}^{\alpha\alpha}=&\frac{\mathcal{L}_{d}^{\alpha}}{2}\left(\int_{\varepsilon^{\bar{\gamma}}<|x'|<R}\frac{x_{i_{\alpha}}^{2}+x_{j_{\alpha}}^{2}}{h_{1}-h_{2}}+\int_{|x'|<\varepsilon^{\bar{\gamma}}}\frac{x_{i_{\alpha}}^{2}+x_{j_{\alpha}}^{2}}{\varepsilon+h_{1}-h_{2}}\right)\notag\\
&+M_{d}^{\ast\alpha}+O(1)\varepsilon^{\min\{\frac{1}{6},\frac{d-1}{12m}\}}\\
=&\frac{\mathcal{L}_{d}^{\alpha}}{2}\int_{|x'|<R}\frac{x_{i_{\alpha}}^{2}+x_{j_{\alpha}}^{2}}{h_{1}-h_{2}}+M_{d}^{\ast\alpha}+O(1)\varepsilon^{\min\{\frac{1}{6},\frac{d+1-m}{12m}\}}\\
=&\frac{\mathcal{L}_{d}^{\alpha}}{2}\int_{\Omega_{R}^{\ast}}(x_{i_{\alpha}}^{2}+x_{j_{\alpha}}^{2})|\partial_{x_{d}}\bar{v}^{\ast}|^{2}+M_{d}^{\ast\alpha}+O(1)\varepsilon^{\min\{\frac{1}{6},\frac{d+1-m}{12m}\}}\\
=&a_{11}^{\ast\alpha\alpha}+O(1)\varepsilon^{\min\{\frac{1}{6},\frac{d+1-m}{12 m}\}}.
\end{align*}

Consequently, we complete the proof of (\ref{LMC}).

\end{proof}
Before proving Theorem \ref{OMG123}, we first state a result on the linear space of rigid displacement $\Psi$. Its proof can be seen in Lemma 6.1 of \cite{BLL2017}.
\begin{lemma}\label{GLW}
Let $\xi$ be an element of $\Psi$, defined by \eqref{LAK01} with $d\geq2$. If $\xi$ vanishes at $d$ distinct points $\bar{x}_{1}$, $i=1,2,...,d$, which do not lie on a $(d-1)$-dimensional plane, then $\xi=0$.
\end{lemma}

At this point we are ready to give the proof of Theorem \ref{OMG123}.
\begin{proof}[Proof of Theorem \ref{OMG123}]
We next divide into three cases to prove Theorem \ref{OMG123}.

$(i)$ If $m\geq d+1$, for $\alpha=1,2,...,\frac{d(d+1)}{2}$, we denote
\begin{gather*}
\mathbb{F}^{\alpha}=\begin{pmatrix} b_{1}^{\alpha}&\sum\limits^{2}_{i=1}a_{i1}^{\alpha1}&\cdots&\sum\limits^{2}_{i=1}a_{i1}^{\alpha\frac{d(d+1)}{2}} \\ \sum\limits^{2}_{i=1}b_{i}^{1}&\sum\limits_{i,j=1}^{2}a_{ij}^{11}&\cdots&\sum\limits_{i,j=1}^{2}a_{ij}^{1\frac{d(d+1)}{2}} \\ \vdots&\vdots&\ddots&\vdots\\
\sum\limits^{2}_{i=1}b_{i}^{\frac{d(d+1)}{2}}&\;\;\sum\limits^{2}_{i,j=1}a_{ij}^{\frac{d(d+1)}{2}1}&\;\cdots&\;\;\sum\limits_{i,j=1}^{2}a_{ij}^{\frac{d(d+1)}{2}\frac{d(d+1)}{2}}
\end{pmatrix}.
\end{gather*}
Applying Lemma \ref{KM323}, we obtain
\begin{align*}
\det\mathbb{F}^{\alpha}=\det\mathbb{F}^{\ast\alpha}+O(\varepsilon^{\frac{m-1}{2m}}),\quad \det\mathbb{D}=\det\mathbb{D}^{\ast}+O(\varepsilon^{\frac{m-1}{2m}}).
\end{align*}
We point out that the matrix $\mathbb{D}^{\ast}$ is positive definite and thus $\det\mathbb{D}^{\ast}>0$. In fact, for $\xi=(\xi_{1},\xi_{2},...,\xi_{\frac{d(d+1)}{2}})^{T}\neq0$, it follows from ellipticity condition \eqref{ellip} that
\begin{align}\label{AJZ001}
\xi^{T}\mathbb{D}^{\ast}\xi=&\int_{\Omega^{\ast}}\Bigg(\mathbb{C}^{0}e\bigg(\sum^{\frac{d(d+1)}{2}}_{\alpha=1}\xi_{\alpha}(v_{1}^{\ast\alpha}+v_{2}^{\ast\alpha})\bigg),e\bigg(\sum^{\frac{d(d+1)}{2}}_{\beta=1}\xi_{\beta}(v_{1}^{\ast\beta}+v_{2}^{\ast\beta})\bigg)\Bigg)\notag\\
\geq&\frac{1}{C}\int_{\Omega^{\ast}}\bigg|e\bigg(\sum^{\frac{d(d+1)}{2}}_{\alpha=1}\xi_{\alpha}(v_{1}^{\ast\alpha}+v_{2}^{\ast\alpha})\bigg)\bigg|^{2}>0.
\end{align}
In the last inequality, we utilized the fact that $e\big(\sum^{\frac{d(d+1)}{2}}_{\alpha=1}\xi_{\alpha}(v_{1}^{\ast\alpha}+v_{2}^{\ast\alpha})\big)$ is not identically zero. Otherwise, if $e\big(\sum^{\frac{d(d+1)}{2}}_{\alpha=1}\xi_{\alpha}(v_{1}^{\ast\alpha}+v_{2}^{\ast\alpha})\big)=0$, then there exist some constants $a_{i}$, $i=1,2,...,\frac{d(d+1)}{2}$ such that $\sum_{\alpha=1}^{\frac{d(d+1)}{2}}\xi_{\alpha}(v_{1}^{\ast\alpha}+v_{2}^{\ast\alpha})=\sum^{\frac{d(d+1)}{2}}_{i=1}a_{i}\psi_{i}$, where $\big\{\psi_{i}\big|\,i=1,2,...,\frac{d(d+1)}{2}\big\}$ is a basis of $\Psi$ defined in \eqref{OPP}. Since $\big\{\psi_{i}\big|\,i=1,2,...,\frac{d(d+1)}{2}\big\}$ is linear independent and $\sum_{\alpha=1}^{\frac{d(d+1)}{2}}\xi_{\alpha}(v_{1}^{\ast\alpha}+v_{2}^{\ast\alpha})=0$ on $\partial D$, it follows from Lemma \ref{GLW} that $a_{i}=0$, $i=1,2,...,\frac{d(d+1)}{2}$. Then in light of $\sum_{\alpha=1}^{\frac{d(d+1)}{2}}\xi_{\alpha}(v_{1}^{\ast\alpha}+v_{2}^{\ast\alpha})=\sum_{\alpha=1}^{\frac{d(d+1)}{2}}\xi_{\alpha}\psi_{\alpha}=0$ on $\partial D_{1}^{\ast}$, we deduce from the linear independence of $\big\{\psi_{i}\big|\,i=1,2,...,\frac{d(d+1)}{2}\big\}$ that $\xi=(\xi_{1},\xi_{2},...,\xi_{\frac{d(d+1)}{2}})^{T}=0$, which is a contradiction.

Then, we have
\begin{align}\label{AGU01}
\frac{\det\mathbb{F}^{\alpha}}{\det\mathbb{D}}=&\frac{\det\mathbb{F}^{\ast\alpha}}{\det\mathbb{D}^{\ast}}\frac{1}{1-{\frac{\det\mathbb{D}^{\ast}-\det\mathbb{D}}{\det\mathbb{D}^{\ast}}}}+\frac{\det\mathbb{F}^{\alpha}-\det\mathbb{F}^{\ast\alpha}}{\det\mathbb{D}}\notag\\
=&\frac{\det\mathbb{F}^{\ast\alpha}}{\det\mathbb{D}^{\ast}}(1+O(\varepsilon^{\frac{1}{4}})).
\end{align}
Using \eqref{LMC1}--\eqref{LMC}, we obtain that for $\alpha=1,2,...,d$, if $m\geq d-1$,
\begin{align}\label{AKU01}
\frac{1}{a_{11}^{\alpha\alpha}}=\frac{\rho^{-1}_{0}(d,m;\varepsilon)}{\mathcal{L}_{d}^{\alpha}\mathcal{M}_{0}}\frac{1}{1-\frac{\mathcal{L}_{d}^{\alpha}\mathcal{M}_{0}-\rho^{-1}_{0}(d,m;\varepsilon)a_{11}^{\alpha\alpha}}{\mathcal{L}_{d}^{\alpha}\mathcal{M}_{0}}}=\frac{1+O(\hat{\varepsilon}_{0}(d,m;\sigma))}{\mathcal{L}_{d}^{\alpha}\mathcal{M}_{0}\rho_{0}(d,m;\varepsilon)},
\end{align}
and for $\alpha=d+1,...,\frac{d(d+1)}{2}$, if $m\geq d+1$,
\begin{align}\label{AKU02}
\frac{1}{a_{11}^{\alpha\alpha}}=\frac{\rho^{-1}_{2}(d,m;\varepsilon)}{\mathcal{L}_{d}^{\alpha}\mathcal{M}_{2}}\frac{1}{1-\frac{\mathcal{L}_{d}^{\alpha}\mathcal{M}_{2}-\rho^{-1}_{2}(d,m;\varepsilon)a_{11}^{\alpha\alpha}}{\mathcal{L}_{d}^{\alpha}\mathcal{M}_{2}}}=\frac{1+O(\hat{\varepsilon}_{2}(d,m;\sigma))}{\mathcal{L}_{d}^{\alpha}\mathcal{M}_{2}\rho_{2}(d,m;\varepsilon)}.
\end{align}
In view of \eqref{PLA001} and \eqref{AGU01}--\eqref{AKU02}, it follows from Cramer's rule that
\begin{align*}
C_{1}^{\alpha}-C_{2}^{\alpha}=&\frac{\prod\limits_{i\neq\alpha}^{\frac{d(d+1)}{2}}a_{11}^{ii}\det\mathbb{F}^{\alpha}}{\prod\limits_{i=1}^{\frac{d(d+1)}{2}}a_{11}^{ii}\det \mathbb{D}}(1+O(\rho_{2}^{-1}(d,m;\varepsilon)))\\
=&
\begin{cases}
\frac{\det\mathbb{F}^{\ast\alpha}}{\det \mathbb{D}^{\ast}}\frac{1+O(\varepsilon_{0}(d,m;\sigma))}{\mathcal{L}_{d}^{\alpha}\mathcal{M}_{0}\rho_{0}(d,m;\varepsilon)},&\alpha=1,2,...,d,\\
\frac{\det \mathbb{F}^{\ast\alpha}}{\det \mathbb{D}^{\ast}}\frac{1+O(\varepsilon_{2}(d,m;\sigma))}{\mathcal{L}_{d}^{\alpha}\mathcal{M}_{2}\rho_{2}(d,m;\varepsilon)},&\alpha=d+1,...,\frac{d(d+1)}{2}.
\end{cases}
\end{align*}

$(ii)$ If $d-1\leq m<d+1$, we define
\begin{gather*}\mathbb{A}_{0}=\begin{pmatrix} a_{11}^{d+1\,d+1}&\cdots&a_{11}^{d+1\frac{d(d+1)}{2}} \\\\ \vdots&\ddots&\vdots\\\\a_{11}^{\frac{d(d+1)}{2}d+1}&\cdots&a_{11}^{\frac{d(d+1)}{2}\frac{d(d+1)}{2}}\end{pmatrix}  ,\;\,
\mathbb{B}_{0}=\begin{pmatrix} \sum\limits^{2}_{i=1}a_{i1}^{d+1\,1}&\cdots&\sum\limits^{2}_{i=1}a_{i1}^{d+1\,\frac{d(d+1)}{2}} \\\\ \vdots&\ddots&\vdots\\\\ \sum\limits^{2}_{i=1}a_{i1}^{\frac{d(d+1)}{2}1}&\cdots&\sum\limits^{2}_{i=1}a_{i1}^{\frac{d(d+1)}{2}\frac{d(d+1)}{2}}\end{pmatrix} ,\end{gather*}\begin{gather*}
\mathbb{C}_{0}=\begin{pmatrix} \sum\limits^{2}_{j=1}a_{1j}^{1\,d+1}&\cdots&\sum\limits^{2}_{j=1}a_{1j}^{1\frac{d(d+1)}{2}} \\\\ \vdots&\ddots&\vdots\\\\ \sum\limits^{2}_{j=1}a_{1j}^{\frac{d(d+1)}{2}\,d+1}&\cdots&\sum\limits^{2}_{j=1}a_{1j}^{\frac{d(d+1)}{2}\frac{d(d+1)}{2}}\end{pmatrix}.
\end{gather*}
Denote
\begin{align*}
\mathbb{F}_{0}=\begin{pmatrix} \mathbb{A}_{0}&\mathbb{B}_{0} \\  \mathbb{C}_{0}&\mathbb{D}
\end{pmatrix}.
\end{align*}

On one hand, for $\alpha=1,2,...,d$, we denote
\begin{gather*}
\mathbb{A}^{\alpha}_{1}=\begin{pmatrix}b_{1}^{\alpha}&a_{11}^{\alpha\,d+1}&\cdots&a_{11}^{\alpha\frac{d(d+1)}{2}} \\ b_{1}^{d+1}&a_{11}^{d+1\,d+1}&\cdots&a_{11}^{d+1\frac{d(d+1)}{2}}\\ \vdots&\vdots&\ddots&\vdots\\b_{1}^{\frac{d(d+1)}{2}}&a_{11}^{\frac{d(d+1)}{2}d+1}&\cdots&a_{11}^{\frac{d(d+1)}{2}\frac{d(d+1)}{2}}
\end{pmatrix},
\end{gather*}
\begin{gather*}
\mathbb{B}^{\alpha}_{1}=\begin{pmatrix}\sum\limits_{i=1}^{2}a_{i1}^{\alpha1}&\sum\limits_{i=1}^{2}a_{i1}^{\alpha2}&\cdots&\sum\limits_{i=1}^{2}a_{i1}^{\alpha\,\frac{d(d+1)}{2}} \\ \sum\limits_{i=1}^{2}a_{i1}^{d+1\,1}&\sum\limits_{i=1}^{2}a_{i1}^{d+1\,2}&\cdots&\sum\limits_{i=1}^{2}a_{i1}^{d+1\,\frac{d(d+1)}{2}} \\ \vdots&\vdots&\ddots&\vdots\\\sum\limits_{i=1}^{2}a_{i1}^{\frac{d(d+1)}{2}\,1}&\sum\limits_{i=1}^{2}a_{i1}^{\frac{d(d+1)}{2}\,2}&\cdots&\sum\limits_{i=1}^{2}a_{i1}^{\frac{d(d+1)}{2}\frac{d(d+1)}{2}}
\end{pmatrix},
\end{gather*}
\begin{gather*}
\mathbb{C}^{\alpha}_{1}=\begin{pmatrix}\sum\limits_{i=1}^{2}b_{i}^{1}&\sum\limits_{j=1}^{2}a_{1j}^{1\,d+1}&\cdots&\sum\limits_{j=1}^{2}a_{1j}^{1\,\frac{d(d+1)}{2}} \\\sum\limits_{i=1}^{2}b_{i}^{2}&\sum\limits_{j=1}^{2}a_{1j}^{2\,d+1}&\cdots&\sum\limits_{j=1}^{2}a_{1j}^{2\,\frac{d(d+1)}{2}}\\ \vdots&\vdots&\ddots&\vdots\\\sum\limits_{i=1}^{2}b_{i}^{\frac{d(d+1)}{2}}&\sum\limits_{j=1}^{2}a_{1j}^{\frac{d(d+1)}{2}\,d+1}&\cdots&\sum\limits_{j=1}^{2}a_{1j}^{\frac{d(d+1)}{2}\frac{d(d+1)}{2}}
\end{pmatrix}.
\end{gather*}
Write
\begin{align*}
\mathbb{F}^{\alpha}_{1}=\begin{pmatrix} \mathbb{A}^{\alpha}_{1}&\mathbb{B}^{\alpha}_{1} \\  \mathbb{C}^{\alpha}_{1}&\mathbb{D}
\end{pmatrix},\;\, \alpha=1,2,...,d.
\end{align*}
Then using Lemma \ref{KM323} and \eqref{LMC}, we derive
\begin{align*}
\det\mathbb{F}_{1}^{\alpha}=\det\mathbb{F}_{1}^{\ast\alpha}+O(\varepsilon^{\frac{d+1-m}{12m}}),\quad\det\mathbb{F}_{0}=\det\mathbb{F}^{\ast}_{0}+O(\varepsilon^{\frac{d+1-m}{12m}}),
\end{align*}
which yields that
\begin{align}\label{QGH01}
\frac{\det\mathbb{F}_{1}^{\alpha}
}{\det\mathbb{F}_{0}}=&\frac{\det\mathbb{F}_{1}^{\ast\alpha}}{\det\mathbb{F}_{0}^{\ast}}\frac{1}{1-{\frac{\det\mathbb{F}_{0}^{\ast}-\det\mathbb{F}_{0}}{\det\mathbb{F}_{0}^{\ast}}}}+\frac{\det\mathbb{F}_{1}^{\alpha}-\det\mathbb{F}_{1}^{\ast\alpha}}{\det\mathbb{F}_{0}}\notag\\
=&\frac{\det\mathbb{F}^{\ast\alpha}_{1}}{\det\mathbb{F}_{0}^{\ast}}(1+O(\varepsilon^{\frac{d+1-m}{12m}})).
\end{align}
Similar to the matrix $\mathbb{D}^{\ast}$ in \eqref{AJZ001}, we obtain that $\det\mathbb{F}_{0}^{\ast}\neq0$. Denote
\begin{align*}
\rho_{d}(\varepsilon)=&
\begin{cases}
|\ln\varepsilon|,&d=2,\\
1,&d\geq3.
\end{cases}
\end{align*}
Then combining \eqref{PLA001}, \eqref{AKU01} and \eqref{QGH01}, it follows from Cramer's rule that for $\alpha=1,2,...,d$,
\begin{align*}
C_{1}^{\alpha}-C_{2}^{\alpha}=&\frac{\prod\limits_{i\neq\alpha}^{d}a_{11}^{ii}\det\mathbb{F}_{1}^{\alpha}}{\prod\limits_{i=1}^{d}a_{11}^{ii}\det \mathbb{F}_{0}}(1+O(\rho_{0}^{-1}(d,m;\varepsilon)\rho_{d}(\varepsilon)))\\
=&\frac{\det\mathbb{F}_{1}^{\ast\alpha}}{\det \mathbb{F}_{0}^{\ast}}\frac{1+O(\bar{\varepsilon}_{0}(d,m;\sigma))}{\mathcal{L}_{d}^{\alpha}\mathcal{M}_{0}\rho_{0}(d,m;\varepsilon)}.
\end{align*}

On the other hand, for $\alpha=d+1,...,\frac{d(d+1)}{2}$, we replace the elements of $\alpha$-th column in the matrix $\mathbb{A}_{0}$ by column vector $\Big(b_{1}^{d+1},...,b_{1}^{\frac{d(d+1)}{2}}\Big)^{T}$ and then denote this new matrix by $\mathbb{A}_{2}^{\alpha}$ as follows:
\begin{gather*}
\mathbb{A}_{2}^{\alpha}=
\begin{pmatrix}
a_{11}^{d+1\,d+1}&\cdots&b_{1}^{d+1}&\cdots&a_{11}^{d+1\,\frac{d(d+1)}{2}} \\\\ \vdots&\ddots&\vdots&\ddots&\vdots\\\\a_{11}^{\frac{d(d+1)}{2}\,d+1}&\cdots&b_{1}^{\frac{d(d+1)}{2}}&\cdots&a_{11}^{\frac{d(d+1)}{2}\,\frac{d(d+1)}{2}}
\end{pmatrix}.
\end{gather*}
Similarly as before, by replacing the elements of $\alpha$-th column in the matrix $\mathbb{C}_{0}$ by column vector $\Big(\sum\limits_{i=1}^{2}b_{i}^{1},...,\sum\limits_{i=1}^{2}b_{i}^{\frac{d(d+1)}{2}}\Big)^{T}$, we get the new matrix $\mathbb{C}_{2}^{\alpha}$ as follows:
\begin{gather*}
\mathbb{C}_{2}^{\alpha}=
\begin{pmatrix}
\sum\limits_{j=1}^{2} a_{1j}^{1\,d+1}&\cdots&\sum\limits_{i=1}^{2}b_{i}^{1}&\cdots&\sum\limits_{j=1}^{2} a_{1j}^{1\,\frac{d(d+1)}{2}} \\\\ \vdots&\ddots&\vdots&\ddots&\vdots\\\\\sum\limits_{j=1}^{2} a_{1j}^{\frac{d(d+1)}{2}\,d+1}&\cdots&\sum\limits_{i=1}^{2}b_{i}^{\frac{d(d+1)}{2}}&\cdots&\sum\limits_{j=1}^{2} a_{1j}^{\frac{d(d+1)}{2}\,\frac{d(d+1)}{2}}
\end{pmatrix}.
\end{gather*}
Define
\begin{align*}
\mathbb{F}^{\alpha}_{2}=\begin{pmatrix} \mathbb{A}^{\alpha}_{2}&\mathbb{B}^{\alpha}_{0} \\  \mathbb{C}^{\alpha}_{2}&\mathbb{D}
\end{pmatrix}.
\end{align*}
Using Lemma \ref{KM323} and \eqref{LMC} again, we arrive at
\begin{align}\label{DBU01}
\frac{\det\mathbb{F}_{2}^{\alpha}
}{\det\mathbb{F}_{0}}=&\frac{\det\mathbb{F}_{2}^{\ast\alpha}}{\det\mathbb{F}_{0}^{\ast}}\frac{1}{1-{\frac{\det\mathbb{F}_{0}^{\ast}-\det\mathbb{F}_{0}}{\det\mathbb{F}_{0}^{\ast}}}}+\frac{\det\mathbb{F}_{2}^{\alpha}-\det\mathbb{F}_{2}^{\ast\alpha}}{\det\mathbb{F}_{0}}\notag\\
=&\frac{\det\mathbb{F}^{\ast\alpha}_{2}}{\det\mathbb{F}_{0}^{\ast}}(1+O(\varepsilon^{\frac{d+1-m}{12m}})).
\end{align}
This, together with \eqref{PLA001} and Cramer's rule, leads to that for $\alpha=d+1,...,\frac{d(d+1)}{2}$,
\begin{align*}
C_{1}^{\alpha}-C_{2}^{\alpha}=&\frac{\det\mathbb{F}_{2}^{\alpha}}{\det \mathbb{F}_{0}}(1+O(\rho_{0}^{-1}(d,m;\varepsilon)))=\frac{\det\mathbb{F}_{2}^{\ast\alpha}}{\det \mathbb{F}_{0}^{\ast}}(1+O(\bar{\varepsilon}_{2}(d,m;\sigma))).
\end{align*}

$(iii)$ If $m<d-1$, we replace the elements of $\alpha$-th column in the matrices $\mathbb{A}$ and $\mathbb{C}$ by column vectors $\Big(b_{1}^{1},...,b_{1}^{\frac{d(d+1)}{2}}\Big)^{T}$ and $\Big(\sum\limits_{i=1}^{2}b_{i}^{1},...,\sum\limits_{i=1}^{2}b_{i}^{\frac{d(d+1)}{2}}\Big)^{T}$, respectively, and then denote these two new matrices by $\mathbb{A}_{3}^{\alpha}$ and $\mathbb{C}_{3}^{\alpha}$ as follows:
\begin{gather*}
\mathbb{A}_{3}^{\alpha}=
\begin{pmatrix}
a_{11}^{11}&\cdots&b_{1}^{1}&\cdots&a_{11}^{1\,\frac{d(d+1)}{2}} \\\\ \vdots&\ddots&\vdots&\ddots&\vdots\\\\a_{11}^{\frac{d(d+1)}{2}\,1}&\cdots&b_{1}^{\frac{d(d+1)}{2}}&\cdots&a_{11}^{\frac{d(d+1)}{2}\,\frac{d(d+1)}{2}}
\end{pmatrix},
\end{gather*}
and
\begin{gather*}
\mathbb{C}_{3}^{\alpha}=
\begin{pmatrix}
\sum\limits_{j=1}^{2} a_{1j}^{11}&\cdots&\sum\limits_{i=1}^{2}b_{i}^{1}&\cdots&\sum\limits_{j=1}^{2} a_{1j}^{1\,\frac{d(d+1)}{2}} \\\\ \vdots&\ddots&\vdots&\ddots&\vdots\\\\\sum\limits_{j=1}^{2} a_{1j}^{\frac{d(d+1)}{2}\,1}&\cdots&\sum\limits_{i=1}^{2}b_{i}^{\frac{d(d+1)}{2}}&\cdots&\sum\limits_{j=1}^{2} a_{1j}^{\frac{d(d+1)}{2}\,\frac{d(d+1)}{2}}
\end{pmatrix}.
\end{gather*}
Define
\begin{align*}
\mathbb{F}^{\alpha}_{3}=\begin{pmatrix} \mathbb{A}^{\alpha}_{3}&\mathbb{B} \\  \mathbb{C}^{\alpha}_{3}&\mathbb{D}
\end{pmatrix},\;\,\alpha=1,2,...,\frac{d(d+1)}{2},\quad \mathbb{F}=\begin{pmatrix} \mathbb{A}&\mathbb{B} \\  \mathbb{C}&\mathbb{D}
\end{pmatrix}.
\end{align*}
Then utilizing Lemmas \ref{KM323} and \ref{lemmabc}, we arrive at
\begin{align*}
\det\mathbb{F}_{3}^{\alpha}=\det\mathbb{F}_{3}^{\ast\alpha}+O(\varepsilon^{\min\{\frac{1}{6},\frac{d-1-m}{12m}\}}),\quad\det\mathbb{F}=\det\mathbb{F}^{\ast}+O(\varepsilon^{\min\{\frac{1}{6},\frac{d-1-m}{12m}\}}).
\end{align*}
Using the same argument as in \eqref{AJZ001}, we deduce that $\det\mathbb{F}^{\ast}\neq0$. Then we have
\begin{align*}
\frac{\det\mathbb{F}_{3}^{\alpha}
}{\det\mathbb{F}}=&\frac{\det\mathbb{F}_{3}^{\ast\alpha}}{\det\mathbb{F}^{\ast}}\frac{1}{1-{\frac{\det\mathbb{F}^{\ast}-\det\mathbb{F}}{\det\mathbb{F}^{\ast}}}}+\frac{\det\mathbb{F}_{3}^{\alpha}-\det\mathbb{F}_{3}^{\ast\alpha}}{\det\mathbb{F}}\notag\\
=&\frac{\det\mathbb{F}^{\ast\alpha}_{3}}{\det\mathbb{F}^{\ast}}(1+O(\varepsilon^{\min\{\frac{1}{6},\frac{d-1-m}{12m}\}})),
\end{align*}
which, in combination with \eqref{PLA001} and Cramer's rule, reads that for $\alpha=1,2,...,\frac{d(d+1)}{2}$,
\begin{align*}
C_{1}^{\alpha}-C_{2}^{\alpha}=&\frac{\det\mathbb{F}_{3}^{\alpha}}{\det \mathbb{F}}=\frac{\det\mathbb{F}_{3}^{\ast\alpha}}{\det \mathbb{F}^{\ast}}(1+O(\varepsilon^{\min\{\frac{1}{6},\frac{d-1-m}{12m}\}})).
\end{align*}
The proof is complete.

\end{proof}

\section{An example of two adjacent curvilinear squares with rounded-off angles}\label{SEC006}

\begin{figure}[htb]
\center{\includegraphics[width=0.45\textwidth]{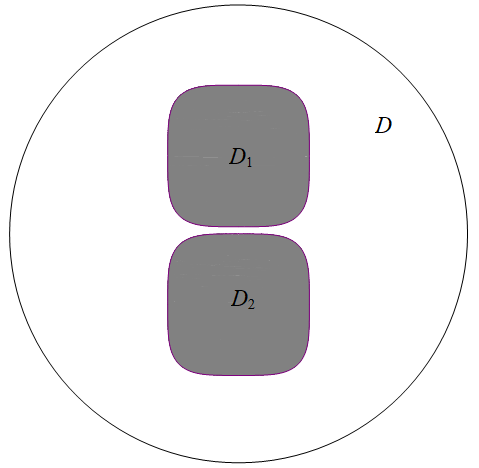}}
\caption{Curvilinear squares with rounded-off angles, $m=4$}
\end{figure}
From the results in \cite{GB2005,HL2018}, we know that the stress concentration will weaken as the convexity index $m$ tends to infinity, which implies that the fiber shapes of curvilinear squares with rounded-off angles (see Figure 1) are superior to circular fibers in the practical design of composite materials. Therefore, it is essential to give a sharp characterization for the gradient in terms of this explicit regular shapes from the view of the shape design of materials. Assume that there exist two positive constants $r_{1}$ and $r_{2}$, independent of $\varepsilon$, such that $\partial D_{1}$ and $\partial D_{2}$ can be expressed as
\begin{align}\label{ZZW986}
|x_{1}|^{m}+|x_{2}-\varepsilon-r_{1}|^{m}=&r_{1}^{m}\quad\mathrm{and}\quad|x_{1}|^{m}+|x_{2}+r_{2}|^{m}=r_{2}^{m},
\end{align}
respectively. Making use of Taylor expansion, we get
\begin{align}\label{ZCZ009}
h_{1}(x_{1})-h_{2}(x_{1})=\tau_{0}|x_{1}|^{m}+O(|x_{1}|^{2m}),\quad\mathrm{in}\;\Omega_{r_{0}},
\end{align}
where $\tau_{0}=\frac{1}{m}\left(\frac{1}{r_{1}^{m-1}}+\frac{1}{r_{2}^{m-1}}\right)$ and $0<r_{0}<\frac{1}{2}\min\{r_{1},r_{2}\}$, $r_{0}$ is a small constant independent of $\varepsilon$. Then, we obtain
\begin{example}\label{coro00389}
Assume as above, condition \eqref{ZZW986} holds. Let $u\in H^{1}(D;\mathbb{R}^{2})\cap C^{1}(\overline{\Omega};\mathbb{R}^{2})$ be the solution of (\ref{La.002}). Then for a sufficiently small $\varepsilon>0$ and $x\in\Omega_{r_{0}}$,

$(i)$ for $m>3$, if $\det \mathbb{D}^{\ast}\neq0$ and $\det\mathbb{F}^{\ast\alpha}\neq0$, $\alpha=1,2,3$,
\begin{align*}
\nabla u=&\sum^{2}_{\alpha=1}\frac{\det\mathbb{F}^{\ast\alpha}}{\det \mathbb{D}^{\ast}}\frac{\varepsilon^{\frac{m-1}{m}}}{\mathcal{L}_{2}^{\alpha}\mathcal{M}_{0}}\frac{1+O(\varepsilon^{\min\{\frac{1}{4},\frac{m-3}{m}\}})}{1+\mathcal{G}^{\ast\alpha}_{m}\varepsilon^{\frac{m-1}{m}}}\nabla\bar{u}_{1}^{\alpha}\notag\\
&+\frac{\det \mathbb{F}^{\ast3}}{\det \mathbb{D}^{\ast}}\frac{\varepsilon^{\frac{m-3}{m}}}{\mathcal{L}_{2}^{3}\mathcal{M}_{2}}\frac{1+O(\varepsilon^{\min\{\frac{1}{4},\frac{12m-35}{12m},\frac{m-3}{m}\}})}{1+\mathcal{G}^{\ast3}_{m}\varepsilon^{\frac{m-3}{m}}}\nabla\bar{u}_{1}^{3}+O(1)\|\varphi\|_{C^{0}(\partial D)};
\end{align*}

$(ii)$ for $m=3$, if $\det \mathbb{D}^{\ast}\neq0$ and $\det\mathbb{F}^{\ast\alpha}\neq0$, $\alpha=1,2,3$,
\begin{align*}
\nabla u=&\sum^{2}_{\alpha=1}\frac{\det\mathbb{F}^{\ast\alpha}}{\det \mathbb{D}^{\ast}}\frac{\varepsilon^{\frac{2}{3}}}{\mathcal{L}_{2}^{\alpha}\mathcal{M}_{0}}\frac{1+O(|\ln\varepsilon|^{-1})}{1+\mathcal{G}^{\ast\alpha}_{3}\varepsilon^{\frac{2}{3}}}\nabla\bar{u}_{1}^{\alpha}\notag\\
&+\frac{\det \mathbb{F}^{\ast3}}{\det \mathbb{D}^{\ast}}\frac{1}{\mathcal{L}_{2}^{3}\mathcal{M}_{2}}\frac{1+O(|\ln\varepsilon|^{-1})}{|\ln\varepsilon|+\mathcal{G}^{\ast3}_{3}}\nabla\bar{u}_{1}^{3}+O(1)\|\varphi\|_{C^{0}(\partial D)};
\end{align*}

$(iii)$ for $2\leq m<3$, if $\det\mathbb{F}_{0}^{\ast}\neq0$, $\det\mathbb{F}_{2}^{\ast3}\neq0$, and $\det\mathbb{F}_{1}^{\ast\alpha}\neq0$, $\alpha=1,2$,
\begin{align*}
\nabla u=&\sum^{2}_{\alpha=1}\frac{\det\mathbb{F}_{1}^{\ast\alpha}}{\det \mathbb{F}_{0}^{\ast}}\frac{\varepsilon^{\frac{m-1}{m}}}{\mathcal{L}_{2}^{\alpha}\mathcal{M}_{0}}\frac{1+O(\varepsilon^{\frac{3-m}{12m}})}{1+\mathcal{G}^{\ast\alpha}_{m}\varepsilon^{\frac{m-1}{m}}}\nabla\bar{u}_{1}^{\alpha}\notag\\
&+\frac{\det\mathbb{F}_{2}^{\ast3}}{\det \mathbb{F}_{0}^{\ast}}(1+O(\varepsilon^{\frac{3-m}{12m}}))\nabla\bar{u}_{1}^{3}+O(1)\|\varphi\|_{C^{0}(\partial D)},
\end{align*}
where the explicit auxiliary functions $\bar{u}_{1}^{\alpha}$, $\alpha=1,2,3$ are defined in \eqref{zzwz002} in the case of $d=2$, the constants $\mathcal{M}_{i}$, $i=0,2$, are defined by \eqref{WEN} with $d=2$, the Lam\'{e} constants $\mathcal{L}_{2}^{\alpha}$, $\alpha=1,2,3$ are defined by \eqref{AZ}, the blow-up factor matrices $\mathbb{D}^{\ast}$, $\mathbb{F}_{0}^{\ast}$, $\mathbb{F}^{\ast\alpha}$, $\alpha=1,2,3$, $\mathbb{F}_{1}^{\ast\alpha}$, $\alpha=1,2,$ and $\mathbb{F}_{2}^{\ast3}$ are defined by \eqref{WEN002}--\eqref{WEN003} and \eqref{GC003}--\eqref{GC006} in the case of $d=2$, the geometry constants $\mathcal{G}^{\ast\alpha}_{m}$, $\alpha=1,2,3$ are defined by \eqref{ZWZWZW0019} below.
\end{example}

\begin{remark}
For two nearly touching curvilinear squares with rounded-off angles, we obtain a more precise asymptotic formula of the stress concentration than that in Theorems \ref{Lthm066} and \ref{ZHthm002} in virtue of the geometry constants $\mathcal{G}^{\ast\alpha}_{m}$, $\alpha=1,2,3$. Moreover, the values of these geometry constants depend only on $m,\lambda,\mu,r_{1},r_{2}$, but not on the distance parameter $\varepsilon$ and the length parameter $r_{0}$ of the narrow region. This is vitally important for numerical calculations and simulations in future investigation. Observe that when $m=2$, these two curvilinear squares become disks and the dependence on the radius is exhibited in the form of the sum of curvature of these two disks. This dependence was previously revealed in \cite{KLY2013,LLY2019,ZH202101} for the perfect conductivity problem, which is regarded as the scalar model of the elasticity problem.

\end{remark}
\begin{remark}
For the purpose of revealing the dependence on the relative principal curvature of the inclusions, we also give an example in terms of the strictly inclusions in dimension three as follows:
\begin{align*}
h_{1}(x')-h_{2}(x')=\kappa_{1}x_{1}^{2}+\kappa_{2}x_{2}^{2},\quad x'\in B_{2R}',
\end{align*}
where $\kappa_{i}$, $i=1,2$ are two positive constants independent of $\varepsilon$. Then applying the proofs of Theorem \ref{ZHthm002} and Example \ref{coro00389} with a slight modification, we obtain that for $x\in\Omega_{R}$,
\begin{align*}
\nabla u=&\sum\limits_{\alpha=1}^{3}\frac{\det\mathbb{F}_{1}^{\ast\alpha}}{\det \mathbb{F}_{0}^{\ast}}\frac{\sqrt{\kappa_{1}\kappa_{2}}\big(1+O(|\ln\varepsilon|^{-1})\big)}{\pi\mathcal{L}_{3}^{\alpha}(|\ln\varepsilon|+\mathcal{G}_{\kappa_{1},\kappa_{2}}^{\ast\alpha})}\nabla\bar{u}_{1}^{\alpha}\notag\\
&+\sum\limits_{\alpha=4}^{6}\frac{\det\mathbb{F}_{2}^{\ast\alpha}}{\det \mathbb{F}_{0}^{\ast}}(1+O(|\ln\varepsilon|^{-1}))\nabla\bar{u}^{\alpha}_{1}+O(1)\|\varphi\|_{C^{0}(\partial D)},
\end{align*}
where
$$\mathcal{G}^{\ast\alpha}_{\kappa_{1},\kappa_{2}}=2\ln R-\frac{2}{\pi}\int^{\frac{\pi}{2}}_{0}\ln(\kappa_{1}^{-1}\cos^{2}\theta+\kappa_{2}^{-1}\sin^{2}\theta)\,d\theta+\frac{\sqrt{\kappa_{1}\kappa_{2}}}{\pi\mathcal{L}_{3}^{\alpha}}\mathcal{M}_{3}^{\ast\alpha},\;\;\alpha=1,2,3,$$ with $\mathcal{M}_{3}^{\ast\alpha}$ defined by \eqref{LLLM001} in the case of $d=3$, the Lam\'{e} constants $\mathcal{L}_{3}^{\alpha}$ are defined by \eqref{AZ110} with $d=3$, the blow-up factor matrices $\mathbb{F}_{0}^{\ast}$, $\mathbb{F}^{\ast\alpha}_{1},$ $\alpha=1,2,3$, $\mathbb{F}^{\ast\alpha}_{2},$ $\alpha=4,5,6$ are defined by \eqref{GC003}--\eqref{GC006} with $d=3$. We here would like to emphasize that for two strictly convex inclusions, Li, Li and Yang \cite{LLY2019} were the first to present the precise calculation of the energy and then establish the asymptotic expansions of the concentrated field for the perfect conductivity problem. For the asymptotic results corresponding to more generalized $m$-convex inclusions with different principal curvatures, see Theorem 3.3 in \cite{MZ2021}.

\end{remark}

\begin{lemma}\label{lem0066666}
Assume as in Example \ref{coro00389}. Then, for a sufficiently small $\varepsilon>0$,

$(i)$ for $\alpha=1,2,$
\begin{align*}
a_{11}^{\alpha\alpha}=&\mathcal{L}_{2}^{\alpha}\mathcal{M}_{0}\rho_{0}(2,m;\varepsilon)+\mathcal{K}_{m}^{\ast\alpha}+O(1)\varepsilon^{\frac{1}{12m}};
\end{align*}

$(ii)$ for $\alpha=3$,
\begin{align*}
a_{11}^{33}=&
\begin{cases}
a_{11}^{\ast33}+O(1)\varepsilon^{\frac{3-m}{12m}},&2\leq m<3,\\
\mathcal{L}_{2}^{3}\mathcal{M}_{2}\rho_{2}(2,m;\varepsilon)+\mathcal{K}_{m}^{\ast3}+O(1)\varepsilon^{\frac{1}{12m}},&m\geq3,
\end{cases}
\end{align*}
where $\mathcal{M}_{i}$, $i=0,2$ are defined by \eqref{WEN} in the case of $d=2$, $\mathcal{K}_{m}^{\ast\alpha}$, $\alpha=1,2,3$ are defined in \eqref{LKM} and \eqref{LKM001} below.

\end{lemma}

\begin{proof}
{\bf Step 1.}
Similar to \eqref{ZPH001}, we obtain that for $\alpha=1,2$,
\begin{align*}
a_{11}^{\alpha\alpha}=&\mathcal{L}_{2}^{\alpha}\left(\int_{\varepsilon^{\frac{1}{12m}}<|x_{1}|<r_{0}}\frac{dx_{1}}{h_{1}(x_{1})-h_{2}(x_{1})}+\int_{|x_{1}|<\varepsilon^{\frac{1}{12m}}}\frac{dx_{1}}{\varepsilon+h_{1}(x_{1})-h_{2}(x_{1})}\right)\notag\\
&+M_{2}^{\ast\alpha}+O(1)\varepsilon^{\frac{1}{12m}},
\end{align*}
where
\begin{align*}
M_{2}^{\ast\alpha}=&\int_{\Omega^{\ast}\setminus\Omega^{\ast}_{r_{0}}}(\mathbb{C}^{0}e(\bar{u}_{1}^{\ast\alpha}),e(\bar{u}_{1}^{\ast\alpha}))+\int_{\Omega_{r_{0}}^{\ast}}(\mathbb{C}^{0}e(v_{1}^{\ast\alpha}-\bar{u}_{1}^{\ast\alpha}),e(v_{1}^{\ast\alpha}-\bar{u}_{1}^{\ast\alpha}))\notag\\
&+\int_{\Omega_{r_{0}}^{\ast}}2\big[(\mathbb{C}^{0}e(v_{1}^{\ast\alpha}-\bar{u}_{1}^{\ast\alpha}),e(\bar{u}_{1}^{\ast\alpha}))+(\mathbb{C}^{0}e(\psi_{\alpha}\bar{v}^{\ast}),e(\mathcal{F}_{\alpha}^{\ast}))\big]\\
&+\int_{\Omega_{r_{0}}^{\ast}}(\mathbb{C}^{0}e(\mathcal{F}_{\alpha}^{\ast}),e(\mathcal{F}_{\alpha}^{\ast}))+
\begin{cases}
\int_{\Omega_{r_{0}}^{\ast}}(\lambda+\mu)(\partial_{x_{1}}\bar{v}^{\ast})^{2}+\mu(\partial_{x_{1}}\bar{v}^{\ast})^{2},&\alpha=1,\\
\int_{\Omega_{r_{0}}^{\ast}}\mu(\partial_{x_{1}}\bar{v}^{\ast})^{2},&\alpha=2.
\end{cases}
\end{align*}

First, it follows from \eqref{ZCZ009} that
\begin{align*}
\int_{\varepsilon^{\frac{1}{12m}}<|x_{1}|<r_{0}}\left(\frac{1}{h_{1}-h_{2}}-\frac{1}{\tau_{0}|x_{1}|^{m}}\right)=\int_{\varepsilon^{\frac{1}{12m}}<|x_{1}|<r_{0}}O(1)=C^{\ast}_{1}+O(1)\varepsilon^{\frac{1}{12m}},
\end{align*}
where $C^{\ast}_{1}$ depends on $m,\tau_{0},r_{0}$, but not on $\varepsilon$. Then
\begin{align*}
\int_{\varepsilon^{\frac{1}{12m}}<|x_{1}|<r_{0}}\frac{dx_{1}}{h_{1}-h_{2}}=&\int_{\varepsilon^{\frac{1}{12m}}<|x_{1}|<r_{0}}\frac{dx_{1}}{\tau_{0}|x_{1}|^{m}}+C^{\ast}_{1}+O(1)\varepsilon^{\frac{1}{12m}}.
\end{align*}
Likewise, we have
\begin{align*}
\int_{|x_{1}|<\varepsilon^{\frac{1}{12m}}}\frac{dx_{1}}{\varepsilon+h_{1}-h_{2}}=&\int_{|x_{1}|<\varepsilon^{\frac{1}{12m}}}\frac{dx_{1}}{\varepsilon+\tau_{0}|x_{1}|^{m}}+O(1)\varepsilon^{\frac{1}{12m}}.
\end{align*}
Then $a_{11}^{\alpha\alpha}$ becomes
\begin{align*}
a_{11}^{\alpha\alpha}=&\mathcal{L}_{2}^{\alpha}\left(\int_{\varepsilon^{\frac{1}{12m}}<|x_{1}|<r_{0}}\frac{dx_{1}}{\tau_{0}|x_{1}|^{m}}+\int_{|x_{1}|<\varepsilon^{\frac{1}{12m}}}\frac{dx_{1}}{\varepsilon+\tau_{0}|x_{1}|^{m}}\right)+\widetilde{M}^{\ast\alpha}_{2}+O(1)\varepsilon^{\frac{1}{12m}},
\end{align*}
where $\widetilde{M}^{\ast\alpha}_{2}=M^{\ast\alpha}_{2}+\mathcal{L}_{2}^{\alpha}C^{\ast}_{1}$. Since
\begin{align*}
&\int_{\varepsilon^{\frac{1}{12m}}<|x_{1}|<r_{0}}\frac{dx_{1}}{\tau_{0}|x_{1}|^{m}}+\int_{|x_{1}|<\varepsilon^{\frac{1}{12m}}}\frac{dx_{1}}{\varepsilon+\tau_{0}|x_{1}|^{m}}\notag\\
=&\int_{-\infty}^{+\infty}\frac{1}{\varepsilon+\tau_{0}|x_{1}|^{m}}-\int_{|x_{1}|>r_{0}}\frac{1}{\tau_{0}|x_{1}|^{m}}+\int_{|x_{1}|>\varepsilon^{\frac{1}{12m}}}\frac{\varepsilon}{\tau_{0}|x_{1}|^{m}(\varepsilon+\tau_{0}|x_{1}|^{m})}\notag\\
=&\frac{2\Gamma[\frac{1}{m}]}{m\tau_{0}^{\frac{1}{m}}}\rho_{0}(2,m;\varepsilon)-\frac{2r_{0}^{1-m}}{\tau_{0}(m-1)}+O(1)\varepsilon^{\frac{10m+1}{12m}},
\end{align*}
then we obtain that for $\alpha=1,2,$
\begin{align}\label{LMT01}
a_{11}^{\alpha\alpha}=&\mathcal{L}_{2}^{\alpha}\mathcal{M}_{0}\rho_{0}(2,m;\varepsilon)+\mathcal{K}_{m}^{\ast\alpha}+O(1)\varepsilon^{\frac{1}{12m}},
\end{align}
where $\mathcal{M}_{0}$ is defined by \eqref{WEN} with $d=2$, and
\begin{align}\label{LKM}
\mathcal{K}_{m}^{\ast\alpha}=\widetilde{M}_{2}^{\ast\alpha}-\frac{2r_{0}^{1-m}\mathcal{L}_{2}^{\alpha}}{\tau_{0}(m-1)}.
\end{align}

{\bf Step 2.} Analogous to \eqref{QZH001}--\eqref{QZH002}, we deduce that for $\alpha=3$,
\begin{align*}
a_{11}^{33}=&\mathcal{L}_{2}^{3}\left(\int_{\varepsilon^{\frac{1}{12m}}<|x_{1}|<r_{0}}\frac{x_{1}^{2}}{h_{1}(x_{1})-h_{2}(x_{1})}+\int_{|x_{1}|<\varepsilon^{\frac{1}{12m}}}\frac{x_{1}^{2}}{\varepsilon+h_{1}(x_{1})-h_{2}(x_{1})}\right)\nonumber\\
&+M_{2}^{\ast3}+O(1)\varepsilon^{\frac{1}{12m}},
\end{align*}
where
\begin{align*}
M_{2}^{\ast3}=&\int_{\Omega^{\ast}\setminus\Omega^{\ast}_{r_{0}}}(\mathbb{C}^{0}e(\bar{u}_{1}^{\ast3}),e(\bar{u}_{1}^{\ast3}))+\int_{\Omega_{r_{0}}^{\ast}}(\mathbb{C}^{0}e(v_{1}^{\ast3}-\bar{u}_{1}^{\ast3}),e(v_{1}^{\ast3}-\bar{u}_{1}^{\ast3}))\notag\\
&+\int_{\Omega_{r_{0}}^{\ast}}\big[2(\mathbb{C}^{0}e(v_{1}^{\ast3}-\bar{u}_{1}^{\ast3}),e(\bar{u}_{1}^{\ast3}))+2(\mathbb{C}^{0}e(\psi_{3}\bar{v}^{\ast}),e(\mathcal{F}_{3}^{\ast}))+(\mathbb{C}^{0}e(\mathcal{F}_{3}^{\ast}),e(\mathcal{F}_{3}^{\ast}))\big]\\
&+
\int_{\Omega_{r_{0}}^{\ast}}\big[\mu(x_{1}^{2}+x_{2}^{2})(\partial_{x_{1}}\bar{v})^{2}+\mu(x_{2}\partial_{x_{2}}\bar{v})^{2}+(\lambda+\mu)(x_{2}\partial_{x_{1}}\bar{v})^{2}\notag\\
&\quad\quad\quad\quad-2(\lambda+\mu)x_{1}x_{2}\partial_{x_{1}}\bar{v}\partial_{x_{2}}\bar{v}\big].
\end{align*}

On one hand, if $2\leq m<3$, then
\begin{align*}
a_{11}^{33}=&\mathcal{L}_{2}^{3}\left(\int_{|x_{1}|<r_{0}}\frac{x_{1}^{2}}{h_{1}-h_{2}}-\int_{|x_{1}|<\varepsilon^{\bar{\gamma}}}\frac{\varepsilon x_{1}^{2}}{(h_{1}-h_{2})(\varepsilon+h_{1}-h_{2})}\right)\nonumber\\
&+M_{2}^{\ast3}+O(1)\varepsilon^{\frac{1}{12m}}\notag\\
=&\mathcal{L}_{2}^{3}\int_{\Omega_{R}^{\ast}}|x_{1}\partial_{x_{2}}\bar{v}^{\ast}|^{2}+M_{2}^{\ast3}+O(1)\varepsilon^{\frac{3-m}{12m}}\\
=&a_{11}^{\ast33}+O(1)\varepsilon^{\frac{3-m}{12m}}.
\end{align*}

On the other hand, if $m\geq3$, using \eqref{ZCZ009}, we obtain
\begin{align*}
\int_{\varepsilon^{\frac{1}{12m}}<|x_{1}|<r_{0}}\left(\frac{x_{1}^{2}}{h_{1}-h_{2}}-\frac{x_{1}^{2}}{\tau_{0}|x_{1}|^{m}}\right)dx_{1}=\int_{\varepsilon^{\frac{1}{12m}}<|x_{1}|<r_{0}}O(x_{1}^{2})\,dx_{1}=C^{\ast}_{2}+O(1)\varepsilon^{\frac{1}{4m}},
\end{align*}
where $C^{\ast}_{2}$ depends on $m,\tau_{0},r_{0}$, but not on $\varepsilon$. Then
\begin{align*}
\int_{\varepsilon^{\frac{1}{12m}}<|x_{1}|<r_{0}}\frac{x_{1}^{2}}{h_{1}-h_{2}}=&\int_{\varepsilon^{\frac{1}{12m}}<|x_{1}|<r_{0}}\frac{x_{1}^{2}}{\tau_{0}|x_{1}|^{m}}+C^{\ast}_{2}+O(1)\varepsilon^{\frac{1}{12m}},
\end{align*}
Similarly, we have
\begin{align*}
\int_{|x_{1}|<\varepsilon^{\frac{1}{12m}}}\frac{x_{1}^{2}}{\varepsilon+h_{1}-h_{2}}=&\int_{|x_{1}|<\varepsilon^{\frac{1}{12m}}}\frac{x_{1}^{2}}{\varepsilon+\tau_{0}|x_{1}|^{m}}+O(1)\varepsilon^{\frac{1}{4m}}.
\end{align*}
Then $a_{11}^{33}$ becomes
\begin{align*}
a_{11}^{33}=&\mathcal{L}_{2}^{3}\left(\int_{\varepsilon^{\frac{1}{12m}}<|x_{1}|<r_{0}}\frac{x_{1}^{2}}{\tau_{0}|x_{1}|^{m}}+\int_{|x_{1}|<\varepsilon^{\frac{1}{12m}}}\frac{x_{1}^{2}}{\varepsilon+\tau_{0}|x_{1}|^{m}}\right)+\widetilde{M}^{\ast3}_{2}+O(1)\varepsilon^{\frac{1}{12m}},
\end{align*}
where $\widetilde{M}^{\ast3}_{2}=M^{\ast3}_{2}+\mathcal{L}_{2}^{3}C^{\ast}_{2}$. A direct calculation yields that for $m=3$,
\begin{align*}
&\int_{\varepsilon^{\frac{1}{36}}<|x_{1}|<r_{0}}\frac{x_{1}^{2}}{\tau_{0}|x_{1}|^{3}}+\int_{|x_{1}|<\varepsilon^{\frac{1}{36}}}\frac{x_{1}^{2}}{\varepsilon+\tau_{0}|x_{1}|^{3}}\notag\\
=&\int_{|x_{1}|<r_{0}}\frac{x_{1}^{2}}{\varepsilon+\tau_{0}|x_{1}|^{3}}+\int_{\varepsilon^{\frac{1}{36}}<|x_{1}|<r_{0}}\frac{\varepsilon x_{1}^{2}}{\tau_{0}|x_{1}|^{3}(\varepsilon+\tau_{0}|x_{1}|^{3})}\notag\\
=&\frac{2}{3\tau_{0}}|\ln\varepsilon|+\frac{2(\ln\tau_{0}+3\ln r_{0})}{3\tau_{0}}+O(1)\varepsilon^{\frac{11}{12}},
\end{align*}
and, for $m>3$,
\begin{align*}
&\int_{\varepsilon^{\frac{1}{12m}}<|x_{1}|<r_{0}}\frac{x_{1}^{2}}{\tau_{0}|x_{1}|^{m}}+\int_{|x_{1}|<\varepsilon^{\frac{1}{12m}}}\frac{x_{1}^{2}}{\varepsilon+\tau_{0}|x_{1}|^{m}}\notag\\
=&\int_{-\infty}^{+\infty}\frac{x_{1}^{2}}{\varepsilon+\tau_{0}|x_{1}|^{m}}-\int_{|x_{1}|>r_{0}}\frac{x_{1}^{2}}{\tau_{0}|x_{1}|^{m}}+\int_{|x_{1}|>\varepsilon^{\frac{1}{12m}}}\frac{\varepsilon x_{1}^{2}}{\tau_{0}|x_{1}|^{m}(\varepsilon+\tau_{0}|x_{1}|^{m})}\notag\\
=&\frac{2\Gamma[\frac{3}{m}]}{m\tau_{0}^{\frac{3}{m}}}\rho_{2}(2,m;\varepsilon)-\frac{2r_{0}^{3-m}}{\tau_{0}(m-3)}+O(1)\varepsilon^{\frac{10m+3}{12m}}.
\end{align*}
Then we obtain that for $m\geq3$,
\begin{align}\label{LMT0111}
a_{11}^{33}=&\mathcal{L}_{2}^{3}\mathcal{M}_{2}\rho_{2}(2,m;\varepsilon)+\mathcal{K}_{m}^{\ast3}+O(1)\varepsilon^{\frac{1}{12m}},
\end{align}
where $\mathcal{M}_{2}$ is defined by \eqref{WEN} in the case of $d=2$, and
\begin{align}\label{LKM001}
\mathcal{K}_{m}^{\ast3}=&
\begin{cases}
\frac{2(\ln\tau_{0}+3\ln r_{0})\mathcal{L}_{2}^{3}}{3\tau_{0}}+\widetilde{M}^{\ast3}_{2},&m=3,\\
\frac{2r_{0}^{3-m}\mathcal{L}_{2}^{3}}{\tau_{0}(3-m)}+\widetilde{M}^{\ast3}_{2},&m>3.
\end{cases}
\end{align}

{\bf Step 3.} We now claim that the constant $\mathcal{K}_{m}^{\ast\alpha}$ captured in \eqref{LKM} and \eqref{LKM001} is independent of the length parameter $r_{0}$ of the narrow region. If not, suppose that there exist $\mathcal{K}_{m}^{\ast\alpha}(r_{1}^{\ast})$ and $\mathcal{K}_{m}^{\ast\alpha}(r_{2}^{\ast})$, $r_{i}^{\ast}>0,\,i=1,2,\,r_{1}^{\ast}\neq r_{2}^{\ast}$, both independent of $\varepsilon$, such that \eqref{LMT01} and \eqref{LMT0111} hold. Then, we arrive at
\begin{align*}
\mathcal{K}_{m}^{\ast\alpha}(r_{1}^{\ast})-\mathcal{K}_{m}^{\ast\alpha}(r_{2}^{\ast})=O(1)\varepsilon^{\frac{1}{12m}},
\end{align*}
which means that $\mathcal{K}_{m}^{\ast\alpha}(r_{1}^{\ast})=\mathcal{K}_{m}^{\ast\alpha}(r_{2}^{\ast})$.

\end{proof}

\begin{proof}[Proof of Example \ref{coro00389}]
Denote
\begin{align}\label{ZWZWZW0019}
\mathcal{G}^{\ast\alpha}_{m}=\frac{\mathcal{K}^{\ast\alpha}_{m}}{\mathcal{L}^{\alpha}_{2}\mathcal{M}_{0}},\;\,\mathrm{for}\;\alpha=1,2,\;m\geq2,\quad \mathcal{G}^{\ast3}_{m}=\frac{\mathcal{K}^{\ast3}_{m}}{\mathcal{L}^{3}_{2}\mathcal{M}_{2}},\;\,\mathrm{for}\;m\geq3.
\end{align}
Similarly as in \eqref{AGU01}--\eqref{AKU02}, it follows from Lemma \ref{lem0066666} that for $\alpha=1,2$, if $m\geq2$,
\begin{align}\label{AZW01}
\frac{1}{a_{11}^{\alpha\alpha}}=&\frac{\rho^{-1}_{0}(2,m;\varepsilon)}{\mathcal{L}_{2}^{\alpha}\mathcal{M}_{0}}\frac{1}{1-\frac{\mathcal{L}_{2}^{\alpha}\mathcal{M}_{0}-\rho^{-1}_{0}(2,m;\varepsilon)a_{11}^{\alpha\alpha}}{\mathcal{L}_{2}^{\alpha}\mathcal{M}_{0}}}\notag\\
=&\frac{\rho^{-1}_{0}(2,m;\varepsilon)}{\mathcal{L}_{2}^{\alpha}\mathcal{M}_{0}}\frac{1}{1+\mathcal{G}^{\ast\alpha}_{m}\rho^{-1}_{0}(2,m;\varepsilon)+O(\varepsilon^{\frac{1}{12m}}\rho_{0}^{-1}(2,m;\varepsilon))}\notag\\
=&\frac{\varepsilon^{\frac{m-1}{m}}}{\mathcal{L}_{2}^{\alpha}\mathcal{M}_{0}}\frac{1+O(\varepsilon^{\frac{12m-11}{12m}})}{1+\mathcal{G}^{\ast\alpha}_{m}\varepsilon^{\frac{m-1}{m}}},
\end{align}
and for $\alpha=3$, if $m\geq 3$,
\begin{align}\label{AZW02}
\frac{1}{a_{11}^{33}}=&\frac{\rho^{-1}_{2}(2,m;\varepsilon)}{\mathcal{L}_{2}^{3}\mathcal{M}_{2}}\frac{1}{1-\frac{\mathcal{L}_{2}^{3}\mathcal{M}_{2}-\rho^{-1}_{2}(2,m;\varepsilon)a_{11}^{33}}{\mathcal{L}_{2}^{3}\mathcal{M}_{2}}}\notag\\
=&\frac{\rho^{-1}_{2}(2,m;\varepsilon)}{\mathcal{L}_{2}^{3}\mathcal{M}_{2}}\frac{1}{1+\mathcal{G}^{\ast3}_{m}\rho^{-1}_{2}(2,m;\varepsilon)+O(\varepsilon^{\frac{1}{12m}}\rho_{2}^{-1}(2,m;\varepsilon))}\notag\\
=&
\begin{cases}
\frac{1}{\mathcal{L}_{2}^{3}\mathcal{M}_{2}}\frac{1+O(\varepsilon^{\frac{1}{12m}}|\ln\varepsilon|)}{|\ln\varepsilon|+\mathcal{G}^{\ast3}_{3}},&m=3,\\
\frac{\varepsilon^{\frac{m-3}{m}}}{\mathcal{L}_{2}^{3}\mathcal{M}_{2}}\frac{1+O(\varepsilon^{\frac{12m-35}{12m}})}{1+\mathcal{G}^{\ast3}_{m}\varepsilon^{\frac{m-3}{m}}},&m>3.
\end{cases}
\end{align}

On one hand, for $m\geq3$, combining \eqref{PLA001}, \eqref{AGU01} and \eqref{AZW01}--\eqref{AZW02}, it follows from Cramer's rule that
\begin{align*}
C_{1}^{\alpha}-C_{2}^{\alpha}=&\frac{\prod\limits_{i\neq\alpha}^{3}a_{11}^{ii}\det\mathbb{F}^{\alpha}}{\prod\limits_{i=1}^{\frac{d(d+1)}{2}}a_{11}^{ii}\det \mathbb{D}}(1+O(\rho_{2}^{-1}(2,m;\varepsilon)))\\
=&
\begin{cases}
\frac{\det\mathbb{F}^{\ast\alpha}}{\det \mathbb{D}^{\ast}}\frac{\varepsilon^{\frac{2}{3}}}{\mathcal{L}_{2}^{\alpha}\mathcal{M}_{0}}\frac{1+O(|\ln\varepsilon|^{-1})}{1+\mathcal{G}^{\ast\alpha}_{3}\varepsilon^{\frac{2}{3}}},&\alpha=1,2,\,m=3,\\
\frac{\det\mathbb{F}^{\ast\alpha}}{\det \mathbb{D}^{\ast}}\frac{\varepsilon^{\frac{m-1}{m}}}{\mathcal{L}_{2}^{\alpha}\mathcal{M}_{0}}\frac{1+O(\varepsilon^{\min\{\frac{1}{4},\frac{m-3}{m}\}})}{1+\mathcal{G}^{\ast\alpha}_{m}\varepsilon^{\frac{m-1}{m}}},&\alpha=1,2,\,m>3,\\
\frac{\det \mathbb{F}^{\ast3}}{\det \mathbb{D}^{\ast}}\frac{1}{\mathcal{L}_{2}^{3}\mathcal{M}_{2}}\frac{1+O(|\ln\varepsilon|^{-1})}{|\ln\varepsilon|+\mathcal{G}^{\ast3}_{3}},&\alpha=3,\,m=3,\\
\frac{\det \mathbb{F}^{\ast3}}{\det \mathbb{D}^{\ast}}\frac{\varepsilon^{\frac{m-3}{m}}}{\mathcal{L}_{2}^{3}\mathcal{M}_{2}}\frac{1+O(\varepsilon^{\min\{\frac{1}{4},\frac{12m-35}{12m},\frac{m-3}{m}\}})}{1+\mathcal{G}^{\ast3}_{m}\varepsilon^{\frac{m-3}{m}}},&\alpha=3,\,m>3.
\end{cases}
\end{align*}

On the other hand, for $2\leq m<3$, utilizing \eqref{PLA001}, \eqref{QGH01} and \eqref{AZW01}, we see from Cramer's rule that for $\alpha=1,2$,
\begin{align*}
C_{1}^{\alpha}-C_{2}^{\alpha}=&\frac{\prod\limits_{i\neq\alpha}^{2}a_{11}^{ii}\det\mathbb{F}_{1}^{\alpha}}{\prod\limits_{i=1}^{2}a_{11}^{ii}\det \mathbb{F}_{0}}(1+O(\rho_{0}^{-1}(2,m;\varepsilon)|\ln\varepsilon|))\\
=&\frac{\det\mathbb{F}_{1}^{\ast\alpha}}{\det \mathbb{F}_{0}^{\ast}}\frac{\varepsilon^{\frac{m-1}{m}}}{\mathcal{L}_{2}^{\alpha}\mathcal{M}_{0}}\frac{1+O(\varepsilon^{\frac{3-m}{12m}})}{1+\mathcal{G}^{\ast\alpha}_{m}\varepsilon^{\frac{m-1}{m}}},
\end{align*}
and for $\alpha=3$, in view of \eqref{DBU01}, we have
\begin{align*}
C_{1}^{3}-C_{2}^{3}=&\frac{\det\mathbb{F}_{2}^{3}}{\det \mathbb{F}_{0}}(1+O(\rho_{0}^{-1}(2,m;\varepsilon)))=\frac{\det\mathbb{F}_{2}^{\ast3}}{\det \mathbb{F}_{0}^{\ast}}(1+O(\varepsilon^{\frac{3-m}{12m}})).
\end{align*}

Therefore, it follows from \eqref{Decom002}, Corollaries \ref{thm86}--\ref{coro00z}, Lemma \ref{PAK001} and Theorem \ref{OMG123} that Example \ref{coro00389} holds.

\end{proof}

\section{Appendix:\,The proof of Proposition \ref{thm8698}}
In the following, we will apply the iterate technique created in \cite{LLBY2014} to prove Proposition \ref{thm8698}. Without loss of generality, we let $\phi=0$ on $\partial D_{2}$ in (\ref{P2.008}) for the convenience of presentation. Note that the solution of (\ref{P2.008}) can be split into
$$v=\sum^{d}_{i=1}v_{i},$$
where $v_{i}=(v_{i}^{1},v_{i}^{2},...,v_{i}^{d})^{T}$, $i=1,2,...,d$, with $v_{i}^{j}=0$ for $j\neq i$, and $v_{i}$ solves the following boundary value problem
\begin{align}\label{P2.010}
\begin{cases}
  \mathcal{L}_{\lambda,\mu}v_{i}:=\nabla\cdot(\mathbb{C}^0e(v_{i}))=0,\quad&
\hbox{in}\  \Omega,  \\
v_{i}=( 0,...,0,\psi^{i}, 0,...,0)^{T},\ &\hbox{on}\ \partial{D}_{1},\\
v_{i}=0,&\hbox{on} \ \partial{D}.
\end{cases}
\end{align}
Then
\begin{align*}
\nabla v=\sum^{d}_{i=1}\nabla v_{i}.
\end{align*}

For $i=1,...,d-1$, $x\in\Omega_{2R}$, we denote
\begin{align*}
\tilde{v}_{i}=&\psi^{i}(x',\varepsilon+h_{1}(x'))\bar{v}e_{i}+\frac{\lambda+\mu}{\lambda+2\mu}f(\bar{v})\psi^{i}(x',\varepsilon+h_{1}(x'))\partial_{x_{i}}\delta\,e_{d},
\end{align*}
and
\begin{align*}
\tilde{v}_{d}=&\psi^{d}(x',\varepsilon+h_{1}(x'))\bar{v}e_{d}+\frac{\lambda+\mu}{\mu}f(\bar{v})\psi^{d}(x',\varepsilon+h_{1}(x'))\sum^{d-1}_{i=1}\partial_{x_{i}}\delta\,e_{i}.
\end{align*}
Then
\begin{align*}
\tilde{v}=\sum^{d}_{i=1}\tilde{v}_{i},
\end{align*}
where $\tilde{v}$ is defined in \eqref{CAN01} with $\phi=0$ on $\partial D_{2}$. A direct calculation gives that for $i=1,2,...,d$, $x\in\Omega_{2R}$,
\begin{align}
|\mathcal{L}_{\lambda,\mu}\tilde{v}_{i}|\leq&\frac{C|\psi^{i}(x',\varepsilon+h_{1}(x'))|}{\delta^{2/m}}+\frac{C|\nabla_{x'}\psi^{i}(x',\varepsilon+h_{1}(x'))|}{\delta}+C\|\psi^{i}\|_{C^{2}(\partial D_{1})}.\label{QQ.103}
\end{align}

For $x\in\Omega_{2R}$, denote
\begin{equation*}
w_i:=v_i-\tilde{v}_i,\quad i=1,2,...,d.
\end{equation*}
Then $w_{i}$ satisfies
\begin{align}\label{LHM001}
\begin{cases}
\mathcal{L}_{\lambda,\mu}w_{i}=-\mathcal{L}_{\lambda,\mu}\tilde{v}_{i},&
\hbox{in}\  \Omega_{2R},  \\
w_{i}=0, \quad&\hbox{on} \ \Gamma^{\pm}_{2R}.
\end{cases}
\end{align}

\noindent{\bf Step 1.}
Let $v_i\in H^1(\Omega;\mathbb{R}^{d})$ be a weak solution of (\ref{P2.010}). Then, for $i=1,2,...,d$,
\begin{align}\label{lem2.2equ}
\int_{\Omega_{2R}}|\nabla w_i|^2dx\leq C\|\psi^{i}\|_{C^{2}(\partial D_{1})}^{2}.
\end{align}

First of all, we extend $\psi\in C^{2}(\partial D_{1};\mathbb{R}^{d})$ to $\psi\in C^{2}(\overline{\Omega};\mathbb{R}^{d})$ so that for $i=1,2,...,d,$ $\|\psi^{i}\|_{C^{2}(\overline{\Omega\setminus\Omega_{R}})}\leq C\|\psi^{i}\|_{C^{2}(\partial D_{1})}$. Construct a cutoff function $\rho\in C^{2}(\overline{\Omega})$ satisfying that $0\leq\rho\leq1$, $|\nabla\rho|\leq C$ in $\overline{\Omega}$, and
\begin{align}\label{Le2.021}
\rho=1\;\,\mathrm{in}\;\,\Omega_{\frac{3}{2}R},\quad\rho=0\;\,\mathrm{in}\;\,\overline{\Omega}\setminus\Omega_{2R}.
\end{align}
For $x\in\Omega$, define
\begin{align*}
\hat{v}_{i}(x)=\left(0,...,0,[\rho(x)\psi^{i}(x',\varepsilon+h_{1}(x'))+(1-\rho(x))\psi^{i}(x)]\bar{v}(x),0,...,0\right)^{T}.
\end{align*}
Then we see
\begin{align*}
\hat{v}_{i}(x)=(0,...,0,\psi^{i}(x',\varepsilon+h_{1}(x'))\bar{v}(x),0,...,0)^{T},\quad\;\,\mathrm{in}\;\,\Omega_{R},
\end{align*}
and in view of (\ref{zh001}) and \eqref{Le2.021},
\begin{align}\label{AHMR001}
\|\hat{v}_{i}\|_{C^{2}(\Omega\setminus\Omega_{R})}\leq C\|\psi^{i}\|_{C^{2}(\partial D_{1})}.
\end{align}
For $i=1,2,...,d$, denote $\hat{w}_{i}:=v_{i}-\hat{v}_{i}$ in $\Omega$. Therefore, $\hat{w}_{i}$ solves
\begin{align}\label{ESGH01}
\begin{cases}
\mathcal{L}_{\lambda,\mu}\hat{w}_{i}=-\mathcal{L}_{\lambda,\mu}\hat{v}_{i},&
\hbox{in}\  \Omega,  \\
\hat{w}_{i}=0, \quad&\hbox{on} \ \partial\Omega.
\end{cases}
\end{align}
%
A straightforward computation yields that for $j=1,...,d-1$, $x\in\Omega_{R}$,
\begin{align}
|\partial_{x_{j}}\hat{v}_{i}|\leq&C|\psi^{i}(x',\varepsilon+h_{1}(x'))|\delta^{-1/m}+\|\psi^{i}\|_{C^{1}(\partial D_{1})},\label{QQ.10101}\\
|\partial_{x_{d}}\hat{v}_{i}|=&|\psi^{i}(x',\varepsilon+h_{1}(x'))|\delta^{-1},\quad\partial_{x_{d}x_{d}}\hat{v}_{i}=0.\label{QQ.101}
\end{align}
Multiplying equation (\ref{ESGH01}) by $\hat{w}_{i}$, it follows from integration by parts that
\begin{align}\label{integrationbypart}
\int_{\Omega}\left(\mathbb{C}^0e(\hat{w}_{i}),e(\hat{w}_{i})\right)dx=\int_{\Omega}\hat{w}_{i}\left(\mathcal{L}_{\lambda,\mu}\hat{v}_{i}\right)dx.
\end{align}

Using the Poincar\'{e} inequality, we have
\begin{align}\label{poincare_inequality}
\|\hat{w}_{i}\|_{L^2(\Omega\setminus\Omega_R)}\leq\,C\|\nabla \hat{w}_{i}\|_{L^2(\Omega\setminus\Omega_R)},
\end{align}
while, in light of the Sobolev trace embedding theorem,
\begin{align}\label{trace}
\int\limits_{\scriptstyle |x'|={R},\atop\scriptstyle
h_{2}(x')<x_{d}<{\varepsilon}+h_1(x')\hfill}|\hat{w}_{i}|\,dx\leq\ C \left(\int_{\Omega\setminus\Omega_{R}}|\nabla \hat{w}_{i}|^2dx\right)^{\frac{1}{2}},
\end{align}
where the constant $C$ is independent of $\varepsilon$. Utilizing (\ref{QQ.10101}), we get
\begin{align}
\int_{\Omega_{R}}|\nabla_{x'}\hat{v}_{i}|^2dx\leq& C\int_{|x'|<R}\delta(x')\left(\frac{|\psi^{i}(x',\varepsilon+h_{1}(x'))|^{2}}{\delta^{2/m}}+\|\psi^{i}\|_{C^{1}(\partial D_{1})}^{2}\right)dx'\notag\\
\leq&C\|\psi^{i}\|_{C^{1}(\partial{D}_{1})}^{2}.\label{nablax'tildeu}
\end{align}

Combining \eqref{ellip}, \eqref{AHMR001}, \eqref{integrationbypart}--\eqref{poincare_inequality} and the first Korn's inequality, we deduce that
\begin{align}
\int_{\Omega}|\nabla\hat{w}_{i}|^2dx\leq &\,2\int_{\Omega}|e(\hat{w}_{i})|^2dx\nonumber\\
\leq&\,C\left|\int_{\Omega_R}\hat{w}_{i}(\mathcal{L}_{\lambda,\mu}\hat{v}_{i})dx\right|+C\left|\int_{\Omega\setminus\Omega_R}\hat{w}_{i}(\mathcal{L}_{\lambda,\mu}\hat{v}_{i})dx\right|\nonumber\\
\leq&\,C\left|\int_{\Omega_R}\hat{w}_{i}(\mathcal{L}_{\lambda,\mu}\hat{v}_{i})dx\right|+C\|\psi^{i}\|_{C^2(\partial{D}_{1})}\int_{\Omega\setminus\Omega_R}|\hat{w}_{i}|dx\nonumber\\
\leq&\,C\left|\int_{\Omega_R}\hat{w}_{i}(\mathcal{L}_{\lambda,\mu}\hat{v}_{i})dx\right|+C\|\psi^{i}\|_{C^2(\partial{D}_{1})}\|\nabla \hat{w}_{i}\|_{L^{2}(\Omega\setminus\Omega_{R})},\nonumber
\end{align}
while, in light of (\ref{QQ.101}) and (\ref{trace})--(\ref{nablax'tildeu}), it follows from integration by parts that
\begin{align}
&\left|\int_{\Omega_R}\hat{w}_{i}(\mathcal{L}_{\lambda,\mu}\hat{v}_{i})dx\right|
\leq\,C\sum_{k+j<2d}\left|\int_{\Omega_{R}}\hat{w}_{i}\partial_{x_{k}x_{j}}\hat{v}_{i}dx\right|\nonumber\\
\leq&C\int_{\Omega_{R}}|\nabla \hat{w}_{i}\|\nabla_{x'}\hat{v}_{i}|dx+\int\limits_{\scriptstyle |x'|={R},\atop\scriptstyle
h_{2}(x')<x_{d}<\varepsilon+h_1(x')\hfill}C|\nabla_{x'}\hat{v}_{i}\|\hat{w}_{i}|dx\nonumber \\
\leq&C\|\nabla \hat{w}_{i}\|_{L^{2}(\Omega_{R})}\|\nabla_{x'}\hat{v}_{i}\|_{L^{2}(\Omega_{R})}+C\|\psi^{i}\|_{C^1(\partial{D}_{1})}\|\nabla \hat{w}_{i}\|_{L^{2}(\Omega\setminus\Omega_{R})}\nonumber\\
\leq&C\|\psi^{i}\|_{C^{1}(\partial{D}_{1})}\|\nabla \hat{w}_{i}\|_{L^{2}(\Omega)}.\nonumber
\end{align}
Thus,
\begin{align*}
\|\nabla \hat{w}_{i}\|_{L^{2}(\Omega)}\leq  C\|\psi^{i}\|_{C^{2}(\partial{D}_{1})}.
\end{align*}
Since
\begin{align*}
w_{i}=\hat{w}_{i}-
\begin{cases}
\frac{\lambda+\mu}{\lambda+2\mu}f(\bar{v})\psi^{i}(x',\varepsilon+h_{1}(x'))\partial_{x_{i}}\delta\,e_{d},&i=1,...,d-1,\\
\frac{\lambda+\mu}{\mu}f(\bar{v})\psi^{d}(x',\varepsilon+h_{1}(x'))\sum^{d-1}_{i=1}\partial_{x_{i}}\delta\,e_{i},&i=d,
\end{cases}\quad \mathrm{in}\;\Omega_{2R},
\end{align*}
we obtain that \eqref{lem2.2equ} holds.

\noindent{\bf Step 2.}
Proof of
\begin{align}\label{step2}
 \int_{\Omega_\delta(z')}|\nabla w_{i}|^2dx\leq& C\delta^{d}(|\psi^{i}(z',\varepsilon+h_{1}(z'))|^{2}\delta^{2-4/m}+|\nabla_{x'}\psi^{i}(z',\varepsilon+h_{1}(z'))|^{2})\notag\\
&+C\delta^{d+2}\|\psi^{i}\|_{C^{2}(\partial D_{1})}^{2},\quad i=1,2,...,d.
\end{align}
For $0<t<s<R$ and $|z'|\leq R$, pick a smooth cutoff function $\eta\in C^{2}(\Omega_{2R})$ such that $\eta(x')=1$ if $|x'-z'|<t$, $\eta(x')=0$ if $|x'-z'|>s$, $0\leq\eta(x')\leq1$ if $t\leq|x'-z'|\leq s$, and $|\nabla_{x'}\eta(x')|\leq\frac{2}{s-t}$. Multiplying equation \eqref{LHM001} by $w_{i}\eta^{2}$ and utilizing integration by parts, we deduce
\begin{align}\label{LFN001}
\int_{\Omega_{s}(z')}\left(\mathbb{C}^0e(w_{i}),e(w_{i}\eta^{2})\right)dx=\int_{\Omega_{s}(z')}w_{i}\eta^{2}\left(\mathcal{L}_{\lambda,\mu}\tilde{v}_{i}\right)dx.
\end{align}
On one hand, by using \eqref{ITERA}, \eqref{ellip} and the first Korn's inequality, we obtain
\begin{align}\label{MAH0199}
\int_{\Omega_{s}(z')}\left(\mathbb{C}^0e(w_{i}),e(w_{i}\eta^{2})\right)dx\geq\frac{1}{C}\int_{\Omega_{s}(z')}|\eta\nabla w_{i}|^{2}-C\int_{\Omega_{s}(z')}|\nabla\eta|^{2}|w_{i}|^{2}.
\end{align}
On the other hand, for the right hand side of \eqref{LFN001}, it follows from H\"{o}lder inequality and Cauchy inequality that
\begin{align*}
\left|\int_{\Omega_{s}(z')}w_{i}\eta^{2}\left(\mathcal{L}_{\lambda,\mu}\tilde{v}_{i}\right)dx\right|\leq \frac{C}{(s-t)^{2}}\int_{\Omega_{s}(z')}|w_{i}|^{2}dx+C(s-t)^{2}\int_{\Omega_{s}(z')}|\mathcal{L}_{\lambda,\mu}\tilde{v}_{i}|^{2}dx.
\end{align*}
This, together with \eqref{LFN001}--\eqref{MAH0199}, yields the iteration formula as follows:
\begin{align*}
\int_{\Omega_{t}(z')}|\nabla w_{i}|^{2}dx\leq\frac{C}{(s-t)^{2}}\int_{\Omega_{s}(z')}|w_{i}|^{2}dx+C(s-t)^{2}\int_{\Omega_{s}(z')}|\mathcal{L}_{\lambda,\mu}\tilde{v}_{i}|^{2}dx.
\end{align*}

For $|z'|\leq R$, $\delta<s\leq\vartheta(\tau,\kappa_{1})\delta^{1/m}$, $\vartheta(\tau,\kappa_{1})=\frac{1}{2^{m+2}\kappa_{1}\max\{1,\tau^{1/m-1}\}}$, it follows from conditions ({\bf{H1}}) and ({\bf{H2}}) that for  $(x',x_{d})\in\Omega_{s}(z')$,
\begin{align}\label{KHW01}
|\delta(x')-\delta(z')|\leq&|h_{1}(x')-h_{1}(z')|+|h_{2}(x')-h_{2}(z')|\notag\\
\leq&(|\nabla_{x'}h_{1}(x'_{\theta_{1}})|+|\nabla_{x'}h_{2}(x'_{\theta_{2}})|)|x'-z'|\notag\\
\leq&\kappa_{1}|x'-z'|(|x'_{\theta_{1}}|^{m-1}+|x'_{\theta_{2}}|^{m-1})\notag\\
\leq&2^{m-1}\kappa_{1}s(s^{m-1}+|z'|^{m-1})\notag\\
\leq&\frac{\delta(z')}{2}.
\end{align}
Then, we have
\begin{align}\label{QWN001}
\frac{1}{2}\delta(z')\leq\delta(x')\leq\frac{3}{2}\delta(z'),\quad\mathrm{in}\;\Omega_{s}(z').
\end{align}
Utilizing (\ref{QQ.103}) and \eqref{QWN001}, we obtain
\begin{align}\label{AQ3.037}
\int_{\Omega_{s}(z')}|\mathcal{L}_{\lambda,\mu}\tilde{v}_{i}|^{2}\leq& C\delta^{-1}s^{d-1}(|\psi^{i}(z',\varepsilon+h_{1}(z'))|^{2}\delta^{2-4/m}+|\nabla_{x'}\psi^{i}(z',\varepsilon+h_{1}(z'))|^{2})\notag\\
&+C\delta^{-1}s^{d+1}\|\psi^{i}\|_{C^{2}(\partial D_{1})}^{2},
\end{align}
while, in the light of $w_{i}=0$ on $\Gamma^{-}_{R}$,
\begin{align}\label{AQ3.038}
\int_{\Omega_{s}(z')}|w_{i}|^{2}\leq C\delta^{2}\int_{\Omega_{s}(z')}|\nabla w_{i}|^{2}.
\end{align}
Define
$$F(t):=\int_{\Omega_{t}(z')}|\nabla w_{i}|^{2}.$$
Then it follows from (\ref{AQ3.037})--(\ref{AQ3.038}) that
\begin{align*}
F(t)\leq& \left(\frac{c\delta}{s-t}\right)^2F(s)+C(s-t)^2s^{d+1}\delta^{-1}\|\psi^{i}\|_{C^{2}(\partial D_{1})}^{2}\notag\\
&+C(s-t)^2\delta^{-1}s^{d-1}(|\psi^{i}(z',\varepsilon+h_{1}(z'))|^{2}\delta^{2-4/m}+|\nabla_{x'}\psi^{i}(z',\varepsilon+h_{1}(z'))|^{2}).
\end{align*}
This, in combination with $s=t_{j+1}$, $t=t_{j}$, $t_{j}=\delta+2cj\delta,\;j=0,1,2,...,\left[\frac{\vartheta(\tau,\kappa_{1})}{4c\delta^{1-1/m}}\right]+1$, yields that
\begin{align*}
F(t_{j})\leq&\frac{1}{4}F(t_{j+1})+C(j+1)^{d+1}\delta^{d+2}\|\psi^{i}\|_{C^{2}(\partial D_{1})}^{2}\notag\\
&+C(j+1)^{d-1}\delta^{d}(|\psi^{i}(z',\varepsilon+h_{1}(z'))|^{2}\delta^{2-4/m}+|\nabla_{x'}(\psi^{i}(z',\varepsilon+h_{1}(z')))|^{2}).
\end{align*}
Consequently, it follows from $\left[\frac{\vartheta(\tau,\kappa_{1})}{4c\delta^{1-1/m}}\right]+1$ iterations and (\ref{lem2.2equ}) that for a sufficiently small $\varepsilon>0$,
\begin{align*}
F(t_{0})\leq C\delta^{d}(|\psi^{i}(z',\varepsilon+h_{1}(z'))|^{2}\delta^{2-4/m}+|\nabla_{x'}\psi^{i}(z',\varepsilon+h_{1}(z'))|^{2}+\delta^{2}\|\psi^{i}\|_{C^{2}(\partial D_{1})}^{2}).
\end{align*}

\noindent{\bf Step 3.}
Proof of
\begin{align}\label{AQ3.052}
|\nabla w_{i}(z)|\leq&C(|\psi^{i}(z',\varepsilon+h_{1}(z'))|\delta^{1-2/m}+|\nabla_{x'}(\psi^{i}(z',\varepsilon+h_{1}(z')))|)\notag\\
&+C\delta\|\psi^{i}\|_{C^{2}(\partial D_{1})},\quad\mathrm{for}\;z\in\Omega_{R}.
\end{align}

To begin with, we make a change of variables in the small narrow region $\Omega_{\delta}(z')$ as follows:
\begin{align*}
\begin{cases}
x'-z'=\delta y',\\
x_{d}=\delta y_{d}.
\end{cases}
\end{align*}
This transformation rescales $\Omega_{\delta}(z')$ into $Q_{1}$, where, for $0<r\leq 1$,
\begin{align*}
Q_{r}=\left\{y\in\mathbb{R}^{d}\,\Big|\,\frac{1}{\delta}h_{2}(\delta y'+z')<y_{d}<\frac{\varepsilon}{\delta}+\frac{1}{\delta}h_{1}(\delta y'+z'),\;|y'|<r\right\}.
\end{align*}
We denote the top and bottom boundaries of $Q_{r}$ by
\begin{align*}
\widehat{\Gamma}^{+}_{r}=&\left\{y\in\mathbb{R}^{d}\,\Big|\,y_{d}=\frac{\varepsilon}{\delta}+\frac{1}{\delta}h_{1}(\delta y'+z'),\;|y'|<r\right\},
\end{align*}
and
\begin{align*}
\widehat{\Gamma}^{-}_{r}=&\left\{y\in\mathbb{R}^{d}\,\Big|\,y_{d}=\frac{1}{\delta}h_{2}(\delta y'+z'),\;|y'|<r\right\},
\end{align*}
respectively. Similar to \eqref{KHW01}, we obtain that for $x\in\Omega_{\delta}(z')$,
\begin{align*}
|\delta(x')-\delta(z')|
\leq&2^{m-1}\kappa_{1}\delta(\delta^{m-1}+|z'|^{m-1})\notag\\
\leq&2^{m+1}\kappa_{1}\max\{1,\tau^{1/m-1}\}\delta^{2-1/m}.
\end{align*}
This implies that
\begin{align*}
\left|\frac{\delta(x')}{\delta(z')}-1\right|\leq2^{m+2}\max\{\tau^{1-1/m},1\}\kappa_{1}R^{m-1},
\end{align*}
which, together with the fact that $R$ is a small positive constant, reads that $Q_{1}$ is of nearly unit size as far as applications of Sobolev embedding theorems and classical $L^{p}$ estimates for elliptic systems are concerned.

Denote
\begin{align*}
W_{i}(y',y_d):=w_{i}(\delta y'+z',\delta y_d),\quad \widetilde{V}_{i}(y',y_d):=\tilde{v}_{i}(\delta y'+z',\delta y_d).
\end{align*}
Then $W_{i}(y)$ satisfies
\begin{align*}
\begin{cases}
\mathcal{L}_{\lambda,\mu}W_{i}=-\mathcal{L}_{\lambda,\mu}\widetilde{V}_{i},&
\hbox{in}\  Q_{1},  \\
W_{i}=0, \quad&\hbox{on} \ \widehat{\Gamma}^{\pm}_{1}.
\end{cases}
\end{align*}
In view of $W_{i}=0$ on $\widehat{\Gamma}^{\pm}_{1}$, it follows from the Poincar\'{e} inequality that
$$\|W_{i}\|_{H^1(Q_1)}\leq C\|\nabla W_{i}\|_{L^2(Q_1)}.$$
By making use of the Sobolev embedding theorem and classical $W^{2, p}$ estimates for elliptic systems, we deduce that for some $p>d$,
\begin{align*}
\|\nabla W_{i}\|_{L^{\infty}(Q_{1/2})}\leq C\|W_{i}\|_{W^{2,p}(Q_{1/2})}\leq C(\|\nabla W_{i}\|_{L^{2}(Q_{1})}+\|\mathcal{L}_{\lambda,\mu}\widetilde{V}_{i}\|_{L^{\infty}(Q_{1})}).
\end{align*}
Then back to $w_{i}$ and $\tilde{v}_{i}$, we get
\begin{align}\label{ADQ601}
\|\nabla w_{i}\|_{L^{\infty}(\Omega_{\delta/2}(z'))}\leq\frac{C}{\delta}\left(\delta^{1-\frac{d}{2}}\|\nabla w_{i}\|_{L^{2}(\Omega_{\delta}(z'))}+\delta^{2}\|\mathcal{L}_{\lambda,\mu}\tilde{v}_{i}\|_{L^{\infty}(\Omega_{\delta}(z'))}\right).
\end{align}
From \eqref{QQ.103} and \eqref{step2}, it follows that for $z'\in B'_{R}$,
\begin{align*}
\delta^{-\frac{d}{2}}\|\nabla w_{i}\|_{L^{2}(\Omega_{\delta}(z'))}\leq& C(|\psi^{i}(z',\varepsilon+h_{1}(z'))|\delta^{1-2/m}+|\nabla_{x'}\psi^{i}(z',\varepsilon+h_{1}(z'))|)\\
&+C\delta\|\psi^{i}\|_{C^{2}(\partial D_{1})},\\
\delta\|\mathcal{L}_{\lambda,\mu}\tilde{v}_{i}\|_{L^{\infty}(\Omega_{\delta}(z'))}\leq&C(|\psi^{i}(z',\varepsilon+h_{1}(z'))|\delta^{1-2/m}+|\nabla_{x'}\psi^{i}(z',\varepsilon+h_{1}(z'))|)\\
&+C\delta\|\psi^{i}\|_{C^{2}(\partial D_{1})},
\end{align*}
which, together with \eqref{ADQ601}, leads to that \eqref{AQ3.052} holds. That is, the proof of Proposition \ref{thm8698} is accomplished.



\end{document}